\documentclass[10pt,journal,twocolumn,twoside]{autart}

\usepackage{amsmath} 
\usepackage{amssymb}  
\usepackage{verbatim}  
\usepackage{enumitem}
\usepackage{url}

\newtheorem{objective}{\bf Objective}
\newtheorem{assumption}{\bf Assumption}

\newtheorem{theorem}{\bf Theorem}
\newtheorem{proposition}{\bf Proposition}
\newtheorem{lemma}{\bf Lemma}
\newtheorem{corollary}{\bf Corollary}
\newtheorem{remark}{\bf Remark}

\usepackage{setspace}

\usepackage{tikz}
 \usetikzlibrary{arrows,patterns,mindmap,backgrounds,shadows}
\usetikzlibrary{backgrounds}
\usetikzlibrary{shapes,arrows,chains}
\usepackage{tcolorbox}

\newcommand{\gu}{g_{\mathrm{u}}}
\newcommand{\gy}{g_{\mathrm{y}}}
\newcommand{\gyx}{g_{\mathrm{y,x}}}
\newcommand{\gyu}{g_{\mathrm{y,u}}}
\newcommand{\Crho}{C_\rho}
\newcommand{\Cell}{C_\ell}
\newcommand{\JN}{\mathcal{J}_N}
\newcommand{\JNkappa}{\mathcal{J}_{N,\kappa}}
\newcommand{\Vf}{V_{\mathrm{f}}}
\newcommand{\Vfkappa}{V_{\mathrm{f},\kappa}}
\newcommand{\fw}{f}

\newcommand{\nx}{n_{\mathrm{x}}}
\renewcommand{\nu}{n_{\mathrm{u}}}
\newcommand{\nw}{n_{\mathrm{w}}}
\newcommand{\ny}{n_{\mathrm{y}}}
\newcommand{\ntheta}{n_{\theta}}
\newcommand{\Vtheta}{V_{\theta}}
\newcommand{\Time}{K}
 
\newbool{arxiv}
\setbool{arxiv}{true}
\def\change{}

\begin{document}
\begin{frontmatter}

\title{Certainty-equivalent adaptive MPC\\ for uncertain nonlinear systems\vspace{-2mm}
}
\author{Johannes Kohler}
\ead{j.kohler@imperial.ac.uk}  
\address{Department of Mechanical Engineering, Imperial College London, London, UK} 
\thanks{This research was primarily carried out while the author was at the Institute for Dynamic Systems and Control, ETH Zurich, Switzerland.}
\begin{abstract}
We provide a method to design adaptive controllers for nonlinear systems using model predictive control (MPC). 
By combining a certainty-equivalent MPC formulation with least-mean-square parameter adaptation, we obtain an adaptive controller with strong robust performance guarantees: The cumulative tracking error and violation of state constraints scale linearly with noise energy, disturbance energy, and path length of parameter variation. 
A key technical contribution is developing the underlying certainty-equivalent MPC that tracks output references, accounts for actuator limitations and desired state constraints, requires no system-specific offline design, and provides strong inherent robustness properties. 
This is achieved by leveraging finite-horizon rollouts, artificial references, recent analysis techniques for optimization-based controllers, and \change{relaxed soft state constraints}. 
For open-loop stable systems, we derive a semi-global result that applies to arbitrarily large measurement noise, disturbances, and parametric uncertainty. 
For stabilizable systems, we derive a regional result that is valid within a given region of attraction and for sufficiently small uncertainty. 
Applicability and benefits are demonstrated with numerical simulations involving systems with large parametric uncertainty: a linear stable chain of mass-spring-dampers and a nonlinear unstable quadrotor navigating obstacles.  
\end{abstract}
\end{frontmatter}
\section{Introduction} 
Controllers should ensure key requirements such as stability, performance, and constraint satisfaction. 
Achieving these objectives becomes particularly challenging when the model is nonlinear and uncertain. 
The treatment of model uncertainty is divided into \textit{robust} control methods~\cite{khalil1996robust} and \textit{adaptive} methods~\cite{goodwin2014adaptive}.  
In this paper, we study adaptive control methods due to their ability to accommodate large model uncertainty. 
 
Boundedness and convergence in adaptive control have been studied for multiple decades and corresponding design methods for direct or indirect adaptive control are well established~\cite{goodwin2014adaptive}.  
Furthermore, robustness issues in adaptive control due to noise or unmodelled dynamics~\cite{rohrs1982robustness} have been addressed in the literature~\cite{middleton1988design}.  
However, most adaptive control designs are limited to specific system classes - for example, linear dynamics, noise-free continuous-time measurements, or matched uncertainty; and face significant challenges when accounting for additional input or state constraints~\cite{anderson2008challenges,tao2014multivariable,annaswamy2021historical}. 
In this paper, we propose a simple-to-implement adaptive model predictive control (MPC) scheme that provides high-performance tracking for nonlinear uncertain systems while accounting for input and state constraints.  

\textit{Adaptive control for nonlinear systems:} 
One of the practical challenges in adaptive control for \textit{nonlinear} systems is the design of a feedback and Lyapunov function that ensures stability for all possible parameter values.
If a corresponding parametrized CLF is given, existing adaptive control methods can be applied to linearly parametrized uncertain nonlinear systems~\cite{pomet1992adaptive,krstic1995control,lopez2021universal}. 
Corresponding  CLFs can be constructed for specific system classes, e.g., using feedback linearization~\cite{campion1990indirect}, sliding-mode/boundary-layer control~\cite{slotine1986adaptive}, or backstepping~\cite{krstic1992adaptive}.
Nonlinear adaptive control is also possible using the immersion and invariance approach~\cite{astolfi2003immersion}, which requires the construction of model-specific functions offline. 
Other recent results for nonlinear adaptive control utilize control contraction metrics~\cite{lopez2021universal} or assume that the value function from infinite-horizon optimal control is analytically given~\cite{lopez2021adaptive}.  
Overall, application of \textit{nonlinear} adaptive control is limited due to the explicit offline design of suitable parametrized feedback and Lyapunov functions (cf.~\cite[Sec.~5.3]{tao2014multivariable}).  

\textit{Adaptive control under constraints:}
Adaptive control under additional \textit{input constraints}  needs to either assume open-loop stable systems or restrict attention to small parametric uncertainty and local initial states to ensure boundedness~\cite{annaswamy1997adaptive}.  
Furthermore, convergence results require additional modifications:  
In~\cite{chaoui1998adaptive,chaoui2001adaptive}, convergence for (asymptotically) constant references is achieved by using a specific pole placement formula that depends on the input saturation. 
Alternatively, in~\cite{zhong2005globally}, sufficient conditions for convergence are derived assuming bounds on the maximal reference magnitude and the parametric error. 
In~\cite{wen2011robust}, the input saturation is treated as a bounded nonlinearity and convergence is achieved using backstepping.  
Results for state constraints in adaptive control typically require the design of a suitable barrier function~\cite{liu2016barrier,taylor2020adaptive}, whose construction is, in general, even more challenging than that of a CLF. 
Hence, while there exist adaptive control methods to account for input saturation or even state constraints, the parametric error must be sufficiently small and additional modifications are required that may significantly degrade performance. 
Overall, the consideration of more general nonlinear systems, the incorporation of constraints, and transient performance beyond boundedness remain limiting factors for the application of adaptive control~\cite{annaswamy2021historical}. 

\textit{Adaptive  MPC}:  
In contrast to classical adaptive control methods,  {MPC} does not rely on an offline-parametrized analytical feedback law. 
Instead, feedback is generated by repeatedly solving finite-horizon open-loop optimal control problems online~\cite{rawlings2017model,grune2017nonlinear}. 
MPC can directly consider general nonlinear systems, incorporate both input and state constraints, and achieve close-to-optimal performance~\cite{rawlings2017model,grune2017nonlinear}.
The combination of MPC with online parameter adaptation has also been explored early in the literature, empirically demonstrating high performance~\cite{yoon1994adaptive}. 
However, the lack of a theoretical foundation in early MPC formulations has led many researchers to instead focus on classical (unconstrained) adaptive control~\cite{bitmead1990adaptive}. 
As a result, the theoretical development of adaptive MPC methods has historically received limited attention~\cite{mayne2014model}. 
\ifbool{arxiv}{In recent years, there have been significant contributions towards adaptive MPC, ranging from finite-impulse response (FIR) models~\cite{tanaskovic2014adaptive}, to general linear systems~\cite{lorenzen2019robust,strelnikova2024adaptive,peschke2023RAMPC_track}, to different classes of nonlinear models~\cite{gonccalves2016robust,lopez2019adaptive,kohler2020robust,sinha2021adaptive,sasfi2022robust,degner2024nonlinear}, and even to non-parametric models~\cite{manzano2020robust,scampicchio2025gaussian,dubied2025robust}.}{In recent years, \change{there have been significant contributions towards adaptive MPC for general linear models~\cite{lorenzen2019robust,strelnikova2024adaptive,peschke2023RAMPC_track} and different classes of nonlinear models~\cite{gonccalves2016robust,lopez2019adaptive,kohler2020robust,sinha2021adaptive,sasfi2022robust,degner2024nonlinear}.}}
However, these MPC formulations 
are best characterized as \emph{robust} MPC methods that additionally use online data to reduce conservatism. 
As a result, they can only be applied if the uncertainty is sufficiently small. 
Robust methods also often rely on offline-computed polytopic invariant sets, which results in scalability issues. 
Both of these issues are significant limitations of existing adaptive MPC formulations for \emph{linear} systems, and they are further exacerbated in the case of \emph{nonlinear} systems. 
In addition, most MPC formulations are limited to the problem of stabilizing a known equilibrium point, although the optimal steady-state is a function of the model parameters and thus requires online re-computation~\cite{peschke2023RAMPC_track,strelnikova2024adaptive,degner2024nonlinear}.  
\change{In contrast, the proposed MPC avoids these cumbersome offline design steps but relies on soft state constraints. }

\textit{Contribution:}  
We present a certainty-equivalent adaptive MPC scheme for nonlinear systems with large parametric uncertainty.  
We combine a certainty-equivalent tracking MPC with a least-mean-square parameter adaptation, yielding an adaptive controller with strong theoretical guarantees:
The cumulative tracking error and violation of state constraints scale linearly with noise energy, disturbance energy, and path length of parameter variation. 
This result relies on a novel certainty-equivalent MPC, which provides strong inherent robustness properties.  
This MPC design combines and extends results on: \\
(i) MPC theory without   {CLF}s~\cite{grune2017nonlinear,koehler2021LP};\\
(ii)  finite-horizon rollout penalties (cf.~\cite{magni2001stabilizing,kohler2021stability,bonassi2024nonlinear});\\  
(iii) MPC-for-tracking formulations~\cite{limon2018nonlinear,soloperto2021nonlinear,krupaModelPredictiveControl2024}. \\
We provide our design and analysis for two cases: 
a) For \emph{open-loop} (exponentially) \emph{stable} systems, we provide a \emph{semi-global} result that applies to arbitrarily large noise, disturbances, and parameter uncertainty.  
b) Under a weaker local \emph{stabilizability}
condition, we provide a \emph{regional} result which holds within a specified region of attraction for sufficiently small uncertainty.  
The proposed MPC approach is easy to implement, as no specific offline design is required. 
In contrast to classical adaptive control methods, we can directly consider input and \change{soft state constraints} and a general class of nonlinear systems.  
We demonstrate these advantages in two numerical examples with large parametric uncertainty: a linear stable mass-spring-damper chain and a nonlinear unstable quadrotor navigating obstacles.  




\textit{Notation:}  
The interior of a set $\mathbb{X}$ is denoted by $\mathrm{int}(\mathbb{X})$.  
The set of integers in the interval $[a,b]\subseteq\mathbb{R}$ is denoted by $\mathbb{I}_{[a,b]}$. 
By $\mathcal{K}_\infty$, we denote the set of continuous functions $\alpha:\mathbb{R}_{\geq 0}\rightarrow\mathbb{R}_{\geq 0}$ that are strictly increasing, unbounded, and satisfy $\alpha(0)=0$. 
We denote the Euclidean norm of a vector $x\in\mathbb{R}^n$ by $\|x\|=\sqrt{x^\top x}$. 
The point-to-set distance for a set $\mathcal{A}\subset\mathbb{R}^n$ and a vector $x\in\mathbb{R}^n$ is denoted by $\|x\|_{\mathcal{A}}:=\inf_{s\in\mathcal{A}}\|x-s\|$. 
By $Q\succ 0$ ($Q\succeq 0$), we indicate that a matrix $Q=Q^\top\in\mathbb{R}^{n\times n}$ is symmetric and positive definite (positive semi-definite). 
We denote the maximal and minimal eigenvalue of a matrix $Q=Q^\top$ by $\sigma_{\max}(Q)$ and $\sigma_{\min}(Q)$, respectively. 
For a positive definite matrix $Q=Q^\top\in\mathbb{R}^n$ and a vector $x\in\mathbb{R}^n$, we denote $\|x\|_Q:=\sqrt{x^\top Qx}$. 
The identity matrix is $I$.

\section{Problem setup and proposed approach}
\label{sec:certainty_equiv_setup}
We first pose the control problem, describe the proposed approach, and detail the considered assumptions.

\subsection{Problem setup}
We consider a nonlinear discrete-time system
\begin{align}
\label{eq:sys}
x_{k+1}=\fw(x_k,u_k,\theta_k,w_k),\quad k\in\mathbb{I}_{\geq 0},
\end{align}
with state $x_k\in\mathbb{R}^{\nx}$, control input $u_k\in\mathbb{U}\subseteq\mathbb{R}^{\nu}$, parameter $\theta_k\in\mathbb{R}^{\ntheta}$, disturbances $w_k\in\mathbb{R}^{\nw}$, time $k\in\mathbb{I}_{\geq 0}$, and a compact input constraint set $\mathbb{U}$. 
The parameters can change \change{over time} with $\theta_{k+1}=\theta_k+\Delta \theta_k\in\mathbb{R}^{\ntheta}$, $k\in\mathbb{I}_{\geq 0}$. 
We have access to noisy state measurements
\begin{align}
\label{eq:noise_measurement}
\hat{x}_k=x_k+v_k,\quad k\in\mathbb{I}_{\geq0 },
\end{align}
 with measurement noise $v_k\in\mathbb{R}^{\nx}$. 
\change{Ideally, the state should lie in a desired polytopic\ifbool{arxiv}{\footnote{%
For general closed sets $\mathbb{X}$, the standard penalty in~\eqref{eq:ell} needs to be replaced by the point-to-set distance $\|x\|_{\mathbb{X}}^2$.}}{} set 
 \begin{align}
\label{eq:sys_constraint}
\mathbb{X}=\{x\in\mathbb{R}^{\nx}|~D_i x\leq d_i,\quad i\in\mathbb{I}_{[1,r]}\}.
\end{align}
We address this through soft penalties that allow for violations.} 
We have given an output target $y_{\mathrm{d}}\in\mathbb{R}^{\ny}$ that should be tracked by the system output 
\begin{align}
\label{eq:sys_output}
y_k=h(x_k,u_k,\theta_k)\in\mathbb{R}^{\ny},\quad k\in\mathbb{I}_{\geq 0}.
\end{align}
While the functions $f,h$ are known, the parameters $\theta_k$, disturbances $w_k$, and noise $v_k$ are not known. 
The (unknown) optimal setpoint is given by
\begin{align}
\label{eq:opt_steady_state}
(x_{\mathrm{rd},\theta},u_{\mathrm{rd},\theta},y_{\mathrm{rd},\theta})\in\arg\min_{(x,u,y)\in\mathbb{S}(\theta)}\|y_{\mathrm{d}}-y\|_T^2,
\end{align}
which is defined with some positive definite weighting matrix $T\in\mathbb{R}^{\ny\times\ny}$ and the set of feasible setpoints
\begin{align}
\label{eq:steady_state_manifold}
\mathbb{S}(\theta)=\{(&x,u,y)\in\mathbb{X}\times\mathbb{U}\times\mathbb{R}^{\ny}|\nonumber\\
&~\fw(x,u,\theta,0)=x,~y=h(x,u,\theta)\}.
\end{align} 
We want to track the optimal feasible output $y_{\mathrm{rd},\theta}$, which coincides with the target $y_{\mathrm{d}}$ whenever the latter is feasible. 
Note that the optimal steady-state $x_{\mathrm{rd},\theta}$ is unknown since it depends on the unknown parameters $\theta$. 
We consider two different objectives: 
\begin{tcolorbox}[colframe=white!, top=2pt,left=2pt,right=6pt,bottom=2pt]
\begin{objective}[Semi-global] 
\label{objective_1} 
Given any compact sets $\mathbb{X}_0,\Theta,\mathbb{W},\mathbb{V}$, design an adaptive controller that ensures  
\begin{align}
&\sum_{k=0}^{\Time-1}\left[\|y_k-y_{\mathrm{rd},\theta_k}\|^2+\|x_k\|_{\mathbb{X}}^2\right]\label{eq:desired_stability}\\
\leq& C_1\sum_{k=0}^{\Time-1}\left[\|w_k\|^2+\|v_k\|^2+\|\Delta \theta_k\|\right]\nonumber\\
&+C_2(\|\theta_0-\hat{\theta}_0\|^2+\|x_0-x_{\mathrm{rd},\theta_0}\|^2),~\forall \Time\in\mathbb{I}_{\geq 0},\nonumber
\end{align}
with \change{constants} $C_1,C_2\geq 0$, 
for all initial conditions $x_0\in\mathbb{X}_0$, $\hat{\theta}_0\in\Theta$, (unknown) parameters $\theta_k\in\Theta$, disturbances $w_k\in\mathbb{W}$, and measurement noise $v_k\in\mathbb{V}$, $k\in\mathbb{I}_{\geq 0}$.
\end{objective}
\end{tcolorbox}
Inequality~\eqref{eq:desired_stability} reflects the fact that both the tracking error and the state constraint violation get small if the measurement noise and disturbances become small and the parameter variation is small. Due to the adaptation, the initial parameter error $\hat{\theta}_0-\theta_0$ only has a transient effect that becomes negligible as $T\rightarrow\infty$. 
This result is semi-global as arbitrarily large sets $\mathbb{X}_0,\mathbb{W},\mathbb{V},\Theta$ are admissible, however, the adaptive control design may depend on a conservative bound of these values.  
Depending on the precise regularity conditions, the result may additionally require sufficiently small parameter variations $\Delta \theta_k$, see Appendix~\ref{app:bounded_TV} for details. 
Since this objective cannot be achieved for general unstable dynamics under compact input constraints~$\mathbb{U}$, we also consider the following weaker \emph{regional} version. 
\begin{tcolorbox}[colframe=white!, top=2pt,left=2pt,right=6pt,bottom=2pt]
\begin{objective}[Regional]
\label{objective_2} 
Given a suitable compact set of initial conditions $(\hat{x}_0,\hat{\theta}_0)\in\mathbb{X}_{\bar{\mathcal{J}}}$ and sufficiently small bounds on parameters $\theta_k\in\Theta$, disturbances $w_k\in\mathbb{W}$, and measurement noise $v_k\in\mathbb{V}$, design an adaptive controller such that the closed-loop system satisfies~\eqref{eq:desired_stability}.
\end{objective}
\end{tcolorbox} 

\subsection{Proposed approach \& Outline}
The proposed approach combines a certainty-equivalent tracking MPC with parameter adaptation. 
We introduce the parameter adaptation in Section~\ref{sec:LMS}, which provides suitable bounds on the cumulative closed-loop prediction error given that the MPC keeps the system in some compact set. 
 The certainty-equivalent MPC is designed to track the desired setpoint and satisfy constraints when a perfect prediction model is available, while also providing inherent robustness to disturbances, noise, and inaccurate model parameters. 
 Section~\ref{sec:global} presents such an MPC formulation for open-loop stable systems, which can handle arbitrarily large disturbances and noise. This approach achieves  Objective~\ref{objective_1} when combined with the parameter adaptation.
For Objective~\ref{objective_2}, a relaxed local stabilizability condition is considered and a slightly modified MPC design is presented, which is analysed in Section~\ref{sec:regional}. 
Section~\ref{sec:discussion} provides a discussion and Section~\ref{sec:num} presents numerical examples. Detailed proofs and auxiliary results are provided in the appendix.

\subsection{Technical assumptions}
We consider the following standing assumptions.
\begin{assumption}(Standing assumptions)
\label{ass:regularity}
\begin{enumerate}[label=\alph*)]
\item (Known set of parameters) There exists a known  set $\Theta\subseteq\mathbb{R}^{\ntheta}$, such that $\theta_k\in\Theta$, $\forall k\in\mathbb{I}_{\geq 0}$.
\label{item:ass_parameter_set}
\item (Feasible setpoints) The set~$\mathbb{S}(\theta)$ is non-empty, closed, and uniformly bounded $\forall \theta\in\Theta$. 
\label{item:ass_steady_state}
\item (Regularity)
The functions $\fw$, $h$ are continuously differentiable and Lipschitz continuous with constants $L_{\mathrm{\fw}}$, $L_{\mathrm{h}}$. 
The sets $\mathbb{U}$, $\Theta$ are convex and compact. 
\label{item:ass_Lipschitz_differentiable}
\end{enumerate}
\end{assumption}
While Conditions~\ref{item:ass_parameter_set}--\ref{item:ass_steady_state} in Assumption~\ref{ass:regularity} are quite standard, the \change{Lipschitz constant} in condition~\ref{item:ass_Lipschitz_differentiable} can be restrictive. 
Appendix~\ref{app:bounded_TV} details how it can be relaxed to a local Lipschitz constant. 
To leverage standard model adaptation techniques, we assume that the model is \textit{linearly} parametrized.
\begin{assumption}(Linear parametrization)
\label{ass:linear_param}
The dynamics~\eqref{eq:sys} are linear in $\theta$, i.e., 
\begin{align*}
\fw(x,u,\theta,w)\equiv \fw(x,u,0,w)+G(x,u,w)\theta. 
\end{align*} 
\end{assumption}
The following assumption characterizes further conditions on the steady-state manifold.  
\begin{assumption}(Regularity of optimal steady-states)
\label{ass:steady_regular}
\begin{enumerate}[label=\alph*)]
\item (Unique optimal setpoint) 
There exist functions $\gyx:\mathbb{R}^{\ny}\times\Theta\rightarrow\mathbb{R}^{\nx}$, $\gyu:\mathbb{R}^{\ny}\times\Theta\rightarrow\mathbb{R}^{\nu}$, such that for all $\theta\in\Theta$ and any $(x,u,y)\in\mathbb{S}(\theta)$, we have  $x=\gyx(y,\theta)$\change{,} $u=\gyu(y,\theta)$.  
Furthermore, $\gyx,\gyu$ are Lipschitz continuous, i.e., there exists a constant $L_{\mathrm{\gy}}\geq 0$, such that for all $\theta,\tilde{\theta}\in\Theta$, $y,\tilde{y}\in\mathbb{R}^{\ny}$:%
\begin{align}
\label{eq:g_Lipschitz}
&\|\gyx(y,\theta)-\gyx(\tilde{y},\tilde{\theta})\|_Q^2+\|\gyu(y,\theta)-\gyu(\tilde{y},\tilde{\theta})\|_R^2\nonumber\\
\leq &L_{\mathrm{\gy}}(\|y-\tilde{y}\|_T^2+\|\theta-\tilde{\theta}\|^2).
\end{align}
\label{item:ass_unique}
\vspace{-\baselineskip}
\item (Convex output space)  For any $\theta\in\Theta$, and all $(x_1,u_1,y_1)\in\mathbb{S}(\theta)$, $(x_2,u_2,y_2)\in\mathbb{S}(\theta)$, $\beta\in[0,1]$, the convex combination $\tilde{y}=\beta y_1+(1-\beta)y_2$, satisfies
$(\tilde{x},\tilde{u},\tilde{y})\in\mathbb{S}(\theta)$
with  $\tilde{x}=\gyx(\tilde{y},\theta)$, $\tilde{u}=\gyu(\tilde{y},\theta)$. 
\label{item:ass_convex}
\item (Feasible target)
For any $\theta\in\Theta$, $y_{\mathrm{rd},\theta}=y^{\mathrm{d}}$. 
\label{item:ass_regularity_steady_state_1}
\item (Regularity - setpoints) For any parameters $\theta,\tilde{\theta}\in\Theta$,  setpoint $(x_{\mathrm{s}},u_{\mathrm{s}},y_{\mathrm{s}})\in \mathbb{S}(\theta)$, it holds that
\begin{align}
\label{eq:regularity_steady_state}
\min_{(\tilde{x}_{\mathrm{s}},\tilde{u}_{\mathrm{s}},\tilde{y}_{\mathrm{s}})\in \mathbb{S}(\tilde{\theta})}\|y_s-\tilde{y}_s\|_T^2\leq L_{\mathrm{s}}\|\theta-\tilde{\theta}\|^2. 
\end{align}
\label{item:ass_regularity_steady_state_2}
\end{enumerate}
\vspace{-\baselineskip}
\end{assumption}
In the special case of known model parameters $\Theta=\{\theta\}$, these conditions reduce to standard assumptions from the MPC-for-tracking literature~\cite{limon2018nonlinear,krupaModelPredictiveControl2024}. 
Conditions~\ref{item:ass_unique}--\ref{item:ass_convex} ensure that the optimal steady-state $x_{\mathrm{rd},\theta}$~\eqref{eq:opt_steady_state} is unique. 
Condition~\ref{item:ass_unique} follows from the implicit function theorem under standard regularity conditions if the system is square ($\nu=\ny\leq \nx$)~\cite[Remark~1]{limon2018nonlinear}. 
Note that Condition~\ref{item:ass_convex} does not require convexity of the set $\mathbb{S}(\theta)$, but only its projection on the output coordinates, which holds for many nonlinear systems.  
Condition~\ref{item:ass_regularity_steady_state_1} requires that the output target $y^{\mathrm{d}}$ is feasible.  
Condition~\ref{item:ass_regularity_steady_state_2} requires further that the change in the feasible outputs is bounded by the change in the parameters. 

In addition to Assumptions~\ref{ass:regularity}--\ref{ass:steady_regular}, we impose either an open-loop stability condition (Asm.~\ref{ass:stable}) or a local stabilizability condition (Asm.~\ref{ass:stable_feedback}) in Sections~\ref{sec:global} and \ref{sec:regional}, respectively.
Furthermore, the parameter adaptation in Section~\ref{sec:LMS} will invoke an additional condition on the parameter gain (Asm.~\ref{ass:param_gain}), which is constructively satisfied in Sections~\ref{sec:global} and \ref{sec:regional}, respectively.
In Section~\ref{sec:discussion_linear}, we show how these assumptions simplify in case of linear dynamics.

\section{Parameter adaptation}
\label{sec:LMS}
In this section, we present the parameter adaptation method. 
Recall that we know a compact and convex set $\Theta$ containing the true unknown parameters $\theta_k$ (Asm.~\ref{ass:regularity}) and 
the model is linear in the unknown parameters (Asm.~\ref{ass:linear_param}). 
Similar to~\cite{lorenzen2019robust,kohler2020robust,degner2024nonlinear}, we use a projected least-mean-square (LMS) adaptation:
\begin{subequations}
\label{eq:LMS}
\begin{align}
\label{eq:LMS_1}
\tilde{\theta}_{k+1}=&\hat{\theta}_k+\Gamma\hat{\Phi}_k^\top(\hat{x}_{k+1}-\hat{x}_{1|k}),\\  
\label{eq:LMS_2}
\hat{\theta}_{k+1}=&\arg\min_{\theta\in\Theta}\|\theta-\tilde{\theta}_{k+1}\|_{\Gamma^{-1}}^2,
\end{align}
\end{subequations}
where $\Gamma=\Gamma^\top\succ 0$, $\Gamma\in\mathbb{R}^{\ntheta\times\ntheta}$ is a parameter gain, $\hat{x}_{1|k}=\fw(\hat{x}_k,u_k,\hat{\theta}_k,0)$ is the certainty-equivalent one-step prediction,  $\hat{\Phi}_k=G(\hat{x}_k,u_k,0)\in\mathbb{R}^{\nx\times\ntheta}$ is the regressor, and  $\hat{\theta}_0\in\Theta$ is an initial estimate. 
 Equation~\eqref{eq:LMS_2} invokes a weighted projection on the parameter set $\Theta$.  
\change{The gain $\Gamma$ is analogous to the step size in gradient descent~\cite{goodwin2014adaptive} and needs to be chosen sufficiently small to ensure stability
of the discrete-time parameter adaptation.}
\begin{assumption} (LMS parameter gain)
\label{ass:param_gain}
The parameter gain satisfies $\Gamma\succ 0$ and $\hat{\Phi}_k \Gamma\hat{\Phi}_k^\top\preceq I$ $\forall k\in\mathbb{I}_{\geq 0}$.
\end{assumption}
\begin{remark}\change{(Parameter gain)
\label{rk:parameter_gain}
A sufficiently small gain $\Gamma\succ 0$ satisfying Assumption~\ref{ass:param_gain} can be chosen if an upper bound on $\|\hat{\Phi}_k\|$ is known. 
For many applications, such a bound is a priori known because the physical system operates in a known domain or the regressor is uniformly bounded, as is the case for sigmoid and trigonometric  functions. 
More generally, the stability properties of the controller are leveraged to derive such a bound. This is addressed separately for Objectives~\ref{objective_1} and \ref{objective_2} in Sections~\ref{sec:global} and \ref{sec:regional}, respectively.}
\end{remark}
\begin{theorem}
\label{thm:LMS}
Let Assumptions~\ref{ass:regularity}\ref{item:ass_parameter_set}, \ref{item:ass_Lipschitz_differentiable}, \ref{ass:linear_param}, \ref{ass:param_gain} hold and 
consider $\Vtheta(\theta):=\|\theta\|_{\Gamma^{-1}}^2$, 
 system~\eqref{eq:sys}, and noisy measurements~\eqref{eq:noise_measurement}. 
For all $k\in\mathbb{I}_{\geq 0}$, the LMS update~\eqref{eq:LMS} satisfies%
\begin{subequations}
\label{eq:LMS_property}
\begin{align}
\label{eq:LMS_Lyap}
&\Vtheta(\hat{\theta}_{k+1}-\theta_{k+1})-\Vtheta(\hat{\theta}_{k}-\theta_{k})\\
\leq& -\|\tilde{x}_{1|k}\|^2+\|\tilde{w}_k\|^2+c_\theta\sqrt{\Vtheta(\Delta \theta_k)},\nonumber\\
\label{eq:LMS_step}
\Vtheta(\hat{\theta}_{k+1}-\hat{\theta}_k)\leq& \|\tilde{x}_{1|k}+\tilde{w}_k\|^2,\\
\label{eq:LMS_w_bound}
\|\tilde{w}_k\|\leq&L_{\mathrm{f}}(\|w_k\|+\|v_k\|)+\|v_{k+1}\|,
\end{align}
\end{subequations}
with
\begin{subequations}
\label{eq:LMS_tilde}
\begin{align}
\label{eq:LMS_tilde_x}
\tilde{x}_{1|k}:=&\hat{\Phi}_k(\theta_k-\hat{\theta}_k),\\
\label{eq:LMS_tilde_w}
\tilde{w}_k:=&\hat{x}_{k+1}-\fw(\hat{x}_k,u_k,\theta_k,0)\\
=&\fw(x_k,u_k,\theta_k,w_k)+v_{k+1}-\fw(x_k+v_k,u_k,\theta_k,0),\nonumber\\
\label{eq:LMS_c_theta}
c_\theta:=&\max_{\theta,\theta'\in\Theta}\sqrt{\Vtheta(\theta-\theta')}.
\end{align}
\end{subequations} 
\end{theorem} 
Theorem~\ref{thm:LMS} extends the LMS results from~\cite[Lem.~5]{lorenzen2019robust} and \cite[Prop.~1]{degner2024nonlinear} to account for general gains $\Gamma\succ 0$, noisy measurements $\hat{x}_k\neq x_k$, and time-varying parameters $\theta_k$. 
The proof can be found in Appendix~\ref{app:proof_LMS}. 
In Equation~\eqref{eq:LMS_tilde}, $\tilde{x}_{1|k}$ reflects the prediction error due to the inaccurate parameter estimate while $\tilde{w}_k$ reflects the prediction error due to disturbances $w_k$ and measurement noise $v_k$, $v_{k+1}$. 
Inequality~\eqref{eq:LMS_Lyap} ensures that the model prediction error $\tilde{x}_{1|k}$ decays asymptotically if the disturbances, noise and parameter variations decay.

\section{\change{Adaptive tracking for stable systems}}
\label{sec:global}
In this section, we address Objective~\ref{objective_1}, i.e., designing a semi-global adaptive tracking controller for open-loop stable systems. 
First, we introduce the open-loop stability assumption (Sec.~\ref{sec:global_stable})
before presenting the proposed adaptive tracking MPC (Sec.~\ref{sec:global_certainty_equiv_design}). 
Then, we establish closed-loop properties of this certainty-equivalent formulation, including nominal stability (Sec.~\ref{sec:global_analysis}) and inherent robustness (Sec.~\ref{sec:global_robust}). 
Finally, we show that combining the certainty-equivalent MPC with the LMS parameter adaptation yields an adaptive controller that achieves Objective~\ref{objective_1} (Sec.~\ref{sec:global_adaptive}).

\subsection{Open-loop stable systems}
\label{sec:global_stable}
For a given initial state $x\in \mathbb{R}^{\nx}$, constant parameter $\theta\in\Theta$, input sequence $\mathbf{u}\in\mathbb{U}^N$, we denote the solution to~\eqref{eq:sys} with $w\equiv 0$ after $k$ steps by $x_{\mathbf{u}}(k,x,\theta)\in \mathbb{R}^{\nx}$, $k\in\mathbb{I}_{[0,N]}$ with $x_{\mathbf{u}}(0,x,\theta)=x$. 
For $\mathbf{u}\in\mathbb{U}^N$, $\mathbf{u}_k\in\mathbb{U}$ denotes the $k$-th element in
the sequence with $k\in\mathbb{I}_{[0,N-1]}$. 
Given any input $u_{\mathrm{s}}\in\mathbb{U}$, we denote the constant input sequence by $\mathbf{u}_{\mathrm{s}}^N\in\mathbb{U}^N$. 
In this section, we consider open-loop exponentially stable systems, as formalized below. 
\begin{assumption}(Open-loop exponentially stable)
\label{ass:stable}
There exist constants $\Crho\geq 1$, $\rho\in(0,1)$, such that for any parameter $\theta\in\Theta$, any corresponding stationary state-input pair $(x_{\mathrm{s}},u_{\mathrm{s}})\in\mathbb{R}^{\nx}\times \mathbb{U}$, i.e., those satisfying $x_{\mathrm{s}}=\fw(x_{\mathrm{s}},u_{\mathrm{s}},\theta,0)$, and any state $x\in \mathbb{R}^{\nx}$, 
 we have 
\begin{align}
\label{eq:exp_stable}
\|x_{\mathbf{u}_{\mathrm{s}}^k}(k,x,\theta)-x_{\mathrm{s}}\|\leq \Crho\rho^k\|x-x_{\mathrm{s}}\|,\quad k\in\mathbb{I}_{\geq 0}.
\end{align}
\end{assumption}
This assumption implies that a constant input yields exponential stability of the associated steady state $x_{\mathrm{s}}$.

\subsection{Certainty-equivalent tracking MPC}
\label{sec:global_certainty_equiv_design}
The goal of steering the system to a given setpoint $(x_{\mathrm{s}},u_{\mathrm{s}})$ while penalizing state constraint violations is encoded in the following standard stage cost (cf.~\cite{limon2018nonlinear,zeilinger2014soft})
\begin{align}
\label{eq:ell}
&\ell(x,u,x_{\mathrm{s}},u_{\mathrm{s}})\\
:=&\|x-x_{\mathrm{s}}\|_Q^2+\|u-u_{\mathrm{s}}\|_R^2+\sum_{i=1}^r q_{\xi,i}\max\{D_i x-d_i,0\}^2,\nonumber
\end{align}
with weighting matrices $Q,R\succ 0$ and weights $q_{\xi,i}> 0$. 
The penalty on the constraint violation can be implemented by defining slack variables~\cite{zeilinger2014soft}. 
Given a horizon $M\in\mathbb{I}_{\geq 1}$, we define the terminal penalty using the following finite-tail cost (cf.~\cite{magni2001stabilizing,kohler2021stability,bonassi2024nonlinear})
\begin{align}
\label{eq:V_f}
\Vf(x,x_{\mathrm{s}},u_{\mathrm{s}},\theta):=\sum_{k=0}^{M-1}\ell(x_{\mathbf{u}_{\mathrm{s}}^M}(k,x,\theta),u_{\mathrm{s}},x_{\mathrm{s}},u_{\mathrm{s}}). 
\end{align} 
For a given state $x\in\mathbb{R}^{\nx}$, parameter $\theta\in\Theta$, horizon $N\in\mathbb{I}_{\geq 1}$, input sequence $\mathbf{u}\in\mathbb{U}^N$, and setpoint $(x_{\mathrm{s}},u_{\mathrm{s}},y_{\mathrm{s}})\in\mathbb{S}(\theta)$, we define the finite-horizon cost 
\begin{align}
\label{eq:J_N}
&\JN(x,\theta,\mathbf{u},x_{\mathrm{s}},u_{\mathrm{s}})\\
=&\sum_{k=0}^{N-1}\ell(x_{\mathbf{u}}(k,x,\theta),\mathbf{u}_k,x_{\mathrm{s}},u_{\mathrm{s}})+\omega \Vf(
x_{\mathbf{u}}(N,x,\theta),x_{\mathrm{s}},u_{\mathrm{s}},\theta),\nonumber
\end{align}
with some tunable weight $\omega> 0$.
The proposed tracking MPC is characterized by the following optimization problem:
\begin{align}
\label{eq:MPC}
\min_{\mathbf{u}\in\mathbb{U}^N,(x_{\mathrm{s}},u_{\mathrm{s}},y_{\mathrm{s}})\in\mathbb{S}(\theta)}\JN(x,\theta,\mathbf{u},x_{\mathrm{s}},u_{\mathrm{s}})+\|y_{\mathrm{s}}-y_{\mathrm{d}}\|_T^2.
\end{align}
This MPC design leverages online computed setpoints $(x_{\mathrm{s}},u_{\mathrm{s}},y_{\mathrm{s}})$~\cite{limon2018nonlinear,soloperto2021nonlinear,krupaModelPredictiveControl2024}, a penalty $\Vf$ based on a finite-horizon rollout~\cite{magni2001stabilizing,kohler2021stability,bonassi2024nonlinear}, and soft state constraints~\cite{zeilinger2014soft}. 
These components ensure that the method is easy to implement and provides strong stability and robustness properties, see Section~\ref{sec:discussion} for a discussion. 
A resulting minimizer is denoted by $\mathbf{u}^\star_{x,\theta}\in\mathbb{U}^N$, $(x_{\mathrm{s},x,\theta}^\star,u_{\mathrm{s},x,\theta}^\star,y_{\mathrm{s},x,\theta}^\star)\in\mathbb{S}(\theta)$ and the corresponding minimum is $\JN^\star(x,\theta)$.\footnote{%
A minimizer exists since the constraints are compact and the involved functions are continuous~\cite[Prop.~2.4]{rawlings2017model}. We assume w.l.o.g. \change{that a minimizer is chosen in a unique way}.} 
The corresponding feedback $\pi:\mathbb{R}^{\nx}\times\Theta\rightarrow\mathbb{U}$ applies the first element of a minimizing input sequence $u_{x,\theta,0}^\star \in \mathbb{U}$.
The closed-loop system is given by
\begin{align}
\label{eq:closed_loop_adaptive}
x_{k+1}=\fw(x_k,\pi(\hat{x}_k,\hat{\theta}_k),\theta_k,w_k),\quad k\in\mathbb{I}_{\geq 0},
\end{align}
where the MPC policy is computed based on the state estimate $\hat{x}_k$ and the parameter estimate $\hat{\theta}_k$.

\subsection{Nominal stability analysis}
\label{sec:global_analysis}
We first focus on the \emph{nominal} closed-loop system
\begin{align}
\label{eq:nominal_closed_loop}
x_{k+1}=\fw(x_k,\pi(x_k,\theta),\theta,0), ~k\in\mathbb{I}_{\geq 0},
\end{align}
which considers no noise, no disturbances, and constant, exactly known parameters. 
The following analysis combines results on the finite-tail cost $\Vf$~\cite{kohler2021stability}, stability analysis techniques for MPC schemes without a local CLF~\cite{koehler2021LP}, and tracking MPC formulations~\cite{soloperto2021nonlinear}.
We first establish some preliminary results regarding the stage cost $\ell$, the terminal cost $\Vf$, and the value function $\JN^\star$.
Then, we show that a suitable choice of the scaling $\omega\geq 1$ and rollout horizon $M$ ensures stability. 
Finally, we show asymptotic stability of the optimal setpoint. 
The proofs are deferred to Appendix~\ref{app:proof_global_analysis}. 
%

\textit{Preliminary results:} 
First, we provide a quadratic bound for the stage cost. 
\begin{proposition}
\label{prop:ell_exp_stable}
Let Assumption~\ref{ass:stable} hold. 
There exists a constant $\Cell\geq 1$, such that for any $(x,\theta)\in\mathbb{R}^{\nx}\times\Theta$ and any $(x_{\mathrm{s}},u_{\mathrm{s}},y_{\mathrm{s}})\in\mathbb{S}(\theta)$, it holds
\begin{align}
\label{eq:ell_exp_stable}
\ell(x_{\mathbf{u}_{\mathrm{s}}^k}(k,x,\theta),u_{\mathrm{s}},x_{\mathrm{s}},u_{\mathrm{s}})\leq \Cell\rho^k\|x-x_{\mathrm{s}}\|_Q^2,~ k\in\mathbb{I}_{\geq 0}. 
\end{align}
\end{proposition} 
The following proposition provides an upper bound on the value function, which also corresponds to the \textit{exponential cost controllability} condition in~\cite{grune2017nonlinear,koehler2021LP}.
\begin{proposition}
\label{prop:cost_controllable}
Let Assumption~\ref{ass:stable} hold.
There exist constants $\gamma_{N}\geq 1$, $N\in\mathbb{I}_{\geq 1}$, such that for any $N\in\mathbb{I}_{\geq 1}$, for any $(x,\theta)\in\mathbb{R}^{\nx}\times\Theta$, and any $(x_{\mathrm{s}},u_{\mathrm{s}},y_{\mathrm{s}})\in\mathbb{S}(\theta)$, it holds that
\begin{align}
\label{eq:cost_control}
\min_{\mathbf{u}\in\mathbb{U}^N}\JN(x,\theta,\mathbf{u},x_{\mathrm{s}},u_{\mathrm{s}})\leq \gamma_{N}\|x-x_{\mathrm{s}}\|_Q^2.
\end{align}
\end{proposition}
\ifbool{arxiv}{This result also yields a \change{horizon independnet} upper  bound with 
\begin{align*}
\gamma_{N}\leq \overline{\gamma}:=\Cell\dfrac{\max\{1,\omega\}}{1-\rho},\quad N\in\mathbb{I}_{\geq 1}.
\end{align*}}{}
The following proposition characterizes how the terminal penalty $\Vf$ is an approximate CLF by extending the linear programming (LP) analysis from~\cite{kohler2021stability}.
\begin{proposition}
\label{prop:finite_tail} 
Let Assumption~\ref{ass:stable} hold and suppose $\omega\geq 1$.
Then, for any $(x,\theta)\in\mathbb{R}^{\nx}\times\Theta$ and any $(x_{\mathrm{s}},u_{\mathrm{s}},y_{\mathrm{s}})\in\mathbb{S}(\theta)$, it holds that
\begin{align}
\label{eq:approx_CLF}
&\min_{u\in\mathbb{U}}\omega \Vf(\fw(x,u,\theta,0),x_{\mathrm{s}},u_{\mathrm{s}},\theta)+\ell(x,u,x_{\mathrm{s}},u_{\mathrm{s}})\nonumber\\
\leq& (1+\epsilon_{\mathrm{f}})\omega \Vf(x,x_{\mathrm{s}},u_{\mathrm{s}},\theta),\\
\label{eq:epsilon_f}
&\epsilon_{\mathrm{f}}:=\max\left\{\dfrac{1-\rho}{1-\rho^M}\left[\Cell\rho^M+\dfrac{1-\omega}{\omega}\right],0\right\}.
\end{align}
\end{proposition}
By choosing the horizon $M$ such that $\Cell\rho^M<1$ and the weight $\omega\geq {1}/(1-\Cell\rho^M)>1$, we can ensure $\epsilon_{\mathrm{f}}=0$.

\textit{Stability analysis:}
The following theorem establishes sufficient conditions to ensure that the value function $\JN^\star$ decreases in the nominal case. 
\begin{theorem}
\label{thm:stable_artificial}
Let Assumptions~\ref{ass:regularity}\ref{item:ass_steady_state} and \ref{ass:stable} hold. 
Given any $\omega\geq 1$, $M\in\mathbb{I}_{\geq 0}$, there exist a horizon  $N\in\mathbb{I}_{\geq 1}$ \change{and a constant $\alpha>0$}, such that the following inequality holds for all $(x,\theta)\in\mathbb{R}^{\nx}\times\Theta$: 
\begin{align}
\label{eq:RDP}
&\JN^\star(\fw(x,\pi(x,\theta),\theta,0),\theta)-\JN^\star(x,\theta)\nonumber\\
\leq&
-\alpha\cdot \ell(x,\pi(x,\theta),x_{\mathrm{s},x,\theta}^\star,u_{\mathrm{s},x,\theta}^\star). 
\end{align}
Furthermore, in case $\Cell\rho^M<1$, there exists a constant $\underline{\omega}\geq 1$, such that $\alpha>0$ holds for all $\omega>\underline{\omega}$ and all horizons $N\in\mathbb{I}_{\geq 1}$.
\end{theorem} 
Using Inequality~\eqref{eq:RDP} with $\alpha>0$ in a telescopic sum and the definition of the stage cost $\ell$~\eqref{eq:ell}, we can ensure that the nominal closed-loop system~\eqref{eq:nominal_closed_loop} converges to a steady-state $x_{\mathrm{s}}$ and the cumulative state constraint violation is bounded. 
The exact condition to ensure $\alpha>0$ can be found in~\eqref{eq:alpha} in Appendix~\ref{app:proof_global_analysis}, which adapts the worst-case linear programming analysis from~\cite[Thm. 6--7]{koehler2021LP}. 
This condition also highlights that the horizon $N$ can be drastically reduced by slightly increasing the rollout horizon $M$, without affecting the stability guarantees.  
A scaling of $\omega=1$ is desirable to approximate optimal performance (cf. discussion and results in~\cite[Se.~4]{magni2001stabilizing}, \cite[Sec.~4]{koehler2021LP}), while short horizons $N,M$ reduce computational demand.  

\textit{Stability of the optimal reachable steady-state:}  
The following proposition is an adaptation of~\cite[Prop.~1]{soloperto2021nonlinear}.  
\begin{proposition}
\label{prop:artificial_bound}
Let Assumptions~\ref{ass:steady_regular}\ref{item:ass_unique}, \ref{item:ass_convex}, and \ref{ass:stable} hold.
There exists a constant $\overline{a}>0$, such that for any $(x,\theta)\in\mathbb{R}^{\nx}\times\Theta$, the minimizer to Problem~\eqref{eq:MPC} satisfies
\begin{align}
\label{eq:artificial_bound}
\|x-x_{\mathrm{s},x,\theta}^\star\|_Q^2\geq \overline{a}\|x_{\mathrm{s},x,\theta}^\star-x_{\mathrm{rd},\theta}\|_Q^2.
\end{align}
\end{proposition}  
Inequality~\eqref{eq:artificial_bound} ensures that convergence to the artificial steady-state $x_{\mathrm{s},x,\theta}^\star$ (optimized by the MPC) ensures convergence to the optimal steady-state $x_{\mathrm{rd},\theta}$. 
\begin{corollary}
\label{corol:nominal}
Let Assumptions~\ref{ass:regularity}\ref{item:ass_steady_state}, \ref{ass:steady_regular}\ref{item:ass_unique}, \ref{item:ass_convex}, and  \ref{ass:stable} hold. 
Suppose $N,M,\omega$ are chosen such that $\alpha>0$ according to Theorem~\ref{thm:stable_artificial}. 
Then, for any $(x_0,\theta)\in\mathbb{R}^{\nx}\times\Theta$, the optimal steady-state $x_{\mathrm{rd},\theta}$ is exponentially stable for the nominal closed loop~\eqref{eq:nominal_closed_loop}. 
\end{corollary} 
Corollary~\ref{corol:nominal} ensures that the proposed MPC scheme successfully stabilizes the optimal setpoint and tracks the desired reference $y_{\mathrm{d}}$ as well as possible. 
However, the stability guarantees in Corollary~\ref{corol:nominal} are only valid for the nominal system~\eqref{eq:nominal_closed_loop}, which neglects noise, disturbances, and uncertainty in the model parameters. 


\subsection{Inherent robustness}
\label{sec:global_robust}
In this section, we establish inherent robustness of the certainty-equivalent tracking MPC to disturbances, noisy measurements, and parameter mismatch.  
Note that the policy $\pi(\hat{x},\hat{\theta})$ resulting from Problem~\eqref{eq:MPC} may in general be discontinuous. 
Hence, the following theorem also leverages ideas from~\cite{roset2008robustness} to provide robust stability guarantees despite the noisy measurements.  
\begin{theorem}
\label{thm:robust_stability_global}
Let Assumptions~\ref{ass:regularity}, \ref{ass:steady_regular}, and \ref{ass:stable} hold. 
Suppose $M,N,\omega$ are chosen such that $\alpha>0$ according to Theorem~\ref{thm:stable_artificial}. 
\change{There exist constants $\rho_{\mathrm{V}}\in(0,1)$, $c_{\mathrm{V}}>0$, such that f}or 
 any parameter estimates $\hat{\theta}_k\in\Theta$, parameters $\theta_k\in\Theta$, disturbances $w_k\in\mathbb{R}^{\nw}$, noise $v_k\in\mathbb{R}^{\nx}$, $k\in\mathbb{I}_{\geq 0}$, and any initial condition $x_0\in\mathbb{R}^{\nx}$, the closed-loop system~\eqref{eq:closed_loop_adaptive} satisfies
\begin{align}
\label{eq:Lyap_decrease_robust}
\JN^\star(\hat{x}_{k+1},\hat{\theta}_{k+1})\leq &\rho_{\mathrm{V}} \JN^\star(\hat{x}_k,\hat{\theta}_k)\\
&+c_{\mathrm{V}}\left(\|\tilde{w}_k+\tilde{x}_{1|k}\|^2+\|\hat{\theta}_{k+1}-\hat{\theta}_k\|^2\right)\nonumber
\end{align}
with 
$\tilde{w}_k,\tilde{x}_{1|k}$ corresponding to the prediction error~\eqref{eq:LMS_tilde}. 
\end{theorem}
The proof is deferred to Appendix~\ref{app:proof_global_robust}. 
Inequality~\eqref{eq:Lyap_decrease_robust} is similar to an input-to-state stability condition,   
with the optimal cost $\JN^\star$ exponentially contracting with $\rho_{\mathrm{V}}\in(0,1)$ 
and the increase in the cost bounded by the change in the parameter estimate $\hat{\theta}_{k+1}-\hat{\theta}_k$ and the prediction error $\tilde{w}_k+\tilde{x}_{1|k}$. 
This prediction error includes the effect of the measurement noise $v_k$, the disturbances $w_k$, and the parametric uncertainty $\hat{\theta}_k-\theta_k$. 

 \subsection{Online adaptation}
\label{sec:global_adaptive}
In this section, we analyse the combination of the certainty-equivalent MPC with the LMS parameter adaptation and show that it achieves Objective~\ref{objective_1}. 
The proofs are deferred to Appendix~\ref{app:proof_global_adaptive}.

First, the following lemma addresses the choice of the parameter gain $\Gamma\succ 0$ satisfying Assumption~\ref{ass:param_gain}. 
\begin{lemma}
\label{lemma:stable_bounded}
Let Assumptions~\ref{ass:regularity}\ref{item:ass_steady_state}, \ref{item:ass_Lipschitz_differentiable}, and \ref{ass:stable} hold. 
\change{Consider} known compact sets $\mathbb{X}_0$, $\mathbb{W}$, $\mathbb{V}$, $\Theta$ satisfying $w_k\in\mathbb{W}$, $u_k\in\mathbb{U}$, $v_k\in\mathbb{V}$, $\theta_k\in\Theta$, $k\in\mathbb{I}_{\geq 0}$,  and $x_0\in\mathbb{X}_0$. 
\change{Then, there exists a sufficiently small gain $\Gamma\succ 0$ satisfying Assumption~\ref{ass:param_gain}.}
\end{lemma}
\change{The gain $\Gamma$ can be computed given a bound on the regressor $\hat{\Phi}_k$. }
A corresponding \change{bound} for $\hat{\Phi}_k$ is provided in the proof of Lemma~\ref{lemma:stable_bounded} \change{in Appendix~\ref{app:proof_global_adaptive}}. 
This \change{bound} relies not only on open-loop stability (Asm.~\ref{ass:stable}), but also leverages \change{Lipschitz continuity} of the dynamics w.r.t. $\theta$ (Asm.~\ref{ass:regularity}\ref{item:ass_Lipschitz_differentiable}). 
Under a more relaxed \emph{local} Lipschitz bound, \change{boundedness} additionally requires a sufficiently small parameter variation $\Delta \theta$, see Appendix~\ref{app:bounded_TV} for details. 

The following theorem combines the LMS properties from Theorem~\ref{thm:LMS} with the inherent robustness properties in Theorem~\ref{thm:robust_stability_global} to provide closed-loop properties of the adaptive MPC.
\begin{theorem}
\label{thm:adaptive_global}
Let Assumptions~\ref{ass:regularity}, \ref{ass:linear_param}, \ref{ass:steady_regular}, and \ref{ass:stable} hold. 
Suppose we have known compact sets $\mathbb{X}_0$, $\mathbb{W}$, $\mathbb{V}$, $\Theta$ satisfying $w_k\in\mathbb{W}$, $v_k\in\mathbb{V}$, $\theta_k\in\Theta$, $k\in\mathbb{I}_{\geq 0}$, and $x_0\in\mathbb{X}_0$. 
Consider the gain $\Gamma\succ 0$ chosen according to Lemma~\ref{lemma:stable_bounded} and $N,M,\omega$ chosen such that $\alpha>0$ according to Theorem~\ref{thm:stable_artificial}. 
Then, the closed-loop system consisting of the dynamics~\eqref{eq:closed_loop_adaptive}, the LMS adaptation~\eqref{eq:LMS} and noisy measurements~\eqref{eq:noise_measurement} satisfies Objective~\ref{objective_1}, i.e., there exist \change{constants} $C_1,C_2>0$, such that Inequality~\eqref{eq:desired_stability} holds. 
\end{theorem}
\begin{corollary}
\label{cor:adaptive_global}
Let the conditions in Theorem~\ref{thm:adaptive_global} hold.
Suppose further that there are no disturbances, no measurement noise, and the parameters are constant, i.e., $v_k\in\mathbb{V}=\{0\}$, $w_k\in\mathbb{W}=\{0\}$, and $\Delta\theta_k=0$, $\forall k\in\mathbb{I}_{\geq 0 }$. 
Then, for any initial condition $x_0\in\mathbb{X}_0$, $\hat{\theta}_0\in\Theta$, the closed-loop system satisfies
\begin{align}
&\limsup_{\Time\rightarrow\infty}\sum_{k=0}^{\Time-1}\left[\|y_k-y_{\mathrm{rd},\theta}\|^2+\|x_k\|_{\mathbb{X}}^2\right]<\infty.
\label{eq:bounded_error_global}
\end{align}
\end{corollary}
Theorem~\ref{thm:adaptive_global} shows that the adaptive controller meets Objective~\ref{objective_1}, i.e., provides strong nominal performance guarantees and inherent robustness to arbitrary large noise, disturbances, and parameter variations.  
In case the model mismatch is only due to the epistemic uncertainty related to the unknown model parameters, Corollary~\ref{cor:adaptive_global} ensures the system converges to the optimal feasible setpoint. 

 
\section{Adaptive tracking for unstable systems}
\label{sec:regional}
In this section, we adjust the design and analysis from Section~\ref{sec:global} to relax the open-loop stability condition (Asm.~\ref{ass:stable}). 
Instead, we consider a locally stabilizing feedback $\kappa$. 
For the design, we construct a finite-tail cost $\Vfkappa$ with a rollout of the feedback $\kappa$, instead of an open-loop input (cf.~\eqref{eq:V_f}). 
Compared to the analysis in Section~\ref{sec:global}, there are two main differences:
(i) the certainty-equivalent MPC only ensures stability within a certain region of attraction; 
(ii) different sublevel arguments are leveraged to recursively ensure boundedness. 
Thus, we only establish the (weaker) regional guarantees (Objective~\ref{objective_2}), which require sufficiently small uncertainty and initial conditions in the region of attraction. 
As most results are analogous to Section~\ref{sec:global}, the following exposition only focuses on highlighting differences. 
The proofs can be found in Appendix~\ref{app:regional}.

\subsection{Stabilizing feedback} 
We consider a feedback of the form $\kappa:\mathbb{R}^{\nx}\times \Theta\times \mathbb{X}\times\mathbb{U}\rightarrow\mathbb{U}$, $\kappa(x,\theta,x_{\mathrm{s}},u_{\mathrm{s}})$, that 
satisfies $\kappa(x_{\mathrm{s}},\theta,x_{\mathrm{s}},u_{\mathrm{s}})=u_{\mathrm{s}}$ $\forall (\theta,x_{\mathrm{s}},u_{\mathrm{s}})$. 
For technical reasons (cf.~\cite{limon2018nonlinear,krupaModelPredictiveControl2024}), we introduce a compact set $\bar{\mathbb{U}}\subseteq\mathrm{int}(\mathbb{U})$ for the steady-state input constraints. 
We denote the feasible setpoints~\eqref{eq:steady_state_manifold} subject to the tighter constraint $\bar{\mathbb{U}}$ by $\bar{\mathbb{S}}(\theta)\subseteq\mathbb{S}(\theta)$. 
In this section, we consider Assumptions~\ref{ass:regularity} and \ref{ass:steady_regular} using these slightly modified sets $\bar{\mathbb{U}},\bar{\mathbb{S}}$. 
We denote by $\ell_\kappa(k,x,\theta,x_{\mathrm{s}},u_{\mathrm{s}})$, $k\in\mathbb{I}_{\geq 0}$, the stage cost $\ell$ evaluated along the nominal prediction under the feedback $\kappa$, i.e., 
\begin{subequations}
\label{eq:rollout_policy}
\begin{align}
x_{k+1}=&\fw(x_k,\theta,u_k,0),~x_0=x,\\ 
u_k=&\kappa(x_k,\theta,x_{\mathrm{s}},u_{\mathrm{s}}),~
\ell_k=\ell_\kappa(k,x,\theta,x_{\mathrm{s}},u_{\mathrm{s}}).
\end{align}
\end{subequations} 
We assume that the feedback $\kappa$ ensures stability in some \emph{local} region: 
\begin{align*}
\mathbb{S}_{\mathrm{loc}}=&\{(x,\theta,x_{\mathrm{s}},u_{\mathrm{s}})\\
&|~\exists y_{\mathrm{s}}:(x_{\mathrm{s}},u_{\mathrm{s}},y_{\mathrm{s}})\in\bar{\mathbb{S}}(\theta),~\|x-x_{\mathrm{s}}\|_Q^2\leq c_{\mathrm{loc}},~\theta\in\Theta\},
\end{align*}
where $c_{\mathrm{loc}}>0$ characterizes the local attraction. 
\begin{assumption}(Local exponentially stabilizing feedback)
\label{ass:stable_feedback}
There exist constants $\Cell\geq 1$, $\rho\in(0,1)$, such that for any 
$(x,\theta,x_{\mathrm{s}},u_{\mathrm{s}})\in\mathbb{S}_{\mathrm{loc}}$, we have 
\begin{align}
\label{eq:ell_exp_stable_feedback}
\ell_\kappa(k,x,\theta,x_{\mathrm{s}},u_{\mathrm{s}})\leq&\Cell\rho^k\|x-x_{\mathrm{s}}\|_Q^2,~ k\in\mathbb{I}_{\geq 0}.
\end{align}
Furthermore, $\kappa$ is Lipschitz continuous with constant $L_\kappa$. 
\end{assumption}
\begin{remark}\change{(Local stabilizing controller)}
\label{rk:feedback}
Analogous to Proposition~\ref{prop:ell_exp_stable}, Condition~\eqref{eq:ell_exp_stable_feedback} holds for an exponentially stabilizing and Lipschitz continuous feedback $\kappa$, see~\cite{magni2001stabilizing,kohler2021stability} for similar conditions. 
\change{Many existing designs require a common feedback and Lyapunov function to certify robust stability for all models $\theta\in\Theta$; a condition that can become infeasible if the parametric error is large (cf. numerical examples in Section~\ref{sec:num}). 
The proposed approach avoids this requirement by allowing the feedback to depend on the current parameter estimate.}
\change{A simple design satisfying Assumption~\ref{ass:stable_feedback} is given by the linear quadratic regulator (LQR) based on the model linearized around the steady-state $(x_{\mathrm{s}},u_{\mathrm{s}})$ with the model parameter $\theta$; assuming the linearization is  stabilizable uniformly in $(\theta,x_{\mathrm{s}},u_{\mathrm{s}})$.}  
In this case, the \emph{local} region of attraction $\mathbb{S}_{\mathrm{loc}}$ of $\kappa$ is determined by the input constraints $\bar{\mathbb{U}}\subseteq\mathrm{int}(\mathbb{U})$ and the local region of attraction of the LQR, assuming the dynamics are twice continuously differentiable~\cite[Lemma~1]{chen1998quasi}. 
\change{This function $\kappa$ is in general not available in an analytical form, but may, e.g., require online re-computation of the LQR for the current parameter estimate $\hat{\theta}_k$.}
\end{remark}

\subsection{Certainty-equivalent tracking MPC} 
For the finite-tail cost in~\eqref{eq:V_f}, we replace the open-loop simulation by the finite-horizon rollout of the policy $\kappa$: 
\begin{align}
\label{eq:V_f_kappa}
\Vfkappa(x,x_{\mathrm{s}},u_{\mathrm{s}},\theta):=\sum_{k=0}^{M-1}\ell_\kappa(k,x,\theta,x_{\mathrm{s}},u_{\mathrm{s}}), 
\end{align}
 similar to~\cite{magni2001stabilizing,kohler2021stability}.
Except for this modified terminal penalty $\Vfkappa$ and the tightened set of setpoints $\bar{\mathbb{S}}$, the MPC formulation and closed-loop system are defined analogous to Section~\ref{sec:global_certainty_equiv_design}. 
We denote the corresponding open-loop cost and optimal cost by $\JNkappa$ and $\JNkappa^\star$, respectively. 
We denote the minimizers with the superscript ${\star,\kappa}$ and the MPC policy by $\pi_\kappa(x,\theta)$.
The resulting closed-loop system is given by 
\begin{align}
\label{eq:closed_loop_adaptive_feedback}
x_{k+1}=\fw(x_k,\pi_\kappa(\hat{x}_k,\hat{\theta}_k),\theta_k,w_k),\quad k\in\mathbb{I}_{\geq 0}.
\end{align}
 
\subsection{Nominal stability analysis}
\label{sec:regional_nominal}
Due to the \emph{local} condition (Asm.~\ref{ass:stable_feedback}), we follow~\cite{soloperto2021nonlinear,kohler2021stability} and characterize the region of attraction as a sublevel set with a user-chosen constant $\bar{\mathcal{J}}>0$: 
\begin{align}
\label{eq:ROA}
\mathbb{X}_{\bar{\mathcal{J}}}:=\{(x,\theta)\in\mathbb{R}^{\nx}\times\Theta|~\JNkappa^\star(x,\theta)\leq \bar{\mathcal{J}}\}.
\end{align} 
We first study the \emph{nominal} closed loop
\begin{align}
\label{eq:nominal_closed_loop_feedback}
 x_{k+1}=\fw(x_k,\pi_{\kappa}(x_k,\theta),\theta,0).
\end{align}
The following theorem extends the nominal stability analysis from Theorem~\ref{thm:stable_artificial}, leveraging the region of attraction characterization from~\cite[Thm.~4.37]{kohler2021dynamic}.
\begin{theorem}
\label{thm:stable_artificial_feedback}
Let Assumptions~\ref{ass:regularity}\ref{item:ass_steady_state} and \ref{ass:stable_feedback} hold. 
For any $\bar{\mathcal{J}}>0$, and any $\omega\geq 1$, $M\in\mathbb{I}_{\geq 0}$, there exist a horizon $N\in\mathbb{I}_{\geq 1}$ \change{and a constant $\alpha>0$}, such 
that the following inequality holds for all $(x,\theta)\in \mathbb{X}_{\bar{\mathcal{J}}}$: 
\begin{align}
\label{eq:RDP_feedback}
&\JNkappa^\star(\fw(x,\pi_\kappa(x,\theta),\theta,0),\theta)-\JNkappa^\star(x,\theta)\nonumber\\
\leq&
-\alpha\cdot \ell(x,\pi_\kappa(x,\theta),x_{\mathrm{s},x,\theta}^{\star,\kappa},u_{\mathrm{s},x,\theta}^{\star,\kappa}). 
\end{align} 
\end{theorem} 
\begin{corollary}
\label{cor:stable_artificial_feedback}
 Assumptions~\ref{ass:regularity}\ref{item:ass_steady_state}, \ref{ass:steady_regular}\ref{item:ass_unique}, \ref{item:ass_convex}, and \ref{ass:stable_feedback} hold. 
Suppose $N,M,\omega,\bar{\mathcal{J}}$ are chosen according to Theorem~\ref{thm:stable_artificial_feedback}. 
For any $(x_0,\theta)\in\mathbb{X}_{\bar{\mathcal{J}}}$, the optimal steady-state $x_{\mathrm{rd},\theta}$ is exponentially stable for the nominal closed loop~\eqref{eq:nominal_closed_loop_feedback}.
\end{corollary} 
Compared to Theorem~\ref{thm:stable_artificial}, the length of the horizon $N$ needs to be increased by a factor that scales linearly in $\bar{\mathcal{J}}>0$, the desired sublevel set; see~\eqref{eq:local_horizon} in Appendix~\ref{app:regional}. 
More discussions regarding the role of the horizon $N$ for the region of attraction can be found in~\cite[Sec.~4.1]{kohler2021dynamic}.

\subsection{Inherent robustness}
The following theorem extends the robustness analysis from Theorem~\ref{thm:robust_stability_global}, assuming a sufficiently small bound on the prediction error.
\begin{theorem}
\label{thm:robust_stability_regional_feedback}
Let Assumptions~\ref{ass:regularity}, \ref{ass:steady_regular}, and \ref{ass:stable_feedback} hold. 
Suppose $M,N,\omega,\bar{\mathcal{J}}$ are chosen according to Theorem~\ref{thm:stable_artificial_feedback}. 
There exist \change{constants} $\rho_{\mathrm{V}}\in(0,1)$, $c_{\mathrm{V}}>0$, such that 
for any initial condition $(x_0,\hat{\theta}_0)\in\mathbb{X}_{\bar{\mathcal{J}}}$, 
 any parameter estimates $\hat{\theta}_k\in\Theta$, parameters
$\theta_k\in\Theta$, noise $v_k\in\mathbb{R}^{\nx}$, disturbances $w_k\in\mathbb{R}^{\nw}$ satisfying 
\begin{align}
\label{eq:robust_feedback_bound_mismatch}
\|\tilde{w}_k+\tilde{x}_{1|k}\|^2+\|\hat{\theta}_{k+1}-\hat{\theta}_k\|^2\leq \dfrac{\bar{\mathcal{J}}}{c_{\mathrm{V}}(1-\rho_{\mathrm{V}})}, ~\forall k\in\mathbb{I}_{\geq 0},
\end{align}
the closed-loop system~\eqref{eq:closed_loop_adaptive_feedback} satisfies $(\hat{x}_k,\hat{\theta}_k)\in\mathbb{X}_{\bar{\mathcal{J}}}$ and 
\begin{align}
\label{eq:Lyap_decrease_robust_feedback}
\JNkappa^\star(\hat{x}_{k+1},\hat{\theta}_{k+1})\leq &\rho_{\mathrm{V}} \JNkappa^\star(\hat{x}_k,\hat{\theta}_k)\\
&+c_{\mathrm{V}}\left(\|\tilde{w}_k+\tilde{x}_{1|k}\|^2+\|\hat{\theta}_{k+1}-\hat{\theta}_k\|^2\right).\nonumber
\end{align} 
\end{theorem}
\change{Inequality~\eqref{eq:robust_feedback_bound_mismatch} requires a sufficiently small uncertainty.
Owing to the local stability assumption (Assumption~\ref{ass:stable_feedback}), the admissible uncertainty cannot be increased arbitrarily. 
This robustness margin can be improved by increasing the rollout horizon $M$ and choosing the scaling $\omega$ close to one. A detailed analysis of how the design choices $N,M,\omega,\bar{\mathcal{J}}$ influence the constants can be found in \ifbool{arxiv}{Appendix~\ref{app:constants}}{\cite[App.~G]{adaptive_arxiv}}.}

\subsection{Online adaptation}
The following lemma addresses the suitable choice of the parameter gain $\Gamma\succ 0$.
\begin{lemma}
\label{lemma:bounded_feedback}
Let Assumptions~\ref{ass:regularity}\ref{item:ass_steady_state}--\ref{item:ass_Lipschitz_differentiable} hold. 
\change{Consider a} known compact set $\mathbb{V}$ and a constant $\bar{\mathcal{J}}>0$, such that $v_k\in\mathbb{V}$, $(\hat{x}_k,\hat{\theta}_k)\in\mathbb{X}_{\bar{\mathcal{J}}}$, $\forall k\in\mathbb{I}_{\geq 0}$. 
\change{Then, there exists a sufficiently small gain $\Gamma\succ 0$ satisfying Assumption~\ref{ass:param_gain}.}
\end{lemma} 
\change{The adaptation gain $\Gamma$ provided in Appendix~\ref{app:regional} depends on the sublevel set $\mathbb{X}_{\bar{\mathcal{J}}}$. 
A larger region of attraction (larger $\bar{\mathcal{J}}$) leads to more conservative choices of the adaptation gain $\Gamma$, which result in slower online adaptation.
Note that} 
Lemma~\ref{lemma:bounded_feedback} relies on positive invariance of $\mathbb{X}_{\bar{\mathcal{J}}}$, which needs to be recursively established. 
To study closed-loop properties, we consider the Lyapunov candidate function 
\begin{align}
\label{eq:V_kappa_combined}
V_\kappa(\hat{x}_k,\hat{\theta}_k,\theta_k):=\JNkappa(\hat{x}_k,\hat{\theta}_k)+\tilde{c}_{\mathrm{V}}\Vtheta(\theta_k-\hat{\theta}_k),
\end{align}
 with $\tilde{c}_{\mathrm{V}}=c_{\mathrm{V}}(1+\sigma_{\min}(\Gamma^{-1}))>0$. 
The following theorem shows that the adaptive controller addresses Objective~\ref{objective_2}. 
\begin{theorem}
\label{thm:adaptive_feedback}
Let Assumptions~\ref{ass:regularity}, \ref{ass:linear_param}, \ref{ass:steady_regular}, and \ref{ass:stable_feedback} hold. 
Consider $M,N,\omega,\bar{\mathcal{J}}$ according to Theorem~\ref{thm:stable_artificial_feedback} and $\Gamma$ according to Lemma~\ref{lemma:bounded_feedback}. 
Suppose $w_k\in\mathbb{W}$, $v_k\in\mathbb{V}$, $\theta_k\in\Theta$, $k\in\mathbb{I}_{\geq 0}$, and $(\hat{x}_0,\hat{\theta}_0)\in\mathbb{X}_{\bar{\mathcal{J}}}$ with sufficiently small sets $\mathbb{W}$, $\mathbb{V}$, $\Theta$. 
Then, the closed-loop system consisting of the dynamics~\eqref{eq:closed_loop_adaptive_feedback} with the LMS adaptation~\eqref{eq:LMS}, and noisy measurements~\eqref{eq:noise_measurement} satisfies Objective~\ref{objective_2}, i.e., there exist \change{constants} $C_1,C_2>0$, such that~\eqref{eq:desired_stability} holds.
\end{theorem}
Compared to Theorem~\ref{thm:adaptive_global}, 
this result uses additional arguments to ensure invariance of $\mathbb{X}_{\bar{\mathcal{J}}}$. 
We have a clear characterization of the region of attraction $\mathbb{X}_{\bar{\mathcal{J}}}$, e.g., we can ensure that all initial conditions close to a steady-state lie in the region of attraction~\cite{berberich2022linear_model}. 
In contrast, the restriction on the sufficiently small sets $\Theta,\mathbb{V},\mathbb{W}$ is only qualitative due to the utilization of conservative Lipschitz continuity bounds. 
The following corollary provides a more quantitative description of the permissible parameters $\Theta$ for a special case.
\begin{corollary}
\label{cor:adaptive_regional}
Let the conditions in Theorem~\ref{thm:adaptive_feedback} hold.
Suppose further that there are no disturbances, no measurement noise, and the parameters are constant, i.e., $v_k\in\mathbb{V}=\{0\}$, $w_k\in\mathbb{W}=\{0\}$, and $\theta_k=\theta$, $\forall k\in\mathbb{I}_{\geq 0 }$. 
Then, for any initial condition $V_\kappa(\hat{x}_0,\hat{\theta}_0,\theta_0)\leq \bar{\mathcal{J}}$, the closed-loop system satisfies~\eqref{eq:desired_stability} 
and 
\begin{align}
&\limsup_{\Time\rightarrow\infty}\sum_{k=0}^{\Time-1}\left[\|y_k-y_{\mathrm{rd},\theta}\|^2+\|x_k\|_{\mathbb{X}}^2\right]<\infty. 
\label{eq:bounded_error_regional}
\end{align}
\end{corollary}
Corollary~\ref{cor:adaptive_regional} ensures that the cumulative tracking error and the state constraint violation are \change{bounded} and thus converge to zero.
The considered bound on the Lyapunov function $V_\kappa$ highlights the interplay between the parameter estimation error $\Vtheta$ and the region of attraction characterized by $\bar{\mathcal{J}}$. 

\section{Discussion}
\label{sec:discussion}
We first discuss the properties of the proposed approach in contrast to existing approaches (Sec.~\ref{sec:discussion_compare}). 
Then, we show how the considered conditions simplify in the special case of linear dynamics (Sec.~\ref{sec:discussion_linear}).  
\subsection{Proposed design \& properties}
\label{sec:discussion_compare}
The proposed MPC design utilizes an (artificial) online computed setpoint $(x_{\mathrm{s}},u_{\mathrm{s}},y_{\mathrm{s}})$ with a quadratic offset cost~\cite{limon2018nonlinear,soloperto2021nonlinear,krupaModelPredictiveControl2024}. 
Thus, we can directly pass arbitrary target setpoints $y_{\mathrm{d}}$ to the controller. 
The prediction horizon $N$ is extended by applying the steady-state input $u_{\mathrm{s}}$ or a local feedback $\kappa$ over $M$ steps~\cite{magni2001stabilizing,kohler2021stability,bonassi2024nonlinear}. 
By choosing $M$ sufficiently long, we ensure stability without requiring an explicit design of a local CLF. 
\change{The state constraints along the prediction horizon are softened using a quadratic penalty in the stage cost $\ell$~\cite{zeilinger2014soft}, which allows for implementation of a simple certainty-equivalent approach while permitting state constraint violations.}  
Compared to a standard MPC implementation~\cite{rawlings2017model}, the complexity of the online computation and offline design are only moderately increased. 
Contrary to many competing robust and adaptive MPC approaches, we do not require any specific offline design and no re-design is required if the parameter $\theta$ or the reference $y_{\mathrm{d}}$ change during online operation. \change{Instead, the proposed approach uses a finite-horizon rollout of the nonlinear dynamics with the current parameter estimate $\hat{\theta}_k$ to implicitly approximate a parametrized CLF in a way that can be directly embedded in the MPC (Prop.~\ref{prop:finite_tail}). 
For unstable systems, this approach requires a locally stabilizing feedback $\kappa$ (Asm.~\ref{ass:stable_feedback}), which is easy to obtain for many nonlinear systems, e.g., using an LQR design online (Remark~\ref{rk:feedback}). 
Compared to existing robust and adaptive MPC approaches, this reduction in offline design complexity  comes at the price of allowing for state constraint violations.}  
In the nominal case (no disturbances, no measurement noise, constant unknown parameters),  Objective~\ref{objective_1}/\ref{objective_2} ensures that we converge to the optimal feasible setpoint, despite the potentially large parametric error (cf. Cor.~\ref{cor:adaptive_global}/\ref{cor:adaptive_regional}). 
In addition, Equation~\eqref{eq:desired_stability} shows that small noise, disturbances, and parameter variation do not have a significant effect on the closed loop. 
\ifbool{arxiv}{Below, we discuss  alternatives and possible modifications. 

\textit{Robustness considerations:}
Unmodelled dynamics and noisy measurements can yield instability in adaptive control implementations~\cite{rohrs1982robustness}. 
Although we establish robustness properties of the proposed approach,  the employment of additional robustification methods can be vital to reduce the tracking error~\cite{middleton1988design}. 
In particular, additional slowly time-varying disturbances can be compensated by additionally estimating additive parameters $\theta$~\cite{morari2012nonlinear,papadimitriou2020control}. 
High-frequency noise can be attenuated with suitably designed filters~\cite{gonccalves2016robust}. 
The least-squares adaptation could be robustified using a deadzone and normalization, though a rigorous analysis is left for future work.
 
\textit{CLF:}  
A natural alternative to the considered MPC formulation would be to compute a CLF offline and use it as a terminal cost~\cite{rawlings2017model}. 
However, computing a terminal cost that is valid for all states $x\in\mathbb{R}^{\nx}$, all setpoints $(x_{\mathrm{s}},u_{\mathrm{s}})$, and all parameters $\theta\in\Theta$ is even more challenging than computing the adaptive CLF required in classical nonlinear adaptive control~\cite{pomet1992adaptive}.  
Most adaptive MPC formulations assume a common Lyapunov function (independent of the parameters $\theta$)~\cite{tanaskovic2014adaptive,lorenzen2019robust,sinha2021adaptive}, in which case also robust methods without model adaptations can be utilized.  
The main benefit of the proposed approach is the ease of implementation without cumbersome offline design, which is achieved through an implicit approximation of a CLF using a finite-horizon rollout. 
Such finite-horizon rollouts are not only utilized in classical~\cite{magni2001stabilizing} and recent~\cite{kohler2021stability,bonassi2024nonlinear} MPC designs, but also play an important role in unifying MPC and reinforcement learning methods~\cite{bertsekas2022lessons,reiter2024ac4mpc}.

 \textit{State constraints and safety considerations:}  
In many applications, we might also be interested in obtaining pointwise in-time satisfaction of state constraints, i.e., $x_k\in\mathbb{X}$, $k\in\mathbb{I}_{\geq 0}$. 
In this case, a certainty-equivalent MPC may fail to ensure recursive feasibility. 
This problem can be circumvented by posing prior bounds on the set of possible parameters and disturbances and employing robust MPC designs~\cite{gonccalves2016robust,lopez2019adaptive,kohler2020robust,sasfi2022robust,degner2024nonlinear,sinha2021adaptive}. 
However, these methods can typically only be applied if the parametric uncertainty is quite small.}{}  
\subsection{Special case: Linear systems}
\label{sec:discussion_linear}
In the following, we discuss how the conditions simplify for linear systems 
\begin{align}
\label{eq:linear_system}
f(x,u,\theta,w)=&A_\theta x+B_\theta u+Ew+e_\theta, \\
h(x,u,\theta)=&C_\theta x+D_\theta u+f_\theta,\nonumber
\end{align}
with $A_\theta,B_\theta,e_\theta,C_\theta,D_\theta,f_\theta$ affine in $\theta$ (cf.~\cite{lorenzen2019robust}). 
We consider polytopic input constraints $\mathbb{U}$ and suppose some polytopic parameter set $\Theta$ with $\theta_k\in\Theta$ is known, 

\textit{Regularity (Asm.~\ref{ass:regularity}):} 
Condition~\ref{item:ass_parameter_set} holds, e.g., if the origin is a feasible equilibrium, i.e., $e_\theta=0$, $0\in\mathbb{U},0\in\mathbb{X}$.
Condition~\ref{item:ass_steady_state} holds trivially. 
Although linear systems are trivially Lipschitz continuous w.r.t. $(x,u,w)$, they are in fact not globally Lipschitz continuous w.r.t. $\theta$, i.e., Condition~\ref{item:ass_Lipschitz_differentiable} does not hold directly. 
However, a slightly weaker local Lipschitz condition holds trivially. 
In Appendix~\ref{app:bounded_TV}, we show how our analysis naturally generalizes to such a local Lipschitz condition, with the main caveat that the parameter variation $\Delta \theta$ needs to be sufficiently small. 
This is not surprising, given that we analyse a linear time-varying systems. 

\textit{Linear parametrization:} Assumption~\ref{ass:linear_param} regarding linear parametrization is also easy to satisfy for linear systems. 

\textit{Steady-states (Asm.~\ref{ass:steady_regular}):} The uniqueness condition (Asm.~\ref{ass:steady_regular}\ref{item:ass_unique}) reduces to a standard rank condition (cf.~\cite[Rk.~1]{limon2018nonlinear}, \cite[Lemma 1.8]{rawlings2017model}) that holds for generic linear systems. 
Convexity (Asm.~\ref{ass:steady_regular}\ref{item:ass_convex}) holds trivially. 
Assumption~\ref{ass:steady_regular}\ref{item:ass_regularity_steady_state_1} follows from linear independence constraint qualification and second order sufficient conditions~\cite[Thm.~4.1]{robinson1980strongly}. 
Assumption~\ref{ass:steady_regular}\ref{item:ass_regularity_steady_state_2} requires a feasible target $y^{\mathrm{d}}$. 

\textit{Stability/stabilizability: }
Assumption~\ref{ass:stable} reduces to $A_\theta$ Schur stable for all $\theta\in\Theta$. 
Assumption~\ref{ass:stable_feedback} holds if $(A_\theta,B_\theta)$ is stabilizable by, e.g., using the LQR $\kappa=u_{\mathrm{s}}+K_{\mathrm{LQR}}(\theta)\cdot (x-x_{\mathrm{s}})$. The restriction to a local condition $c_{\mathrm{loc}}>0$ ensures $\kappa(\cdot)\in\mathbb{U}$ given $u_{\mathrm{s}}\in\bar{\mathbb{U}}\subseteq\mathrm{int}(\mathbb{U})$. 

\textit{Implementation:} 
The MPC optimization problem~\eqref{eq:MPC} 
is a convex linearly constrained quadratic program, which can be efficiently solved. 
In the absence of state constraints, i.e., $\mathbb{X}=\mathbb{R}^n$, the terminal costs $\Vf/\Vfkappa$~\eqref{eq:V_f}/\eqref{eq:V_f_kappa} can be implemented as a quadratic function $\|x-x_s\|_{P(\hat{\theta}_k)}^2$, where $P(\hat{\theta}_k)$ is computed online.

Overall, most conditions are satisfied by generic linear systems. 
The restriction to square systems ($\ny=\nu$) to ensure uniqueness can be naturally relaxed~\cite{soloperto2021nonlinear}. 
In addition, as detailed in Appendix~\ref{app:bounded_TV}, the parameter variation $\Delta \theta$ needs to be sufficiently small, which holds trivially if we consider a linear time-invariant system.
A horizon $N,M$ and weighting $\omega$ ensuring stability (Thm.~\ref{thm:stable_artificial}/\ref{thm:stable_artificial_feedback}) can be efficiently determined based on a generalized eigenvalue~\cite{koehler2021LP,kohler2021dynamic}. 

\section{Numerical examples} 
\label{sec:num} 
First, we consider a linear open-loop stable chain of mass-spring-dampers (Sec.~\ref{sec:num_lin}).
Then, we consider a nonlinear unstable quadrotor navigating obstacles (Sec.~\ref{sec:num_drone}). 
Optimization problems are solved using Matlab with quadprog or IPOPT with CasADi~\cite{andersson2019casadi}. 
Computation times were measured on a laptop with an Intel Core i7-10750H processor (2.6 GHz), 32 GB RAM, running Windows 11, without compilation. 
All constants are provided in the open-source code: 
{\url{https://github.com/KohlerJohannes/Adaptive}} 

\subsection{Linear stable mass-spring-damper chain}
\label{sec:num_lin}
\begin{figure}[t]
\includegraphics[width=0.45\textwidth]{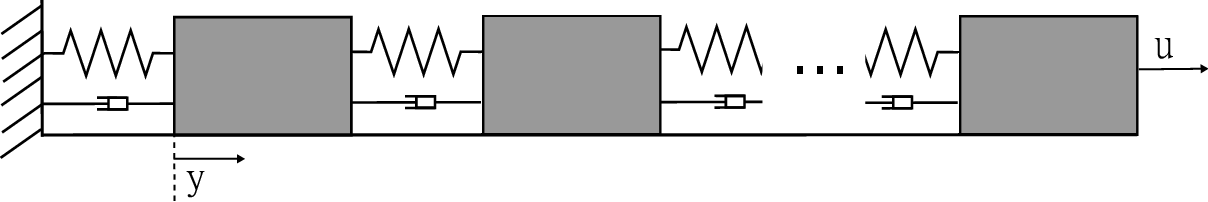}
\vspace{-2mm}
\caption{Chain of masses connected by springs and dampers.}
\label{fig:illustration}
\end{figure}
We first consider a linear system adapted from~\cite{koehler2021LP}:
We have $10$ series-connected mass-spring-dampers ($\nx=20$) with one actuator ($\nu=1$) and the output $y$ corresponding to the position of the first mass ($\ny=1$), see Figure~\ref{fig:illustration}.
We have bounded inputs $\mathbb{U}=[-25,25]$ and a soft state constraint $y\leq 0.7$. 
The system is open-loop stable, under-actuated, and non-minimum-phase. 

\textit{Adaptive MPC design:} 
We have given an initial parameter estimate and the true (unknown) physical parameters may differ by up to $\pm 50\%$. 
The linear dynamics are exactly discretized with a sampling time of $500$~[ms] and we treat all entries in the resulting discrete-time system as uncertain, yielding $\ntheta=20\cdot 21=420$ uncertain parameters. 
Despite this large uncertainty, it is easy to verify that the system is open-loop stable for all $\theta\in\Theta$.
The MPC implementation uses a horizon of $M=22$, $N=6$, and a scaling $\omega=5$, which satisfies the stability conditions in Theorem~\ref{thm:stable_artificial}. 
The parameter gain $\Gamma$ is chosen according to~\ifbool{arxiv}{\eqref{eq:compute_gamma}}{Lemma~\ref{lemma:stable_bounded}} 
using $\mathbb{U}$ and a box $\mathcal{R}\subseteq\mathbb{R}^{\nx}$ on the magnitude of possible states.

\textit{Numerical simulations:} 
The true parameters are constant and deviate exactly by $\pm 50\%$ from the initial estimate $\hat{\theta}_0$. Disturbances $w_k$ and noise $v_k$ are uniformly distributed. 
A piece-wise constant reference $y_{\mathrm{rd}}$ is chosen, which yields optimal steady-states that are sometimes on the boundary of the constraints $\mathbb{X}$. 
We compare the following designs: \\
(i) \emph{Proposed}: the proposed adaptive MPC;\\
(ii) \emph{No-term}: \textit{Proposed} without terminal cost ($M=0$);\\
(iii) \emph{No-adapt}: \textit{Proposed} without adaptation ($\Gamma=0$).\\
Numerical results for the remaining methods can be seen in Figure~\ref{fig:MassSPring_results}. A quantitative comparison is found in Table~\ref{tab:num_lin}. 
\change{We also considered the robust adaptive MPC from~\cite{lorenzen2019robust}. However, this approach requires an offline design of a Lyapunov function that is valid for all $\theta\in\Theta$, which became infeasible with uncertainty of only $\pm20\%$.}

\begin{table}[ht]
\centering
\begin{tabular}{c|c|c|c|c|c}
 &Proposed& No-term &No-adapt \\\hline
Track& $1.00$&$~2.51$&$~~1.38$ \\
Constr.& $1.00$&$18.97$&$698.05$
\end{tabular}
 \caption{Cumulative tracking error $\sum_k\|y_k-y_{\mathrm{rd},\theta}\|^2$ and constraint violation $\sum_k\|x_k\|_{\mathbb{X}}^2$ for a simulation of 200~[s], normalized to the performance of the proposed method.}
 \label{tab:num_lin}
\end{table} 
The proposed adaptive MPC successfully converges to the optimal setpoints. In the first few steps, the method shows constraint violations and overshoot. However, as the model is adapted during closed-loop operation, overshoot and constraint violations become negligible. 
The approach without the finite-horizon rollout terminal cost $\Vf$ (\emph{No-term}) also improves over time and successfully converges to the targets. 
However, the controller shows a rather sluggish response and takes visibly longer to get close to the target. 
The approach without adaptation (\emph{No-adapt}) yields large and persistent constraint violations and tracking errors, i.e., fails to achieve Objective~\ref{objective_1}. 
Computation times of the proposed MPC are $60\pm 5~[ms]$, which is almost one order of magnitude faster than required for real-time.
\begin{figure}[t]
\includegraphics[width=0.425\textwidth]{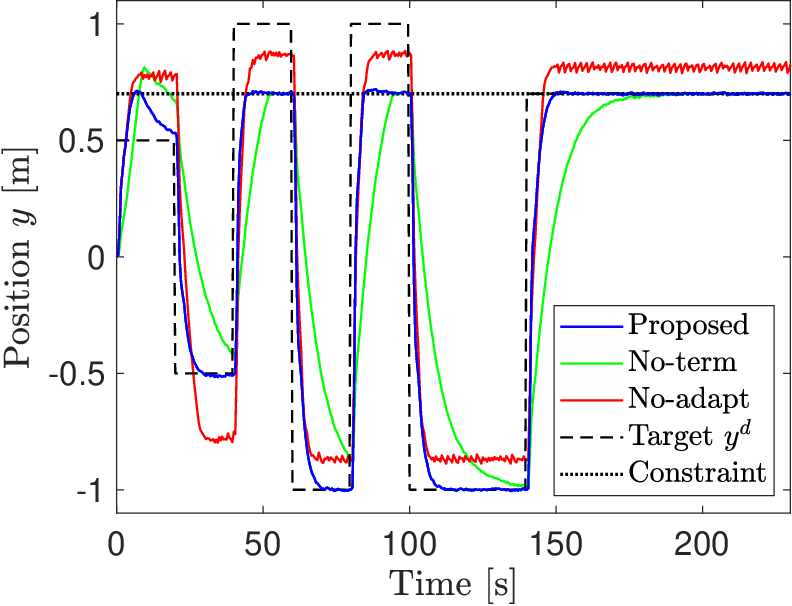}
\caption{Simulation result for stable mass-spring-dampers. }
\label{fig:MassSPring_results}
\end{figure}

\subsection{Nonlinear unstable quadrotor navigating obstacles}
\label{sec:num_drone}
We consider a planar quadrotor problem adapted from~\cite{sasfi2022robust}:
\begin{align*}
\begin{pmatrix}
\dot{p}_1\\\dot{p}_2\\\dot{\phi}\\\dot{v}_1\\\dot{v}_2\\\ddot{\phi}
\end{pmatrix}=
\begin{pmatrix}
v_1 \cos(\phi)-v_2\sin(\phi)\\
v_1\sin(\phi)+v_2\cos(\phi)\\
\dot{\phi}\\
v_2\dot{\phi}- \mathrm{g}\cdot \sin(\phi)+\cos(\phi)w\\
-v_1\dot{\phi}-\mathrm{g}\cdot\cos(\phi)+\theta_1(u_1+u_2)-\sin(\phi)w\\
\theta_2 (u_1-u_2)
\end{pmatrix}
\end{align*}
with positions $p_1,p_2$, velocities $v_1,v_2$, angle and angular velocities $\phi,\dot{\phi}$, thrust $u_1,u_2$, wind disturbance $w$, and dimensions $\nx=6$, $\nu=2$, $\nw=1$, $\ntheta=2$. 
The uncertain parameters correspond to uncertainty in mass and geometry (inertia and distance of propellers). 
The system is nonlinear and unstable. 
This example focuses on demonstrating performance under realistic model errors, and we do not explicitly verify whether the sufficient conditions on the prediction horizons and uncertainty levels in Theorem~\ref{thm:adaptive_feedback} are satisfied.  

We have compact input constraints $\mathbb{U}=[-1,4]^2$. 
In addition to desired bounds on velocity and angles, the state constraint $\mathbb{X}$ includes obstacle avoidance constraints. 
We consider the problem of reaching a distant setpoint $y^{\mathrm{d}}$ while subject to large disturbances, noise, and parametric uncertainty. 
The initial parameter estimate $\hat{\theta}_0$ deviates by a factor $2$ from the unknown true parameters.
The model is discretized using Euler with sampling time $25~[ms]$. 
We choose the feedback $\kappa$ (Asm.~\ref{ass:stable_feedback}) based on the LQR.
The LMS gain $\Gamma$ is chosen according to~\ifbool{arxiv}{\eqref{eq:compute_gamma}}{Lemma~\ref{lemma:stable_bounded}}. 
The horizon is chosen as $N=5$, $M=10$. 
We considered again the following designs: \\
(i)\emph{Proposed}: the proposed adaptive MPC;\\
(ii) \emph{No-term}: \textit{Proposed} without terminal cost ($M=0$);\\
(iii) \emph{No-adapt}: \textit{Proposed} without adaptation ($\Gamma=0$).\\
Numerical results for the remaining methods can be seen in Figure~\ref{fig:drone_results}. 
\change{We also considered the robust adaptive MPC from~\cite{sasfi2022robust}, however, this approach became infeasible for $6\%$ parametric uncertainty, which was well below the considered range of over $50\%$ uncertainty.}
The \emph{No-adapt} implementation was unstable and diverged due to the significant model error. 
The \emph{Proposed} implementation successfully navigates the obstacles and reaches the target in few seconds. 
The \emph{No-term} implementation also avoids the obstacle, but its response is rather slow, taking about five times longer to reach within one centimetre of the target. 
Even after convergence, the \emph{No-term} implementation shows significant oscillations around the target ($>10$~[cm]), while the \emph{Proposed} approach manages to hover close to the target ($\approx 4$~[mm]). 
Computation times for the proposed MPC are $14.1\pm 5.0~[ms]$, which is significantly below the real-time requirement. 
\begin{figure}[t]
\includegraphics[width=0.425\textwidth]{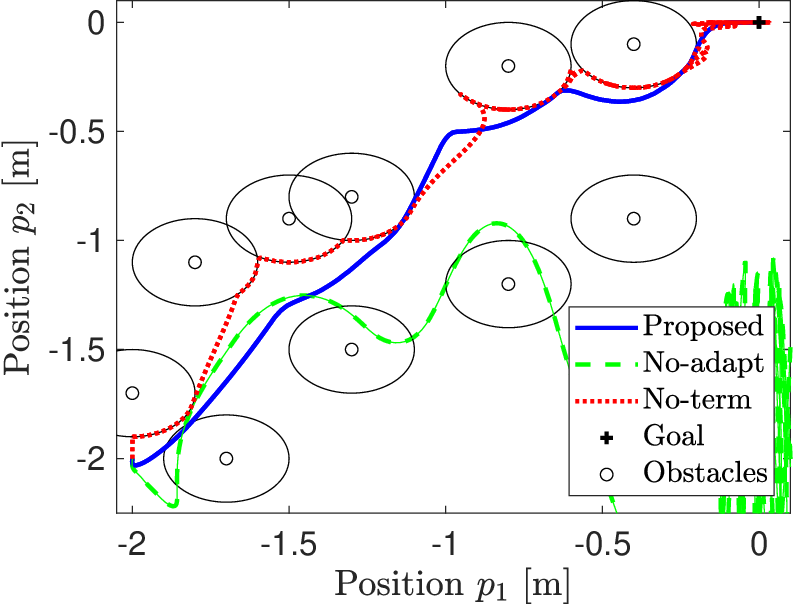}
\caption{Simulation result for unstable nonlinear quadrotor. }
\label{fig:drone_results}
\end{figure}

\section{Conclusion} 
\label{sec:sum}
We presented an adaptive control design for nonlinear systems by combining a certainty-equivalent tracking MPC formulation with online parameter adaptation.
The proposed methodology provides strong inherent robustness properties and can be directly applied to nonlinear dynamics, input constraints, (soft) state constraints, and setpoint tracking.  
Two numerical examples illustrate the practicality and benefits of the proposed method.
Future work is focused on assessing reliability in experimental settings.  
\begin{ack} 
The author thanks colleagues at ETH Zurich for helpful suggestions, in particular Melanie Zeilinger.   
\end{ack}
\bibliographystyle{plain}        
\bibliography{Literature} 
\appendix
\allowdisplaybreaks
\ifbool{arxiv}{\newpage}{}
\section{Proof of Section~\ref{sec:LMS}}
\label{app:proof_LMS}
\textit{Proof of Theorem~\ref{thm:LMS}:}
Assumptions~\ref{ass:regularity}\ref{item:ass_parameter_set} and \ref{ass:regularity}\ref{item:ass_Lipschitz_differentiable} together with non-expansiveness of the projection operator ensure 
$\Vtheta(\hat{\theta}_{k+1}-\theta_{k+1})\leq \Vtheta(\tilde{\theta}_{k+1}-\theta_{k+1})$. 
Furthermore, the update~\eqref{eq:LMS_1} satisfies 
$\tilde{\theta}_{k+1}-\hat{\theta}_k=\Gamma\hat{\Phi}_k^\top(\tilde{x}_{1|k}+\tilde{w}_k)$
using definitions~\eqref{eq:LMS_tilde_x}--\eqref{eq:LMS_tilde_w}.
Thus, it holds that 
\begin{align}
\label{eq:LMS_Lyap_proof_1}
&\Vtheta(\hat{\theta}_{k+1}-\theta_k)-\Vtheta(\hat{\theta}_k-\theta_k)\\
\leq& \Vtheta(\tilde{\theta}_{k+1}-\theta_k)-\Vtheta(\hat{\theta}_k-\theta_k)\nonumber\\
=&\|\tilde{\theta}_{k+1}-\hat{\theta}_k\|_{\Gamma^{-1}}^2+2(\tilde{\theta}_{k+1}-\hat{\theta}_k)^\top \Gamma^{-1}(\hat{\theta}_k-\theta_k)\nonumber\\
\stackrel{\eqref{eq:LMS_1}}{=}&\|\hat{\Phi}_k^\top(\tilde{x}_{1|k}+\tilde{w}_k)\|_{\Gamma}^2+2(\tilde{x}_{1|k}+\tilde{w}_k)^\top \underbrace{\hat{\Phi}_k(\hat{\theta}_k-\theta_k)}_{=-\tilde{x}_{1|k}}\nonumber\\
\stackrel{\mathrm{Asm.}~\ref{ass:param_gain}}\leq &\|\tilde{x}_{1|k}+\tilde{w}_k\|^2-2\|\tilde{x}_{1|k}\|^2-2\tilde{x}_{1|k}^\top \tilde{w}_k\nonumber\\
=&-\|\tilde{x}_{1|k}\|^2+\|\tilde{w}_k\|^2.\nonumber
\end{align}
To account for the change in the parameters $\Delta \theta_k=\theta_{k+1}-\theta_k$, note that the quadratic function $\Vtheta$ satisfies 
\begin{align}
\label{eq:LMS_Lyap_proof_2}
&\Vtheta(\hat{\theta}_{k+1}-\theta_{k+1})-\Vtheta(\hat{\theta}_{k+1}-\theta_{k})\\
=& \|\Delta \theta_k\|_{\Gamma^{-1}}^2+\Delta\theta_k^\top\Gamma^{-1}(\theta_k-\hat{\theta}_k)\nonumber\\
=&\Delta\theta_k^\top\Gamma^{-1}(\theta_{k+1}-\hat{\theta}_k)- \|\Delta \theta_k\|_{\Gamma^{-1}}^2
\leq  \sqrt{\Vtheta(\Delta \theta_k)}c_\theta,\nonumber
\end{align}
with $c_\theta$ from~\eqref{eq:LMS_c_theta} and $\hat{\theta}_k,\hat{\theta}_{k+1}\in\Theta$ due to~\eqref{eq:LMS_2}. 
Combining~\eqref{eq:LMS_Lyap_proof_1} and \eqref{eq:LMS_Lyap_proof_2} yields \eqref{eq:LMS_Lyap}. 
The bound~\eqref{eq:LMS_step} follows with 
\begin{align*}
&\|\hat{\theta}_{k+1}-\hat{\theta}_k\|_{\Gamma^{-1}}^2\stackrel{\eqref{eq:LMS_2}}{\leq} \|\tilde{\theta}_{k+1}-\hat{\theta}_k\|_{\Gamma^{-1}}^2\\
\stackrel{\eqref{eq:LMS_1}}{=}& \|\hat{\Phi}_k^\top(\tilde{x}_{1|k}+\tilde{w}_k)\|_{\Gamma}^2\stackrel{\mathrm{Asm.}~\ref{ass:param_gain}}{\leq} \|\tilde{x}_{1|k}+\tilde{w}_k\|^2.
\end{align*}
Lastly, the bound~\eqref{eq:LMS_w_bound} follows from~\eqref{eq:noise_measurement}, \eqref{eq:LMS_tilde_w} and Lipschitz continuity (Asm.~\ref{ass:regularity}\ref{item:ass_Lipschitz_differentiable}).

\ifbool{arxiv}{\newpage}{}

\section{Proofs of Section~\ref{sec:global_analysis}}
\label{app:proof_global_analysis}
\textit{Proof of Prop.~\ref{prop:ell_exp_stable}:}
\ifbool{arxiv}{We have $x_{\mathrm{s}}\in\mathbb{X}$ from~\eqref{eq:steady_state_manifold} and thus 
\begin{align*}
D_i x-d_i=D_i(x-x_{\mathrm{s}})+D_ix_{\mathrm{s}}-d_i\leq D_i(x-x_{\mathrm{s}}),
\end{align*}
for any $x\in\mathbb{R}^{\nx}$, $i\in\mathbb{I}_{[1,r]}$. 
Correspondingly, we 
\begin{align*}
\max\{D_i x-d_i,0\}^2\leq \max\{D_i (x-x_{\mathrm{s}}),0\}^2
\leq (D_i(x-x_{\mathrm{s}}))^2.
\end{align*}
Thus, for any $x\in\mathbb{R}^{\nx}$, the stage cost satisfies}{%
For any $x\in\mathbb{R}^{\nx}$ and $x_{\mathrm{s}}\in\mathbb{X}$, the stage cost satisfies} 
\ifbool{arxiv}{\begin{align*}
&\ell(x,u_{\mathrm{s}},x_{\mathrm{s}},u_{\mathrm{s}})
=\|x-x_{\mathrm{s}}\|_Q^2+\sum_{i=1}^r q_{\xi,i}\max\{D_i x-d_i,0\}^2\nonumber\\
\leq &\|x-x_{\mathrm{s}}\|_Q^2+\sum_{i=1}^r q_{\xi,i}\|D_i (x-x_{\mathrm{s}})\|^2\nonumber\\
\leq &(\sigma_{\max}(Q)+\underbrace{\sigma_{\max}\left(\sum_{i=1}^r q_{\xi,i} D_i^\top D_i\right)}_{=\sigma_{\xi}})\|x-x_{\mathrm{s}}\|^2.\nonumber 
\end{align*}}{%
\begin{align*}
&\ell(x,u_{\mathrm{s}},x_{\mathrm{s}},u_{\mathrm{s}})
=\|x-x_{\mathrm{s}}\|_Q^2+\sum_{i=1}^r q_{\xi,i}\max\{D_i x-d_i,0\}^2\nonumber\\
\leq &(\sigma_{\max}(Q)+\underbrace{\sigma_{\max}\left(\sum_{i=1}^r q_{\xi,i} D_i^\top D_i\right)}_{=\sigma_{\xi}})\|x-x_{\mathrm{s}}\|^2.\nonumber 
\end{align*}}
Assumption~\ref{ass:stable} implies
\ifbool{arxiv}{\begin{align*}
&\ell(x_{\mathbf{u}^k_{\mathrm{s}}}(k,x,\theta),u_{\mathrm{s}},x_{\mathrm{s}},u_{\mathrm{s}})\\
\leq &(\sigma_{\max}(Q)+\sigma_{\xi})\|x_{\mathbf{u}^k_{\mathrm{s}}}(k,x,\theta)-x_{\mathrm{s}}\|^2\\
\stackrel{\eqref{eq:exp_stable}}{\leq} &\Crho(\sigma_{\max}(Q)+\sigma_{\xi})\rho^k\|x-x_{\mathrm{s}}\|^2\\
\leq& 
\underbrace{\Crho\dfrac{\sigma_{\max}(Q)+\sigma_{\xi}}{\sigma_{\min}(Q)}}_{=:\Cell}\rho^k\|x-x_{\mathrm{s}}\|_Q^2.
\end{align*}}{%
\begin{align*}
&\ell(x_{\mathbf{u}^k_{\mathrm{s}}}(k,x,\theta),u_{\mathrm{s}},x_{\mathrm{s}},u_{\mathrm{s}})\\
\stackrel{\eqref{eq:exp_stable}}{\leq}& 
\underbrace{\Crho\dfrac{\sigma_{\max}(Q)+\sigma_{\xi}}{\sigma_{\min}(Q)}}_{=:\Cell}~\rho^k\|x-x_{\mathrm{s}}\|_Q^2.
\end{align*}
}
\textit{Proof of Prop.~\ref{prop:cost_controllable}:}
Considering the candidate solution $\tilde{\mathbf{u}}=\mathbf{u}_{\mathrm{s}}^N\in\mathbb{U}^N$ and Proposition~\ref{prop:ell_exp_stable}, we get 
\begin{align*}
&\min_{\mathbf{u}\in\mathbb{U}^N}\JN(x,\theta,\mathbf{u},x_{\mathrm{s}},u_{\mathrm{s}})
\leq \JN(x,\theta,\tilde{\mathbf{u}},x_{\mathrm{s}},u_{\mathrm{s}})\\
\stackrel{\eqref{eq:J_N},\eqref{eq:ell_exp_stable}}{\leq}& 
\sum_{k=0}^{N-1}\Cell\rho^k\|x-x_{\mathrm{s}}\|_Q^2
+\omega\sum_{k=N}^{N+M-1}\Cell\rho^k\|x-x_{\mathrm{s}}\|_Q^2\\
= &\underbrace{\Cell\left[\dfrac{1-\rho^N}{1-\rho}+\omega\rho^N\dfrac{1-\rho^M}{1-\rho}\right]}_{=:\gamma_{N}}\|x-x_{\mathrm{s}}\|_Q^2. 
\end{align*}  
\textit{Proof of Prop.~\ref{prop:finite_tail}:} 
\ifbool{arxiv}{%
Abbreviate $\ell_k=\ell(x_{\mathbf{u}^k_{\mathrm{s}}}(k,x,\theta),u_{\mathrm{s}},x_{\mathrm{s}},u_{\mathrm{s}})$, $k\in\mathbb{I}_{[0,M]}$. \\
\textbf{Part I: }
We prove Inequality~\eqref{eq:approx_CLF} with the candidate input $u=u_{\mathrm{s}}\in\mathbb{U}$. 
Given the definition of $\Vf$~\eqref{eq:V_f}, and the candidate input, Condition~\eqref{eq:approx_CLF} reduces to
\begin{align*}
\ell_0+\omega\sum_{k=1}^M\ell_k\leq (1+\epsilon_{\mathrm{f}})\omega\sum_{k=0}^{M-1}\ell_k,
\end{align*}
or equivalently
\begin{align*}
\ell_M+\dfrac{1-\omega}{\omega}\ell_0\leq \epsilon_{\mathrm{f}} \sum_{k=0}^{M-1}\ell_k.
\end{align*}
For the special case $\omega=1$, this condition has been analysed in~\cite[Prop.~4]{koehler2021LP}. 
Following the same arguments, a constant $\epsilon_{\mathrm{f}}$ satisfying~\eqref{eq:approx_CLF} can be computed using the following LP:
\begin{subequations}
\label{eq:LP}
\begin{align}
\label{eq:LP_cost}
\epsilon_{\mathrm{f}}:=&\max_{\tilde{\ell}_k}\tilde{\ell}_M+\dfrac{1-\omega}{\omega}\tilde{\ell}_0\\
\label{eq:LP_eq}
\text{s.t. }&\sum_{k=0}^{M-1}\tilde{\ell}_k=1,\\
\label{eq:LP_ineq}
&\tilde{\ell}_M\leq \Cell\rho^{M-k}\tilde{\ell}_k,\quad k\in\mathbb{I}_{[0,M-1]},\\
\label{eq:LP_non_neg}
&\tilde{\ell}_M\geq 0. 
\end{align}
\end{subequations}
We denote a maximizer by $\tilde{\ell}_k^\star$. 
Note that Inequality~\eqref{eq:LP_ineq} corresponds to Inequality~\eqref{eq:ell_exp_stable} from Proposition~\ref{prop:ell_exp_stable}.\\
\textbf{Part II: }Next, we show that the maximizer satisfies the constraints~\eqref{eq:LP_ineq} with equality for $k=0$. 
For contradiction, suppose Inequality~\eqref{eq:LP_ineq} for $k=0$ is not active, i.e., $\tilde{\ell}^\star_M<\Cell\rho^M\tilde{\ell}^\star_0$.
Consider the candidate $\tilde{\ell}_M=\tilde{\ell}_M^\star+\delta$, $\tilde{\ell}_k=\tilde{\ell}_k^\star+\frac{\delta}{\Cell\rho^{M-k}}$, $k\in\mathbb{I}_{[1,M-1]}$ with some $\delta>0$, which satisfies~\eqref{eq:LP_ineq} for $k\in\mathbb{I}_{[1,M-1]}$.
Equation~\eqref{eq:LP_eq} holds by choosing $\tilde{\ell}_0=\tilde{\ell}_0^\star-\frac{\delta}{\Cell}\sum_{k=1}^{M-1}\rho^{k-M}=
\tilde{\ell}_0^\star-\frac{\delta}{\Cell\rho^{M-1}}\frac{1-\rho^{M-1}}{1-\rho}$. 
Inequality~\eqref{eq:LP_ineq} for $k=0$ remains valid by choosing $\delta>0$ sufficiently small. 
Considering the cost of this candidate, we have
\begin{align*}
&\tilde{\ell}_M-\tilde{\ell}_M^\star+\dfrac{1-\omega}{\omega}(\tilde{\ell}_0-\tilde{\ell}_0^\star)\\
=&\delta\left[
1-\dfrac{1-\omega}{\omega}\dfrac{1}{\Cell\rho^{M-1}}\dfrac{1-\rho^{M-1}}{1-\rho}
\right]\geq \delta>0,
\end{align*}
where the last inequality used $\omega\geq 1$. 
This result contradicts optimality. 
Thus, Inequality~\eqref{eq:LP_ineq} for $k=0$ holds with equality and the LP~\eqref{eq:LP} reduces to 
\begin{subequations}
\label{eq:LP_2}
\begin{align}
\label{eq:LP_2_cost}
\epsilon_{\mathrm{f}}:=&\max_{\tilde{\ell}_k}\tilde{\ell}_M\left[1+\dfrac{1-\omega}{\omega\Cell\rho^M}\right]\\
\label{eq:LP_2_eq}
\text{s.t. }&\tilde{\ell}_M/(\Cell\rho^M)+\sum_{k=1}^{M-1}\tilde{\ell}_k=1,\\
\label{eq:LP_2_ineq}
&\tilde{\ell}_M\leq \Cell\rho^{M-k}\tilde{\ell}_k,\quad k\in\mathbb{I}_{[1,M-1]},\\
\label{eq:LP_2_non_neg}
&\tilde{\ell}_M\geq 0. 
\end{align}
\end{subequations}
\textbf{Part III: }Next, we use a case distinction to show that the solution to~\eqref{eq:LP_2} is given by~\eqref{eq:epsilon_f}.\\
\textit{Case a: }Suppose that $1+\frac{1-\omega}{\omega\Cell\rho^M}>0$, i.e., the objective is to maximize $\tilde{\ell}_M$. 
For contradiction, suppose that $\tilde{\ell}^\star_M<\Cell\rho^{M-k'}\tilde{\ell}^\star_{k'}$ for some $k'\in\mathbb{I}_{[1,M-1]}$.
Consider the candidate $\tilde{\ell}_M=\tilde{\ell}_M^\star+\delta$, $\tilde{\ell}_j=\tilde{\ell}_j^\star+\frac{\delta}{\Cell\rho^{M-j}}$, $j\in\mathbb{I}_{[1,k'-1]\cup[k'+1,M-1]}$ with some $\delta>0$ and $\tilde{\ell}_{k'}$ such that Equality~\eqref{eq:LP_2_eq} holds.
Inequalities~\eqref{eq:LP_2_ineq}, $k\in\mathbb{I}_{[1,k'-1]\cup[k'+1,M-1]}$ hold by definition and Inequality~\eqref{eq:LP_2_ineq} for $k=k'$ holds by choosing $\delta>0$ sufficiently small. 
Hence, we have a feasible candidate with a strictly better cost ($\tilde{\ell}_M>\tilde{\ell}_M^\star$), contradicting optimality.
Given that~\eqref{eq:LP_2_ineq} hold with equality, Equation~\eqref{eq:LP_eq}/\eqref{eq:LP_2_eq} yields
\begin{align*}
\Cell=\Cell\sum_{k=0}^{M-1}\tilde{\ell}_k^\star=\tilde{\ell}_M^\star\sum_{k=0}^{M-1}\rho^{k-M}=\tilde{\ell}^\star_M\rho^{-M}\dfrac{1-\rho^M}{1-\rho},
\end{align*}
and thus
\begin{align*}
\tilde{\ell}_M^\star=\Cell\rho^M\dfrac{1-\rho}{1-\rho^M}.
\end{align*}
Hence, the maximum in~\eqref{eq:LP}/\eqref{eq:LP_2} is given by
\begin{align*} 
\epsilon_{\mathrm{f}}^\star=&\tilde{\ell}^\star_M\left[1+\dfrac{1-\omega}{\omega \Cell\rho^M}\right] 
=\dfrac{1-\rho}{1-\rho^M}\left[\Cell\rho^M+\dfrac{1-\omega}{\omega}\right].
\end{align*}
\textit{Case b: } 
Suppose that $1+\frac{1-\omega}{\Cell\rho^M\omega}\leq 0$, i.e., the objective is to minimize $\tilde{\ell}_M$ (or is independent of $\tilde{\ell}_M$).
In this case, a (non-unique) maximizer is given by $\tilde{\ell}_M^\star=0$, $\tilde{\ell}_1^\star=1$, $\tilde{\ell}_k^\star=0$, $k\in\mathbb{I}_{[2,M-1]}$, with $\epsilon_{\mathrm{f}}=0$.\\ 
\textit{Combination: }
By combining the two cases, we get~\eqref{eq:epsilon_f}.}{%
Inequality~\eqref{eq:epsilon_f} is derived similar to~\cite[Prop.~4]{koehler2021LP}, by computing the worst-case cost sequence compatible with Inequality~\eqref{eq:ell_exp_stable}, see~\cite[App.~B]{adaptive_arxiv} for details.}

\textit{Proof of Thm.~\ref{thm:stable_artificial}:}
Considering the candidate $(x_{\mathrm{s}},u_{\mathrm{s}},y_{\mathrm{s}})=(x_{\mathrm{s},x,\theta}^\star,u_{\mathrm{s},x,\theta}^\star,y_{\mathrm{s},x,\theta}^\star)$, Inequality~\eqref{eq:RDP} holds if 
\begin{align*}
&\min_{u\in\mathbb{U}^N}\JN(f(x,\pi(x,\theta),\theta),\theta,u,x^\star_{\mathrm{s},x,\theta},u_{\mathrm{s},x,\theta}^\star)\\
\leq& \min_{u\in\mathbb{U}^N}\JN(x,\theta,u,x^\star_{\mathrm{s},x,\theta},u_{\mathrm{s},x,\theta}^\star)\\
&-\alpha\ell(x,\pi(x,\theta),x_{\mathrm{s},x,\theta}^\star,u_{\mathrm{s},x,\theta}^\star),
\end{align*}
i.e., it suffices to consider Problem~\eqref{eq:MPC} with a fixed artificial setpoint $x_{\mathrm{s}}$.
This problem has been analysed in~\cite[Thm. 6--7]{koehler2021LP} using an LP analysis. 
In particular, the bounds derived in Propositions~\ref{prop:cost_controllable}--\ref{prop:finite_tail} correspond to~\cite[Asm.~4-5]{koehler2021LP}.\footnote{%
The lower and upper bound in~\cite[Asm.~4-5]{koehler2021LP} hold trivially by defining $\overline{c}_{\mathrm{f}}=\gamma_1/(1+\epsilon_{\mathrm{f}})$ and $\underline{c}_{\mathrm{f}}=\omega$.} 
Using~\cite[Thm.~7]{koehler2021LP}, Inequality~\eqref{eq:RDP} holds with 
{\begin{align}
\label{eq:alpha}
\alpha=1-\dfrac{\epsilon_{\mathrm{f}}(\gamma_N-1)\prod_{j=1}^{N-1}(\gamma_{N-j+1}-1)}{(1+\epsilon_{\mathrm{f}})\prod_{j=1}^{N-1}\gamma_{N-j+1}-\epsilon_{\mathrm{f}}\prod_{j=1}^{N-1}(\gamma_{N-j+1}-1)}.
\end{align}}
\ifbool{arxiv}{Using the fact that $\gamma_k\leq \overline{\gamma}$, we also get 
\begin{align*}
&\alpha\geq 1-\dfrac{\epsilon_{\mathrm{f}}(\overline{\gamma}-1)^{N}}{(1+\epsilon_{\mathrm{f}})\overline{\gamma}^{N-1}-\epsilon_{\mathrm{f}}(\overline{\gamma}-1)^{N-1}}.
\end{align*}
 Correspondingly, $\alpha>0$ if $(1+\epsilon_{\mathrm{f}})\overline{\gamma}^{N-1}>\epsilon_{\mathrm{f}}\overline{\gamma}(\overline{\gamma}-1)^{N-1}$. 
For any finite $\epsilon_{\mathrm{f}}\geq 0$, this condition holds by choosing $N$ large enough. 
Furthermore, to ensure any horizon $N\in\mathbb{I}_{\geq 1}$ ensures stability, a sufficient condition is given by $\overline{\gamma}\epsilon_{\mathrm{f}}<1$.
Using the specific formulas for $\overline{\gamma},\epsilon_{\mathrm{f}}$,\footnote{%
The case distinction in~\eqref{eq:epsilon_f} can be ignored, since $\epsilon_{\mathrm{f}}=0$ trivially satisfies the condition.} this condition reduces to}{
Furthermore, $\alpha>0$ for $N$ sufficiently large~\cite{koehler2021LP}.  
Using the formula for $\gamma_N$, a simple sufficient condition for $\alpha>0$ with $N\geq 1$ is given by  (cf.~\cite[App.~B]{adaptive_arxiv})}: 
\begin{align*} 
\Cell\omega\left(1-\Cell\rho^M\right)>\Cell-1+\rho^M.
\end{align*}
In case $1-\Cell\rho^M>0$, this holds if
\begin{align*}
&\omega>\underline{\omega}:=\max\left\{\dfrac{\Cell-1+\rho^M}{\Cell(1-\Cell\rho^M)},1\right\}.
\end{align*} 
\textit{Proof of Prop.~\ref{prop:artificial_bound}:} 
\ifbool{arxiv}{%
The proof follows the arguments in~\cite[Prop.~1]{soloperto2021nonlinear}. 
For contradiction, suppose Inequality~\eqref{eq:artificial_bound} does not hold, i.e., 
\begin{align}
\label{eq:artificial_bound_contradiction}
\|x-x_{\mathrm{s},x,\theta}^\star\|_Q^2> \overline{a}\|x_{\mathrm{s},x,\theta}^\star-x_{\mathrm{rd},\theta}\|_Q^2.
\end{align} 
Consider $\tilde{y}=\beta y_{\mathrm{s},x,\theta}^\star+(1-\beta)y_{\mathrm{rd},\theta}$ with some later specified constant $\beta\in(0,1)$. 
Assumptions~\ref{ass:steady_regular}\ref{item:ass_unique}--\ref{item:ass_convex} ensure that $(\tilde{x}_{\mathrm{s}},\tilde{u}_{\mathrm{s}},\tilde{y}_{\mathrm{s}})\in\mathbb{S}(\theta)$ with $\tilde{x}_{\mathrm{s}}=\gyx(\tilde{y}_{\mathrm{s}},\theta)$, $\tilde{u}_{\mathrm{s}}=\gyu(\tilde{y}_{\mathrm{s}},\theta)$. 
Furthermore, the convex quadratic offset cost satisfies 
\begin{align}
\label{eq:convex_offset_cost}
&\|\tilde{y}_{\mathrm{s}}-y_{\mathrm{d}}\|_T^2-\|y_{\mathrm{s},x,\theta}^\star-y_{\mathrm{d}}\|^2_T\leq -(1-\beta^2)\|y_{\mathrm{s},x,\theta}^\star-y_{\mathrm{rd},\theta}\|_T^2.
\end{align} 
Utilizing the feasible candidate solution $(\tilde{x}_{\mathrm{s}},\tilde{u}_{\mathrm{s}},\tilde{y}_{\mathrm{s}})\in\mathbb{S}(\theta)$, we have
\begin{align*}
0\leq& \JN^\star(x,\theta)-\|y^\star_{\mathrm{s},x,\theta}-y_{\mathrm{d}}\|_T^2\\
\leq& \min_{\mathbf{u}\in\mathbb{U}^N}\JN(x,\theta,\mathbf{u},\tilde{x}_{\mathrm{s}},\tilde{u}_{\mathrm{s}})
+\|\tilde{y}_{\mathrm{s}}-y_{\mathrm{d}}\|_T^2-\|y^\star_{\mathrm{s},x,\theta}-y_{\mathrm{d}}\|_T^2\\
\stackrel{\eqref{eq:cost_control},\eqref{eq:convex_offset_cost}}{\leq} &\gamma_N\|x-\tilde{x}_{\mathrm{s}}\|_Q^2-(1-\beta^2)\|y_{\mathrm{s},x,\theta}^\star-y_{\mathrm{rd},\theta}\|_T^2\\
\leq& 2\gamma_N( \|x-x_{\mathrm{s},x,\theta}^\star\|_Q^2+\|x_{\mathrm{s},x,\theta}^\star-\tilde{x}_{\mathrm{s}}\|_Q^2)\\
&-(1-\beta^2)\|y_{\mathrm{s},x,\theta}^\star-y_{\mathrm{rd},\theta}\|_T^2\\
\stackrel{\eqref{eq:g_Lipschitz},\eqref{eq:artificial_bound_contradiction}}{<}&2\gamma_N\overline{a} \|x_{\mathrm{s},x,\theta}^\star-x_{\mathrm{rd},\theta}\|_Q^2
+2\gamma_NL_{\mathrm{g}}\|y_{\mathrm{s},x,\theta}^\star-\tilde{y}_{\mathrm{s}}\|_T^2\\
&-(1-\beta^2)\|y_{\mathrm{s},x,\theta}^\star-y_{\mathrm{rd},\theta}\|_T^2\\ 
\stackrel{\eqref{eq:g_Lipschitz}}{\leq} &\left[ 2\gamma_NL_{\mathrm{g}}(\overline{a}+(1-\beta)^2)-(1-\beta^2) \right] \|y_{\mathrm{s},x,\theta}^\star-y_{\mathrm{rd},\theta}\|_T^2,
\end{align*}
where the last step also used $ y_{\mathrm{s},x,\theta}^\star-\tilde{y}= (1-\beta)(y_{\mathrm{s},x,\theta}^\star-y_{\mathrm{rd},\theta})$.
For $\beta$ sufficiently close to $1$ and $\overline{a}$ sufficiently small, the last term is non-positive, which results in a contradiction.}
{The proof follows the arguments in~\cite[Prop.~1]{soloperto2021nonlinear}, see~\cite[App.~B]{adaptive_arxiv} for details.}

\textit{Proof of Corollary~\ref{corol:nominal}:}
We show stability using the Lyapunov function $V(x,\theta):=\JN^\star(x,\theta)-\|y_{\mathrm{rd},\theta}-y_{\mathrm{d}}\|_T^2$~\cite{limon2018nonlinear,soloperto2021nonlinear}. 
Abbreviate $x^+=\fw(x,u,\theta,0)$, $u=\pi(x,\theta)$. 
Proposition~\ref{prop:artificial_bound} implies
\begin{align}
\label{eq:ell_x_rd}
&\ell(x,u,x_{\mathrm{s},x,\theta}^\star,u_{\mathrm{s},x,\theta}^\star)\geq \|x-x_{\mathrm{s},x,\theta}^\star\|_Q^2\nonumber\\
\stackrel{\eqref{eq:artificial_bound}}{\geq}&\frac{1}{2}\|x-x_{\mathrm{s},x,\theta}^\star\|_Q^2+\frac{\overline{a}}{2}\|x_{\mathrm{s},x,\theta}^\star-x_{\mathrm{rd},\theta}\|_Q^2\nonumber\\
\geq& \frac{\min\{\overline{a},1\}}{4}\|x-x_{\mathrm{rd},\theta}\|_Q^2.
\end{align}
By applying this bound to the decrease condition in Theorem~\ref{thm:stable_artificial}, we get
\begin{align}
\label{eq:value_decrease_nominal}
&V(x^+,\theta)-V(x,\theta)\stackrel{\eqref{eq:RDP}}{\leq} -\alpha\ell(x,\pi(x,\theta),x_{\mathrm{s},x,\theta},u_{\mathrm{s},x,\theta}^\star)\nonumber\\ 
\stackrel{\eqref{eq:ell_x_rd}}{\leq} & -\frac{\min\{\overline{a},1\}\alpha}{4}\|x-x_{\mathrm{rd},\theta}\|_Q^2.
\end{align}	
The lower bound on $V$ holds with
\begin{align}
\label{eq:value_lower_bound}
&V(x,\theta)\geq \ell(x,\pi(x,\theta),x_{\mathrm{s},x,\theta},u_{\mathrm{s},x,\theta}^\star)\nonumber\\
\stackrel{\eqref{eq:ell_x_rd}}{\geq}& \frac{\min\{\overline{a},1\}}{4}\|x-x_{\mathrm{rd},\theta}\|_Q^2,
\end{align}	
where the first inequality uses the fact that $y_{\mathrm{rd},\theta}$ is a minimizer for~\eqref{eq:opt_steady_state}. 
An upper bound on $V$ follows using Proposition~\ref{prop:cost_controllable} and the candidate $(x_{\mathrm{s}},u_{\mathrm{s}},y_{\mathrm{s}})=(x_{\mathrm{rd},\theta},u_{\mathrm{rd},\theta},y_{\mathrm{rd},\theta})$:
\begin{align}
\label{eq:value_upper_bound}
V(x,\theta)\leq \min_{\mathbf{u}\in\mathbb{U}^N}\JN(x,\theta,\mathbf{u},x_{\mathrm{rd},\theta},u_{\mathrm{rd},\theta})\stackrel{\eqref{eq:cost_control}}{\leq} \gamma_N\|x-x_{\mathrm{rd},\theta}\|_Q^2.
\end{align} 
Exponential stability of $x_{\mathrm{rd},\theta}$ follows using standard Lyapunov arguments~\cite[App.~B]{rawlings2017model}.

\ifbool{arxiv}{\newpage}{}

\section{Proofs of Section~\ref{sec:global_robust}}
\label{app:proof_global_robust}
\ifbool{arxiv}{
We first provide a series of lemmas to establish continuity of the stage cost (Lemma~\ref{lemma:stage_cost_epsilon}), dynamics (Lemma~\ref{lemma:prediction_lipschitz}), open-loop cost (Lemma~\ref{lemma:continuity_openloop_cost}), feasible steady-states (Lemma~\ref{lemma:continuity_steady_state_theta}), and the optimal cost (Lemma~\ref{lemma:continuity_opt_cost}), before proving Theorem~\ref{thm:robust_stability_global}.

\begin{lemma}
\label{lemma:stage_cost_epsilon}
There exists a constant $\sigma_\xi>0$, such that for any $x,\tilde{x},\tilde{x}_{\mathrm{s}},x_{\mathrm{s}}\in\mathbb{R}^{\nx}$, $u,\tilde{u},u_{\mathrm{s}},\tilde{u}_{\mathrm{s}}\in\mathbb{R}^{\nu}$, and any $\epsilon>0$, the quadratic stage cost satisfies
\begin{align}
\label{eq:stage_cost_epsilon}
&\dfrac{1}{1+\epsilon}\ell(x+\tilde{x},u+\tilde{u},x_{\mathrm{s}}+\tilde{x}_{\mathrm{s}},u_{\mathrm{s}}+\tilde{u}_{\mathrm{s}})\\
\leq &\ell(x,u,x_{\mathrm{s}},u_{\mathrm{s}})
+\dfrac{1}{\epsilon}\left[\|\tilde{x}-\tilde{x}_{\mathrm{s}}\|_Q^2+\|\tilde{u}-\tilde{u}_{\mathrm{s}}\|_R^2+\sigma_\xi\|\tilde{x}\|^2\right].\nonumber
\end{align}
\end{lemma}
\begin{pf}
Cauchy-Schwarz and Young’s inequality applied to the quadratic part of the stage cost~\eqref{eq:ell} imply
\begin{align*}
&\|x+\tilde{x}-(x_{\mathrm{s}}+\tilde{x}_{\mathrm{s}})\|_Q^2+\|u+\tilde{u}-(u_{\mathrm{s}}+\tilde{u}_{\mathrm{s}})\|_R^2\\
\leq &
(1+\epsilon)\left[\|x-x_{\mathrm{s}}\|_Q^2+\|u-u_{\mathrm{s}}\|_R^2\right]\\
&+\dfrac{1+\epsilon}{\epsilon} \left[\|\tilde{x}-\tilde{x}_{\mathrm{s}}\|_Q^2+\|\tilde{u}-\tilde{u}_{\mathrm{s}}\|_R^2\right].
\end{align*}
For the soft-constrained penalty, note that 
\begin{align*}
\max\{D_i (x+\tilde{x})-d_i,0\}\leq \max\{D_i x-d_i,0\}+|D_i\tilde{x}| 
\end{align*}
and hence 
\begin{align*}
&\sum_{i=1}^r q_{\xi,i}\max\{D_i (x+\tilde{x})-d_i,0\}^2\\
\leq& (1+\epsilon) \sum_{i=1}^r q_{\xi,i}\max\{D_i x-d_i,0\}^2
+\dfrac{1+\epsilon}{\epsilon}\sigma_{\xi}\|\tilde{x}\|^2,
\end{align*}
with $\sigma_{\xi}=\sigma_{\max}(\sum_{i=1}^r q_{\xi,i} D_i^\top D_i)$ from Appendix~\ref{app:proof_global_analysis}. 
Adding both bounds, using the definition of the cost~\eqref{eq:ell} and dividing by $1+\epsilon$ yields~\eqref{eq:stage_cost_epsilon}.
\end{pf}
\begin{lemma}(Lipschitz continuity)
\label{lemma:prediction_lipschitz}
Let Assumption~\ref{ass:regularity}\ref{item:ass_Lipschitz_differentiable} hold. 
There exist constants $c_{\mathrm{x},k}\geq 0$, $k\in\mathbb{I}_{\geq 0}$, such that for all $x,\tilde{x}\in\mathbb{R}^{\nx}$, $\theta,\tilde{\theta}\in\Theta$, $N\in\mathbb{I}_{\geq 0}$ and all $\mathbf{u},\tilde{\mathbf{u}}\in\mathbb{U}^N$, it holds that 
\begin{align}
\label{eq:prediction_Lipschitz}
&\|x_{\tilde{\mathbf{u}}}(k,\tilde{x},\tilde{\theta})-x_{\mathbf{u}}(k,x,\theta)\|\\
\leq& c_{\mathrm{x},k}\left[\|\tilde{x}-x\|+\|\tilde{\theta}-\theta\|+\max_{j\in\mathbb{I}_{[0,k-1]}}\|\tilde{\mathbf{u}}_j-\mathbf{u}_j\|\right].\nonumber 
\end{align}
\end{lemma} 
\begin{pf} 
Denote $\tilde{x}_k=x_{\tilde{\mathbf{u}}}(k,\tilde{x},\tilde{\theta})$, $x_k=x_{\mathbf{u}}(k,x,\theta)$, $k\in\mathbb{I}_{\geq 0}$. For any $k\in\mathbb{I}_{\geq 0}$, it holds that 
\begin{align*}
&\|\tilde{x}_{k+1}-x_{k+1}\|
\leq L_{\mathrm{\fw}}(\|\tilde{x}_k-x_k\|+\|\tilde{\theta}-\theta\|+\|\tilde{\mathbf{u}}_k-\mathbf{u}_k\|)\\
\leq& \dots\\
\leq & L_{\mathrm{\fw}}^{k+1}(\|\tilde{x}_0-x_0\|
+\sum_{j=1}^{k+1}L_{\mathrm{\fw}}^j
(\|\tilde{\theta}-\theta\|
+\max_{j\in\mathbb{I}_{[0,k]}}\|\tilde{\mathbf{u}}_j-\mathbf{u}_j\|), 
\end{align*}
which ensures~\eqref{eq:prediction_Lipschitz} with $c_{\mathrm{x},k}=\sum_{j=1}^{k}L_{\mathrm{\fw}}^j$.
\end{pf}
While the constant $c_{\mathrm{x},k}$ in~\eqref{eq:prediction_Lipschitz} depends on the horizon $k$, \change{horizon independent bounds} could also be established by using the fact that exponential stability (Asm.~\ref{ass:stable}) and a compact invariant set (cf. proof Lemma~\ref{lemma:stable_bounded}) imply incremental stability~\cite[Thm.~8]{karapetyan2025closed}. 
\begin{lemma}
\label{lemma:continuity_openloop_cost}
Let Assumptions~\ref{ass:regularity}\ref{item:ass_Lipschitz_differentiable} hold. 
For any $N,M\in\mathbb{I}_{\geq 0}$, $\omega\in\mathbb{R}$, there exists a constant $c_{\mathrm{J}}>0$, such that for all $x,\tilde{x},x_{\mathrm{s}},\tilde{x}_{\mathrm{s}}\in\mathbb{R}^{\nx}$, $\theta,\tilde{\theta}\in\Theta$, $u_{\mathrm{s}},\tilde{u}_{\mathrm{s}}\in\mathbb{U}$, $N\in\mathbb{I}_{\geq 0}$, all $\mathbf{u}\in\mathbb{U}^N$, and any $\epsilon>0$ it holds that
\begin{align}
\label{eq:continuity_openloop_cost}
&\frac{\epsilon}{1+\epsilon}\JN(\tilde{x},\tilde{\theta},\mathbf{u},\tilde{x}_{\mathrm{s}},\tilde{u}_{\mathrm{s}})
-\epsilon\JN(x,\theta,\mathbf{u},x_{\mathrm{s}},u_{\mathrm{s}})\\
\leq & c_{\mathrm{J}}\left[\|x-\tilde{x}\|^2+\|\theta-\tilde{\theta}\|^2+\|u_{\mathrm{s}}-\tilde{u}_{\mathrm{s}}\|^2+\|x_{\mathrm{s}}-\tilde{x}_{\mathrm{s}}\|^2\right].\nonumber
\end{align}
\end{lemma} 
\begin{pf} 
Let us define the extended input sequences $\mathbf{u}',\tilde{\mathbf{u}}'\in\mathbb{U}^{N+M}$ with $\mathbf{u}'_k=\tilde{\mathbf{u}}_k=\mathbf{u}_k$, $k\in\mathbb{I}_{[0,N-1]}$, $\mathbf{u}'_k=u_{\mathrm{s}}$, $\tilde{\mathbf{u}}'_k=\tilde{u}_{\mathrm{s}}$, $k\in\mathbb{I}_{[N,N+M-1]}$. 
Furthermore, we denote the two state predictions by $x_k=x_{\mathbf{u}'}(k,x,\theta)$, $\tilde{x}_k=x_{\tilde{\mathbf{u}}'}(k,\tilde{x},\tilde{\theta})$, $k\in\mathbb{I}_{[0,N+M-1]}$. 
Lemma~\ref{lemma:prediction_lipschitz} in combination with Cauchy Schwarz inequality yields 
\begin{align}
\label{eq:continuity_openloop_cost_step1}
\|x_k-\tilde{x}_k\|^2\leq 3c_{\mathrm{x},k}\left[\|x-\tilde{x}\|^2+\|\theta-\tilde{\theta}\|^2+\|u_{\mathrm{s}}-\tilde{u}_{\mathrm{s}}\|^2\right].
\end{align}
Applying Lemma~\ref{lemma:stage_cost_epsilon} then yields the following bound on the stage cost 
\begin{align*}
&\dfrac{\epsilon}{1+\epsilon}\ell(\tilde{x}_k,\tilde{\mathbf{u}}'_k,\tilde{x}_{\mathrm{s}},\tilde{u}_{\mathrm{s}})
-\epsilon\ell(x_k,\mathbf{u}'_k,x_{\mathrm{s}},u_{\mathrm{s}})\\
\leq&\|x_k-\tilde{x}_k+\tilde{x}_{\mathrm{s}}-x_{\mathrm{s}}\|_Q^2+\|u_{\mathrm{s}}-\tilde{u}_{\mathrm{s}}\|_R^2+\sigma_\xi \|x_k-\tilde{x}_k\|^2\\
\leq&(2\sigma_{\max}(Q)+\sigma_\xi) \|x_k-\tilde{x}_k\|^2+2\|\tilde{x}_{\mathrm{s}}-x_{\mathrm{s}}\|_Q^2+\|u_{\mathrm{s}}-\tilde{u}_{\mathrm{s}}\|_R^2.
\end{align*}
Applying this bound in the definition of the finite-horizon cost~\eqref{eq:V_f}, \eqref{eq:J_N} and applying~\eqref{eq:continuity_openloop_cost_step1} yields~\eqref{eq:continuity_openloop_cost} with $c_{\mathrm{J}}>0$. 
\end{pf}
\begin{lemma}
\label{lemma:continuity_steady_state_theta}
Let Assumptions~\ref{ass:regularity}\ref{item:ass_steady_state} and \ref{ass:steady_regular}\ref{item:ass_unique}, \ref{item:ass_regularity_steady_state_2} hold.
For any $\theta,\tilde{\theta}\in\Theta$, any $(x_{\mathrm{s}},u_{\mathrm{s}},y_{\mathrm{s}})\in\mathbb{S}(\theta)$, there exists $(\tilde{x}_{\mathrm{s}},\tilde{u}_{\mathrm{s}},\tilde{y}_{\mathrm{s}})\in\mathbb{S}(\tilde{\theta})$:
\begin{align}
\label{eq:continuity_steady_state_theta}
&\|\tilde{x}_{\mathrm{s}}-x_{\mathrm{s}}\|_Q^2+\|\tilde{u}_{\mathrm{s}}-u_{\mathrm{s}}\|_R^2+\|\tilde{y}_{\mathrm{s}}-y_{\mathrm{s}}\|_T^2\leq c_{\mathrm{s}}\|\theta-\tilde{\theta}\|^2.
\end{align} 
\end{lemma} 
\begin{pf}
Pick $(\tilde{x}_{\mathrm{s}},\tilde{u}_{\mathrm{s}},\tilde{y}_{\mathrm{s}})\in\mathbb{S}(\tilde{\theta})$ according to Assumption~\ref{ass:steady_regular}\ref{item:ass_regularity_steady_state_2}. 
The Lipschitz continuous mapping $(x,u)=\gy(y,\theta)$ from Assumption~\ref{ass:steady_regular}\ref{item:ass_unique} ensures 
\begin{align*}
&\|\tilde{x}_{\mathrm{s}}-x_{\mathrm{s}}\|_Q^2+\|\tilde{u}_{\mathrm{s}}-u_{\mathrm{s}}\|_R^2+\|\tilde{y}_{\mathrm{s}}-y_{\mathrm{s}}\|_T^2\\
\stackrel{\eqref{eq:g_Lipschitz}}{\leq}& (L_{\mathrm{\gy}}+1) \|\tilde{y}_{\mathrm{s}}-y_{\mathrm{s}}\|_T^2+L_{\mathrm{\gy}}\|\theta-\tilde{\theta}\|^2\\
\stackrel{\eqref{eq:regularity_steady_state}}{\leq} & 
(L_{\mathrm{s}}(L_{\mathrm{\gy}}+1)+L_{\mathrm{\gy}})\|\theta-\tilde{\theta}\|^2. 
\end{align*}
\end{pf}}{%
We first establish a continuity bound on the optimal cost.}
\begin{lemma}
\label{lemma:continuity_opt_cost} 
Let Assumptions~\ref{ass:regularity}\ref{item:ass_steady_state}, \ref{item:ass_Lipschitz_differentiable} and \ref{ass:steady_regular}\ref{item:ass_unique}, \ref{item:ass_regularity_steady_state_2} hold.
There exists a constant $c_{\mathrm{J}^\star}>0$, such that for any $x,\tilde{x}\in\mathbb{R}^{\nx}$, $\theta,\tilde{\theta}\in\Theta$, and any $\epsilon>0$, the optimal cost $\JN^\star$ from Problem~\eqref{eq:MPC} satisfies
\begin{align}
\label{eq:continuity_opt_cost}
\dfrac{\epsilon}{1+\epsilon}\JN^\star(\tilde{x},\tilde{\theta})\leq \epsilon \JN^\star(x,\theta)+c_{\mathrm{J}^\star}\left[\|x-\tilde{x}\|^2+\|\theta-\tilde{\theta}\|^2\right]
\end{align}
\end{lemma} 
\ifbool{arxiv}{\begin{pf}
For $(x,\theta)$, the optimal solution to Problem~\eqref{eq:MPC} is given by $\JN^\star(x,\theta)$, $\mathbf{u}^\star_{x,\theta}\in\mathbb{U}^N$, and $(x_{\mathrm{s},x,\theta}^\star,u_{\mathrm{s},x,\theta}^\star,y_{\mathrm{s},x,\theta}^\star)\in\mathbb{S}(\theta)$. 
Consider $(\tilde{x}_{\mathrm{s}},\tilde{u}_{\mathrm{s}},\tilde{y}_{\mathrm{s}})\in\mathbb{S}(\tilde{\theta})$ according to Lemma~\ref{lemma:continuity_steady_state_theta}, i.e.,
\begin{align}
\label{eq:continuity_steady_state_theta_in_proof}
&\|\tilde{x}_{\mathrm{s}}-x_{\mathrm{s},x,\theta}\|_Q^2+\|\tilde{u}_{\mathrm{s}}-u_{\mathrm{s},x,\theta}\|_R^2+\|\tilde{y}_{\mathrm{s}}-y_{\mathrm{s},x,\theta}\|_T^2\leq c_{\mathrm{s}}\|\theta-\tilde{\theta}\|^2.
\end{align}
We derive an upper bound for $\JN^\star(\tilde{x},\tilde{\theta})$ by using the feasible candidate solution $\tilde{\mathbf{u}}=\mathbf{u}^\star_{x,\theta}\in\mathbb{U}^N$, $(\tilde{x}_{\mathrm{s}},\tilde{u}_{\mathrm{s}},\tilde{y}_{\mathrm{s}})\in\mathbb{S}(\tilde{\theta})$
 to Problem~\eqref{eq:MPC}:
\begin{align*}
&\dfrac{\epsilon}{1+\epsilon}\JN^\star(\tilde{x},\tilde{\theta})\\
\leq&\dfrac{\epsilon}{1+\epsilon} \JN(\tilde{x},\tilde{\theta},\tilde{\mathbf{u}},\tilde{x}_{\mathrm{s}},\tilde{u}_{\mathrm{s}})
+\dfrac{\epsilon}{1+\epsilon}\|\tilde{y}_{\mathrm{s}}-y_{\mathrm{d}}\|_T^2\\
\stackrel{\eqref{eq:continuity_openloop_cost}}{\leq}& \epsilon \JN(x,\theta,\mathbf{u}^\star_{x,\theta},x^\star_{\mathrm{s},x,\theta},u^\star_{\mathrm{s},x,\theta})+\epsilon\|y_{\mathrm{s},x,\theta}^\star-y_{\mathrm{d}}\|_T^2\\
 &+ c_{\mathrm{J}}\|x-\tilde{x}\|^2+ c_{\mathrm{J}}\|\theta-\tilde{\theta}\|^2\\
&+ c_{\mathrm{J}}\|u^\star_{\mathrm{s},x,\theta}-\tilde{u}_{\mathrm{s}}\|^2+ c_{\mathrm{J}}\|x^\star_{\mathrm{s},x,\theta}-\tilde{x}_{\mathrm{s}}\|^2
+\|y_{\mathrm{s},x,\theta}^\star-\tilde{y}_{\mathrm{s}}\|_T^2\\
\stackrel{\eqref{eq:continuity_steady_state_theta_in_proof}}{\leq} &\epsilon \JN^\star(x,\theta)
+c_{\mathrm{J}}\|x-\tilde{x}\|^2\\
&+\underbrace{\left(c_{\mathrm{J}}+c_{\mathrm{s}}\max\left\{1,\dfrac{c_{\mathrm{J}}}{\min\{\sigma_{\min}(Q),\sigma_{\min}(R)\}}\right\}\right)}_{c_{\mathrm{J}^\star}}\|\theta-\tilde{\theta}\|^2, 
\end{align*} 
where, analogous to Lemma~\ref{lemma:stage_cost_epsilon}, the second inequality applied Cauchy-Schwarz and Young’s inequality to the quadratic term.
\end{pf}}{%
This bound follows from the considered regularity conditions and the fact that Problem~\eqref{eq:MPC} only has soft state constraints, see~\cite[App.~C]{adaptive_arxiv} for details.}

\textit{Proof of Thm.~\ref{thm:robust_stability_global}:}
The closed-loop dynamics~\eqref{eq:closed_loop_adaptive} satisfy: 
\begin{align}
\label{eq:robust_stability_proof_1}
\hat{x}_{k+1}=&\fw(\hat{x}_k-v_k,\pi(\hat{x}_k,\hat{\theta}_k),\theta_k,w_k)-v_{k+1}\\
=&\fw(\hat{x}_k,\pi(\hat{x}_k,\hat{\theta}_k),\hat{\theta}_k,0)+\tilde{w}_k+\tilde{x}_{1|k},~ k\in\mathbb{I}_{\geq 0},\nonumber
\end{align}
with $\tilde{x}_{1|k}$ and $\tilde{w}_k$ defined in~\eqref{eq:LMS_tilde}. 
Inequality~\eqref{eq:value_decrease_nominal} from Corollary~\ref{corol:nominal} applied to $(\hat{x}_k,\hat{\theta}_k$) ensures the following nominal decrease
\begin{align}
\label{eq:robust_stability_proof_3}
&\JN^\star(\fw(\hat{x}_k,\pi(\hat{x}_k,\hat{\theta}_k),\hat{\theta}_k,0),\hat{\theta}_k)-\JN^\star(\hat{x}_k,\hat{\theta}_k)\nonumber\\
\leq& -\alpha\underline{c}_1 \|\hat{x}_k-x_{\mathrm{rd},\hat{\theta}_k}\|_Q^2.
\end{align}	
with $\underline{c}_1:=\frac{\min\{\overline{a},1\}}{4}>0$ and $\alpha>0$ from Theorem~\ref{thm:stable_artificial}. 
Combining this with the continuity bound from Lemma~\ref{lemma:continuity_opt_cost} and the prediction error~\eqref{eq:robust_stability_proof_1} yields
\begin{align}
\label{eq:robust_stabiltity_proof_4}
&\JN^\star(\hat{x}_{k+1},\hat{\theta}_{k+1})\\
\stackrel{\eqref{eq:continuity_opt_cost},\eqref{eq:robust_stability_proof_1}}{\leq}& (1+\epsilon) \JN^\star(\fw(\hat{x}_k,\pi(\hat{x}_k,\hat{\theta}_k),\hat{\theta}_k,0),\hat{\theta}_k)\nonumber\\
&+\dfrac{1+\epsilon}{\epsilon}c_{\mathrm{J}^\star}\left[\|\tilde{w}_k+\tilde{x}_{1|k}\|^2+\|\hat{\theta}_{k+1}-\hat{\theta}_k\|^2\right]\nonumber\\
\stackrel{\eqref{eq:robust_stability_proof_3}}{\leq} &(1+\epsilon) \JN^\star(\hat{x}_k,\hat{\theta}_k)-\alpha\underline{c}_1 \|\hat{x}_k-x_{\mathrm{rd},\hat{\theta}_k}\|_Q^2\nonumber\\
&+\dfrac{1+\epsilon}{\epsilon}c_{\mathrm{J}^\star}\left[\|\tilde{w}_k+\tilde{x}_{1|k}\|^2+\|\hat{\theta}_{k+1}-\hat{\theta}_k\|^2\right].\nonumber
\end{align} 
Due to Assumption~\ref{ass:steady_regular}\ref{item:ass_regularity_steady_state_1}, \ifbool{arxiv}{%
\begin{align*}
V(\hat{x}_k,\hat{\theta}_k)=\JN^\star(\hat{x}_k,\hat{\theta}_k)-\underbrace{\|y_{\mathrm{rd},\hat{\theta}_k}-y_{\mathrm{d}}\|_T^2}_{=0},
\end{align*}
i.e., the Lyapunov function from Corollary~\ref{corol:nominal} is simply the optimal cost $\JN^\star$.}{we have $V\equiv\JN^\star$.} 
Inequalities~\eqref{eq:value_lower_bound}--\eqref{eq:value_upper_bound} ensure 
\begin{align}
\label{eq:robust_stabiltity_proof_6}
\underline{c}_1\|\hat{x}_k-x_{\mathrm{rd},\hat{\theta}_k}\|_Q^2\leq \JN^\star(\hat{x}_k,\hat{\theta}_k)\leq \gamma_N\|\hat{x}_k-x_{\mathrm{rd},\hat{\theta}_k}\|_Q^2.
\end{align} 
Thus, we have 
\begin{align}
\label{eq:robust_stabiltity_proof_7}
&\epsilon \JN^\star(\hat{x}_k,\hat{\theta}_k)-\alpha\underline{c}_1\|\hat{x}_k-x_{\mathrm{rd},\hat{\theta}_k}\|_Q^2\\
\stackrel{\eqref{eq:robust_stabiltity_proof_6}}\leq & -\left[\alpha\underline{c}_1-\gamma_N\epsilon\right]\|\hat{x}_k-x_{\mathrm{rd},\hat{\theta}_k}\|_Q^2\nonumber\\
=& -\alpha\underline{c}_1/2\|\hat{x}_k-x_{\mathrm{rd},\hat{\theta}_k}\|_Q^2\nonumber\\
\leq &-(1-\rho_{\mathrm{V}}) \JN^\star(\hat{x}_k,\hat{\theta}_k),\nonumber
\end{align}
where the second-to-last equality holds by choosing $\epsilon=\alpha\underline{c}_1/(2\gamma_N)>0$ and the last equality holds with $\rho_{\mathrm{V}}:=1-\alpha\underline{c}_1/(2 \gamma_N)\in(0,1)$. 
Plugging Inequality~\eqref{eq:robust_stabiltity_proof_7} in \eqref{eq:robust_stabiltity_proof_4} yields~\eqref{eq:Lyap_decrease_robust} with $c_{\mathrm{V}}=\dfrac{1+\epsilon}{\epsilon}c_{\mathrm{J}^\star}$.

\ifbool{arxiv}{\newpage}{}
\section{Proofs of Section~\ref{sec:global_adaptive}}
\label{app:proof_global_adaptive}
\ifbool{arxiv}{%
The following lemma provides a further characterization of the steady-states, which is needed to show Lemma~\ref{lemma:stable_bounded}. 
\begin{lemma}
\label{lemma:Lipschitz_steady_state}
Let Assumptions~\ref{ass:regularity}\ref{item:ass_steady_state}, \ref{item:ass_Lipschitz_differentiable}, and \ref{ass:stable} hold. 
There exists a function $\gu:\mathbb{U}\times \Theta\rightarrow\mathbb{R}^{\nx}$, such that for all $(u,\theta)\in\mathbb{U}\times \Theta$, $x=\gu(u,\theta)$ satisfies $x=\fw(x,u,\theta,0)$. 
Furthermore, $\gu$ is Lipschitz continuous with constant $L_{\mathrm{\gu}}\geq 0$ and the set of steady-states for $(u,\theta)\in\mathbb{U}\times\Theta$ is compact. 
\end{lemma} 
\begin{pf}
The proof builds on ideas from~\cite{limon2018nonlinear,berberich2022linear_model}. 
The function $\gu$ needs to satisfy $\fw(\gu(u,\theta),u,\theta,0)-\gu(u,\theta)=0$ for all $(u,\theta)\in\mathbb{U}\times\Theta$. 
Analogous to~\cite{limon2018nonlinear}, existence of this function follows from the implicit function theorem, taking into account stability (Asm.~\ref{ass:stable}), differentiability (Asm.~\ref{ass:regularity}\ref{item:ass_Lipschitz_differentiable}), existence of some steady-state (Asm.~\ref{ass:regularity}\ref{item:ass_steady_state}), and convexity of the set $\mathbb{U}\times\Theta$ (Asm.~\ref{ass:regularity}\ref{item:ass_Lipschitz_differentiable}). 
Let us denote the Jacobian of $\fw$ w.r.t. $\{x,u,\theta\}$ by $\{f_x,f_u,f_\theta\}$ and the Jacobian of $\gu$ by $\gu'$, where we omit the arguments for simplicity. 
Using the implicit function theorem, the Jacobian satisfies $\gu'=(I-f_x)^{-1}[f_u^\top,f_\theta^\top]^\top$. 
Exponential stability (Asm.~\ref{ass:stable}) and continuous differentiability (Asm.~\ref{ass:regularity}\ref{item:ass_Lipschitz_differentiable}) imply that $\|(I-f_x)^{-1}\|\leq \sigma$, with some \change{constant} $\sigma>0$, see also~\cite[Asm.~4]{berberich2022linear_model}.
Furthermore, $\|[f_u^\top,f_\theta^\top]^\top\|$ is \change{bounded} given Assumption~\ref{ass:regularity}\ref{item:ass_Lipschitz_differentiable}. 
Thus, $\gu'$ is \change{bounded} and $\gu$ is Lipschitz continuous. 
Compactness of the set of steady-states follows from Lipschitz continuity of $\gu$ and compactness of $\mathbb{U}\times\Theta$. 
\end{pf}

\textit{Proof of Lemma~\ref{lemma:stable_bounded}:}
First, we derive a compact set $\mathcal{R}$ such that $x_k\in\mathcal{R}$ $\forall k\in\mathbb{I}_{\geq 0}$ by adapting the results in~\cite[Lemmas~1--2]{nonhoff2024online}. 
Then we compute a gain $\Gamma\succ 0$ satisfying Assumption~\ref{ass:param_gain}. \\
Given Lemma~\ref{lemma:Lipschitz_steady_state}, recall that an input $u\in\mathbb{U}$ and parameter $\theta\in\Theta$ uniquely define a corresponding steady-state $x_{\mathrm{s}}=\gu(u,\theta)$. 
Consider the Lyapunov function
\begin{align}
\label{eq:Lyap_converse}
W(x,\theta,u):=\sum_{k=0}^{M-1}\|x_{\mathbf{u}^k}(k,x,\theta)-\gu(u,\theta)\|
\end{align}
with any fixed finite horizon $M\in\mathbb{I}_{\geq 0}$ satisfying $\Crho\rho^M<1$ with $\Crho>0$, $\rho\in[0,1)$ from Assumption~\ref{ass:stable}. 
Analogous to~\cite[Lemma~1]{nonhoff2024online}, for any $(x,\theta,u)\in \mathbb{R}^{\nx}\times \Theta \times \mathbb{U}$ and $x_{\mathrm{s}}=\gu(u,\theta)$:
\begin{align}
\label{eq:Lyap_converse_bounds}
&\|x-x_{\mathrm{s}}\|\leq W(x,\theta,u)\stackrel{\eqref{eq:exp_stable}}{\leq} \underbrace{\dfrac{\Crho}{1-\rho}}_{=:c_2}\|x-x_{\mathrm{s}}\|,\\
&W(\fw(x,u,\theta,0),\theta,u)-W(x,\theta,u)\nonumber\\
\label{eq:Lyap_converse_contraction}
\stackrel{\eqref{eq:exp_stable}}{\leq}& -\underbrace{(1-\Crho\rho^M)}_{=:c_3>0} \|x-x_{\mathrm{s}}\|\leq -(1-\lambda)W(x,\theta,u),
\end{align}
with the exponential contraction rate $\lambda=1-\frac{c_3}{c_2}\in(0,1)$. 
Furthermore, for any $(x,\theta,u),(\tilde{x},\tilde{\theta},\tilde{u})\in \mathbb{R}^{\nx}\times\Theta\times\mathbb{U}$ with steady-state $x_{\mathrm{s}}=\gu(u,\theta)$, $\tilde{x}_{\mathrm{s}}=\gu(\tilde{u},\tilde{\theta})$, it holds that
\begin{align}
\label{eq:Lyap_converse_cont_step}
&W(x,\theta,u)-W(\tilde{x},\tilde{\theta},\tilde{u})\nonumber\\
\leq& \sum_{k=0}^{M-1}\|x_{\mathbf{u}^k}(k,x,\theta)-x_{\tilde{\mathbf{u}}^k}(k,\tilde{x},\tilde{\theta})\|+\|x_{\mathrm{s}}-\tilde{x}_{\mathrm{s}}\|.
\end{align}
For any $k\in\mathbb{I}_{\geq 1}$, Lipschitz continuity of $\fw$ implies 
\begin{align*}
&\|x_{\mathbf{u}^k}(k,x,\theta)-x_{\tilde{\mathbf{u}}^k}(k,\tilde{x},\tilde{\theta})\|\\
\leq& L_{\mathrm{f}}\|x_{\mathbf{u}^k}(k-1,x,\theta)-x_{\tilde{\mathbf{u}}^k}(k-1,\tilde{x},\tilde{\theta})\|\\
&+ L_{\mathrm{f}}(\|\theta-\tilde{\theta}\|+\|u-\tilde{u}\|)\\
\leq &\dots \leq L_{\mathrm{f}}^k\|x-\tilde{x}\|+(\|\theta-\tilde{\theta}\|+\|u-\tilde{u}\|) \sum_{j=1}^k L_{\mathrm{f}}^j.
\end{align*}
Plugging this bound into~\eqref{eq:Lyap_converse_cont_step} and using Lipschitz continuity of $\gu$ (Lemma~\ref{lemma:Lipschitz_steady_state}) yields
\begin{align}
\label{eq:Lyap_converse_cont}
&W(x,\theta,u)-W(\tilde{x},\tilde{\theta},\tilde{u})\\
\leq& \sum_{k=0}^{M-1}\left[ L_{\mathrm{f}}^k\|x-\tilde{x}\|+(\|\theta-\tilde{\theta}\|+\|u-\tilde{u}\|) \sum_{j=1}^k L_{\mathrm{f}}^j\right]\nonumber\\
&+M\cdot L_{\mathrm{\gu}}(\|u-\tilde{u}\|+\|\theta-\tilde{\theta}\|)\nonumber\\
=:&c_{\mathrm{x}}\|x-\tilde{x}\|+c_{\theta}\|\theta-\tilde{\theta}\|+c_{\mathrm{u}}\|u-\tilde{u}\|.\nonumber
\end{align}
For arbitrary inputs $u_k\in\mathbb{U}$, disturbances $w_k\in\mathbb{W}$, and parameters $\theta_k\in\Theta$, System~\eqref{eq:sys} satisfies
\begin{align}
\label{eq:Lyap_converse_ISS}
&W(x_{k+1},\theta_{k+1},u_{k+1})\\
\stackrel{\eqref{eq:sys}}{=}& W(\fw(x_{k},u_{k},\theta_k,w_k),\theta_{k+1},u_{k+1})\nonumber\\
\stackrel{\eqref{eq:Lyap_converse_cont}}{\leq} &W(\fw(x_{k},u_{k},\theta_k,0),\theta_k,u_{k})\nonumber\\
&+c_\theta\|\theta_{k+1}-\theta_k\|
+c_{\mathrm{u}}\|u_{k+1}-u_{k}\|+c_{\mathrm{x}}L_{\mathrm{f}}\|w_k\|\nonumber\\
\stackrel{\eqref{eq:Lyap_converse_contraction}}{\leq} &\lambda W(x_k,\theta_k,u_k)+c_\theta\|\Delta \theta_k\|+c_{\mathrm{u}}\|u_{k+1}-u_{k}\|+c_{\mathrm{x}}L_{\mathrm{f}}\|w_k\|,\nonumber
\end{align}
where the first inequality also used Lipschitz continuity of $\fw$ (Asm.~\ref{ass:regularity}\ref{item:ass_Lipschitz_differentiable}). 
Given the compact sets $\mathbb{W},\mathbb{U},\Theta$, we define
\begin{align*}
\bar{c}:=c_\theta\max_{\theta,\tilde{\theta}\in\Theta}\|\theta-\tilde{\theta}\|+c_{\mathrm{u}}\max_{u,\tilde{u}\in\mathbb{U}} \|u-\tilde{u}\|+c_{\mathrm{x}}L_{\mathrm{f}}\max_{w\in\mathbb{W}}\|w\|.
\end{align*} 
This yields
\begin{align*}
&W(x_{k+1},\theta_{k+1},u_{k+1})\leq \lambda W(x_k,\theta_k,u_k)+\bar{c}\\
\leq& \dots\leq \lambda^{k+1} W(x_0,\theta_0,u_0)+\dfrac{1-\lambda^{k+1}}{1-\lambda}\bar{c}.
\end{align*}
Lastly, we can bound 
\begin{align}
W(x_0,\theta_0,u_0)\stackrel{\eqref{eq:Lyap_converse_bounds}}{\leq} \max_{u\in\mathbb{U},x\in\mathbb{X}_0,\theta\in\Theta} \dfrac{\Crho}{1-\rho}\|x-\gu(u,\theta)\|=:\bar{W}_0
\end{align}
using the compact set of initial conditions $x_0\in \mathbb{X}_0$ and Lipschitz continuity of $\gu$. 
Thus, the following set
\begin{align*}
\mathcal{R}=\left\{x\in\mathbb{R}^{\nx}|~W(x,\theta,u)\leq \bar{W}_0+\dfrac{\bar{c}}{1-\lambda},u\in\mathbb{U},\theta\in\Theta\right\},
\end{align*}
satisfies $x_k\in\mathcal{R}$, $k\in\mathbb{I}_{\geq 0}$. 
Furthermore, $\mathcal{R}$ is compact given~\eqref{eq:Lyap_converse_bounds} and the compact set of steady-states (Lemma~\ref{lemma:Lipschitz_steady_state}). 
Using $\hat{\Phi}_k=G(\hat{x}_k,u_k,0)$, $\hat{x}_k=x_k+v_k$ with $v_k\in\mathbb{V}$, 
a gain $\Gamma\succ 0$ satisfying the Assumption~\ref{ass:param_gain} can be computed with the following optimization problem 
\begin{align}
\label{eq:compute_gamma}
\Gamma=\arg\max~&\mathrm{trace}(\Gamma)\\
\text{s.t. }& \Phi(x+v,u) \Gamma \Phi(x+v,u)^\top\preceq I\nonumber\\
& \forall (x,u,v)\in\mathcal{R}\times\mathbb{U}\times\mathbb{V}.\nonumber
\end{align}
Note that $\Gamma\succ 0$ follows from Lipschitz continuity of $\Phi$ (Asm.~\ref{ass:steady_regular}\ref{item:ass_Lipschitz_differentiable}) and compactness of $\mathcal{R}\times\mathbb{U}\times\mathbb{V}$.}{%
\textit{Proof of Lemma~\ref{lemma:stable_bounded}:} 
Lipschitz continuity, open-loop stability, and bounded control inputs ensure that $x_k$ lies in a compact set \change{$\mathcal{R}$. The derivation of this set is analogous to~\cite[Lemmas~1--2]{nonhoff2024online}, see \cite[App.~D]{adaptive_arxiv} for a detailed proof. 
Given the set $\mathcal{R}$, the gain $\Gamma\succ0$ can be chosen according to
the following optimization problem} 
\begin{align}
\label{eq:compute_gamma}
&\Gamma=\arg\max~\mathrm{trace}(\Gamma)\\
&\text{s.t. } \Phi(x+v,u) \Gamma \Phi(x+v,u)^\top\preceq I\nonumber
\quad \forall (x,u,v)\in\mathcal{R}\times\mathbb{U}\times\mathbb{V}.\nonumber
\end{align}}
%

\textit{Proof of Thm.~\ref{thm:adaptive_global}:}
\ifbool{arxiv}
{%
We first relate the stage cost $\ell$ to the desired quantities (output tracking and constraint violation) in Objective~\ref{objective_1}.
Then, we use the robust stability properties of the MPC (Thm.~\ref{thm:robust_stability_global}) and the LMS (Thm.~\ref{thm:LMS}) in a telescopic sum. 
Lastly, we bound all remaining terms using noise, disturbances, parameter variations, and initial condition.}{} \\
\textbf{Part I: } 
\ifbool{arxiv}{Given that the state constraints $\mathbb{X}$ are given by a finite number of half space constraints~\eqref{eq:sys_constraint}, Hoffman's Lemma~\cite{hoffman2003approximate} yields a \change{constant} $c_{\mathbb{X}}>0$, such that for all $\forall x\in\mathbb{R}^{\nx}$:
\begin{align}
\label{eq:hoffman_inequality}
\|x\|_{\mathbb{X}}^2\leq c_{\mathbb{X}} \sum_{i=1}^r \max\{D_i x-d_i,0\}^2 , 
\end{align}
i.e., the soft constraint penalty provides a bound on the distance to the constraint set. 
Given Lipschitz continuity of $y=h(x)$ (Asm.~\ref{ass:regularity}\ref{item:ass_Lipschitz_differentiable}) and $y_{\mathrm{rd},\theta}$ independent of $\theta$ (Asm.~\ref{ass:steady_regular}\ref{item:ass_regularity_steady_state_1}), the output tracking error satisfies
\begin{align*}
\|y_k-y_{\mathrm{rd},\theta}\|\leq L_{\mathrm{h}}\|x_k-x_{\mathrm{rd},\hat{\theta}_k}\|
\stackrel{\eqref{eq:noise_measurement}}{\leq} L_{\mathrm{h}}\|\hat{x}_k-x_{\mathrm{rd},\hat{\theta}_k}\|+L_{\mathrm{h}}\|v_k\|.
\end{align*}
Analogous to Inequality~\eqref{eq:ell_x_rd}, the stage cost $\ell$ satisfies 
\begin{align}
\label{eq:ell_x_rd_v2}
&\ell(x,u,x_{\mathrm{s},x,\theta}^\star,u_{\mathrm{s},x,\theta}^\star)\\
\geq& \frac{\min\{\overline{a},1\}}{4}\|x-x_{\mathrm{rd},\theta}\|_Q^2+\sum_{i=1}^r q_{\xi,i}\max\{D_i x-d_i,0\}^2,\nonumber
\end{align}
for all $x\in\mathbb{R}^{\nx}$, $\theta\in\Theta$. 
Combining all three bounds yields
\begin{align}
\label{eq:proof_adaptive_ell_bound}
&\|x_k\|_{\mathbb{X}}^2+\|y_k-y_{\mathrm{rd},\theta_k}\|^2\\
\leq& 2\|\hat{x}_k\|_{\mathbb{X}}^2+2\|v_k\|^2
+ 2 L_{\mathrm{h}}^2\|\hat{x}_k-x_{\mathrm{rd},\hat{\theta}_k}\|^2+2L_{\mathrm{h}}^2\|v_k\|^2\nonumber\\
\leq & 2c_{\mathbb{X}} \sum_{i=1}^r \max\{D_i \hat{x}_k-d_i,0\}^2
+2\|v_k\|^2\nonumber\\
&+ 2 L_{\mathrm{h}}^2\|\hat{x}_k-x_{\mathrm{rd},\hat{\theta}_k}\|^2+2L_{\mathrm{h}}^2\|v_k\|^2\nonumber\\
\leq &\underbrace{\max\left\{\frac{2c_{\mathbb{X}}}{\min_i q_{\xi,i}},\dfrac{8 L_{\mathrm{h}}^2}{\sigma_{\min}(Q)\min\{1,\bar{a}\}}\right\}}_{=:c_\ell}\nonumber\\
&\left[\frac{\min\{\overline{a},1\}}{4}\|\hat{x}_k-x_{\mathrm{rd},\hat{\theta}_k}\|_Q^2+\sum_{i=1}^r q_{\xi,i}\max\{D_i \hat{x}_k-d_i,0\}^2\right]\nonumber\\
&+2(1+L_{\mathrm{h}}^2)\|v_k\|^2\nonumber\\
\leq & c_\ell\ell(\hat{x}_k,u_k,x_{\mathrm{s},\hat{x}_k,\hat{\theta}_k}^\star,u_{\mathrm{s},\hat{x}_k,\hat{\theta}_k}^\star)+2(1+L_{\mathrm{h}}^2)\|v_k\|^2,\nonumber
\end{align}
where the last inequality used~\eqref{eq:ell_x_rd_v2} with $x=\hat{x}_k$, $\theta=\hat{\theta}_k$ and the first inequality used $\|a+b\|^2\leq 2\|a\|^2+2\|b\|^2$.}
{%
First, we relate the stage cost $\ell$ to the desired quantities (output tracking and constraint violation) in Objective~\ref{objective_1} as follows:
\begin{align}
\label{eq:proof_adaptive_ell_bound}
&\|x_k\|_{\mathbb{X}}^2+\|y_k-y_{\mathrm{rd},\theta_k}\|^2\\
\leq& 2\|\hat{x}_k\|_{\mathbb{X}}^2+2\|v_k\|^2
+ 2 L_{\mathrm{h}}^2\|\hat{x}_k-x_{\mathrm{rd},\hat{\theta}_k}\|^2+2L_{\mathrm{h}}^2\|v_k\|^2\nonumber\\
\leq & 2c_{\mathbb{X}} \sum_{i=1}^r \max\{D_i \hat{x}_k-d_i,0\}^2
+2\|v_k\|^2\nonumber\\
&+ 2 L_{\mathrm{h}}^2\|\hat{x}_k-x_{\mathrm{rd},\hat{\theta}_k}\|^2+2L_{\mathrm{h}}^2\|v_k\|^2\nonumber\\
\leq &\underbrace{\max\left\{\frac{2c_{\mathbb{X}}}{\min_i q_{\xi,i}},\dfrac{8 L_{\mathrm{h}}^2}{\sigma_{\min}(Q)\min\{1,\bar{a}\}}\right\}}_{=:c_\ell}\nonumber\\
&\left[\frac{\min\{\overline{a},1\}}{4}\|\hat{x}_k-x_{\mathrm{rd},\hat{\theta}_k}\|_Q^2+\sum_{i=1}^r q_{\xi,i}\max\{D_i \hat{x}_k-d_i,0\}^2\right]\nonumber\\
&+2(1+L_{\mathrm{h}}^2)\|v_k\|^2\nonumber\\
\leq & c_\ell\ell(\hat{x}_k,u_k,x_{\mathrm{s},\hat{x}_k,\hat{\theta}_k}^\star,u_{\mathrm{s},\hat{x}_k,\hat{\theta}_k}^\star)+2(1+L_{\mathrm{h}}^2)\|v_k\|^2.\nonumber
\end{align}
The first inequality used Lipschitz continuity of $y=h(x)$ (Asm.~\ref{ass:regularity}\ref{item:ass_Lipschitz_differentiable}),  $y_{\mathrm{rd},\theta_k}$ independent of $\theta$ (Asm.~\ref{ass:steady_regular}\ref{item:ass_regularity_steady_state_1}), the noisy output measurements~\eqref{eq:noise_measurement}, and $\|a+b\|^2\leq 2\|a\|^2+2\|b\|^2$. 
The second inequality uses Hoffman's Lemma~\cite{hoffman2003approximate}, which guarantees that for a polytope $\mathbb{X}$ there exists a \change{constant} $c_{\mathbb{X}}>0$ such that
\begin{align}
\label{eq:hoffman_inequality}
\|x\|_{\mathbb{X}}^2 \le c_{\mathbb{X}} \sum_{i=1}^r \max\{D_i x - d_i, 0\}^2, \quad \forall x \in \mathbb{R}^{\nx}.
\end{align} 
The last inequality is analogous to~\eqref{eq:ell_x_rd}.
}\\
\textbf{Part II: }
Applying Inequality~\eqref{eq:LMS_step} from the LMS to the robust stability Inequality~\eqref{eq:Lyap_decrease_robust} yields
\begin{align}
\label{eq:proof_global_adaptive_decrease}
&\JN^\star(\hat{x}_{k+1},\hat{\theta}_{k+1})-\rho_{\mathrm{V}} \JN^\star(\hat{x}_k,\hat{\theta}_k)\nonumber\\
\stackrel{\eqref{eq:Lyap_decrease_robust}}{\leq} &c_{\mathrm{V}}\left(\|\tilde{w}_k+\tilde{x}_{1|k}\|^2+\|\hat{\theta}_{k+1}-\hat{\theta}_k\|^2\right)\nonumber\\
\stackrel{\eqref{eq:LMS_step}}{\leq}& c_{\mathrm{V}}(1+\sigma_{\min}(\Gamma^{-1}))\|\tilde{w}_k+\tilde{x}_{1|k}\|^2\nonumber\\
\leq& \underbrace{2 c_{\mathrm{V}}(1+\sigma_{\min}(\Gamma^{-1}))}_{\tilde{c}_{\mathrm{V}}}\left[\|\tilde{w}_k\|^2+\|\tilde{x}_{1|k}\|^2\right].
\end{align}
The combined Lyapunov function $V(\hat{x}_k,\hat{\theta}_k,\theta_k):=\JN^\star(\hat{x}_k,\hat{\theta}_k)+\tilde{c}_{\mathrm{V}} V_\theta(\hat{\theta}_k-\theta_k)$ satisfies
\begin{align}
\label{eq:proof_global_adaptive_LMS_combined}
&V(\hat{x}_{k+1},\hat{\theta}_{k+1},\theta_{k+1})-V(\hat{x}_k,\hat{\theta}_k,\theta_k)\\
\stackrel{\eqref{eq:LMS_Lyap}}\leq& \tilde{c}_{\mathrm{V}} c_\theta\sqrt{\Vtheta(\Delta \theta_k)}+2\tilde{c}_{\mathrm{V}}\|\tilde{w}_k\|^2-(1-\rho_{\mathrm{V}})\JN^\star(\hat{x}_{k},\hat{\theta}_{k}).\nonumber
\end{align}
\ifbool{arxiv}{Using $V\geq 0$ and a telescopic sum yields
\begin{align}
\label{eq:proof_global_adaptive_telescopic_1}
&\sum_{k=0}^{\Time-1}(1-\rho_{\mathrm{V}})\JN^\star(\hat{x}_k,\hat{\theta}_k)\\
\leq&\sum_{k=0}^{\Time-1} \tilde{c}_{\mathrm{V}} c_\theta\sqrt{\Vtheta(\Delta \theta_k)}+2\tilde{c}_{\mathrm{V}}\|\tilde{w}_k\|^2+ V(\hat{x}_0,\hat{\theta}_0,\theta_0).\nonumber
\end{align}
Furthermore, $\ell(\hat{x}_k,u_k,x_{\mathrm{s},\hat{x}_k,\hat{\theta}_k}^\star,u_{\mathrm{s},\hat{x}_k,\hat{\theta}_k}^\star)\leq \JN^\star(\hat{x}_k,\hat{\theta}_k)$ 
yields
\begin{align}
\label{eq:proof_global_adaptive_telescopic}
&\sum_{k=0}^{\Time-1}\|x_k\|_{\mathbb{X}}^2+\|y_k-y_{\mathrm{rd},\theta_k}\|^2\nonumber\\
\stackrel{\eqref{eq:proof_adaptive_ell_bound}}{\leq} &\sum_{k=0}^{\Time-1}c_\ell\ell(\hat{x}_k,u_k,x_{\mathrm{s},\hat{x}_k,\hat{\theta}_k}^\star,u_{\mathrm{s},\hat{x}_k,\hat{\theta}_k}^\star)+2(1+L_{\mathrm{h}}^2)\|v_k\|^2\nonumber\\
\stackrel{\eqref{eq:proof_global_adaptive_telescopic_1}}{\leq} &\frac{c_\ell}{1-\rho_{\mathrm{V}}}\left[V(\hat{x}_0,\hat{\theta}_0,\theta_0)+\sum_{k=0}^{\Time-1} \tilde{c}_{\mathrm{V}} c_\theta\sqrt{\Vtheta(\Delta \theta_k)}+2\tilde{c}_{\mathrm{V}}\|\tilde{w}_k\|^2 \right]\nonumber\\
&+\sum_{k=0}^{\Time-1}2(1+L_{\mathrm{h}}^2)\|v_k\|^2.
\end{align}}{%
Using a telescopic sum,   $\ell(\hat{x}_k,u_k,x_{\mathrm{s},\hat{x}_k,\hat{\theta}_k}^\star,u_{\mathrm{s},\hat{x}_k,\hat{\theta}_k}^\star)\leq \JN^\star(\hat{x}_k,\hat{\theta}_k)$, and  $V\geq 0$ yields
\begin{align}
\label{eq:proof_global_adaptive_telescopic}
&\sum_{k=0}^{\Time-1}\|x_k\|_{\mathbb{X}}^2+\|y_k-y_{\mathrm{rd},\theta_k}\|^2\nonumber\\
{\leq} &\frac{c_\ell}{1-\rho_{\mathrm{V}}}\left[V(\hat{x}_0,\hat{\theta}_0,\theta_0)+\sum_{k=0}^{\Time-1} \tilde{c}_{\mathrm{V}} c_\theta\sqrt{\Vtheta(\Delta \theta_k)}+2\tilde{c}_{\mathrm{V}}\|\tilde{w}_k\|^2 \right]\nonumber\\
&+\sum_{k=0}^{\Time-1}2(1+L_{\mathrm{h}}^2)\|v_k\|^2.
\end{align}
}
\textbf{Part III: }
\ifbool{arxiv}{%
To ensure~\eqref{eq:desired_stability}, all terms on the right hand side in~\eqref{eq:proof_global_adaptive_telescopic} need to be bounded proportional to $\|w_k\|^2+\|v_k\|^2+\|\Delta\theta_k\|$. 
The prediction error due to noise $\tilde{w}$ satisfies 
\begin{align*}
\|\tilde{w}_k\|^2
\stackrel{\eqref{eq:LMS_w_bound}}{\leq} \left(\|v_{k+1}\|+L_{\mathrm{f}}(\|w_k\|+\|v_k\|)\right)^2\\
\leq 3\max\{1,L_{\mathrm{f}}^2\}\left[\|v_{k+1}\|^2+\|v_k\|^2+\|w_k\|^2\right].
\end{align*}
The parameter variation satisfies
\begin{align*}
\sqrt{\Vtheta(\Delta\theta_k)}=\sqrt{\Delta \theta_k^\top\Gamma^{-1}\Delta\theta_k}\leq \sqrt{\sigma_{\min}(\Gamma^{-1})}\|\Delta \theta_k\|.
\end{align*}
From Assumption~\ref{ass:steady_regular}\ref{item:ass_unique} and $y_{\mathrm{rd},\theta}=y_{\mathrm{rd},\hat{\theta}_0}$ (Asm.~\ref{ass:steady_regular}\ref{item:ass_regularity_steady_state_1}), we have
\begin{align}
\label{eq:proof_global_adaptive_Lipschitz_init}
\|x_{\mathrm{rd},\hat{\theta}_0}-x_{\mathrm{rd},\theta_0}\|_Q^2 \leq L_{\mathrm{\gy}} \|\theta_0-\hat{\theta}_0\|^2.
\end{align}
Thus, the Lyapunov function $V$ at time $k=0$ can be bounded using Proposition~\ref{prop:cost_controllable}:
\begin{align*}
&V(\hat{x}_0,\hat{\theta}_0,\theta_0)=\JN^\star(\hat{x}_0,\hat{\theta}_0)+\tilde{c}_{\mathrm{V}} V_\theta(\hat{\theta}_0-\theta_0)\\
\stackrel{\eqref{eq:value_upper_bound}}{\leq}&\gamma_N\|\hat{x}_0-x_{\mathrm{rd},\hat{\theta}_0}\|_Q^2+ \tilde{c}_{\mathrm{V}}V_\theta(\hat{\theta}_0-\theta_0)\\
\leq &2\gamma_N(\|\hat{x}_0-x_{\mathrm{rd},\theta_0}\|_Q^2+\|x_{\mathrm{rd},\theta_0}-x_{\mathrm{rd},\hat{\theta}_0}\|_Q^2)+ \tilde{c}_{\mathrm{V}} V_\theta(\hat{\theta}_0-\theta_0)\\
\stackrel{\eqref{eq:proof_global_adaptive_Lipschitz_init}}{\leq} &( \tilde{c}_{\mathrm{V}} \sigma_{\min}(\Gamma^{-1})+2\gamma_N L_{\mathrm{\gy}})\|\hat{\theta}_0-\theta_0\|^2\\
&+4\gamma_N\left(\|x_0-x_{\mathrm{rd},\theta_0}\|^2+\|v_0\|^2\right).
\end{align*}}
{The initial condition further satisfies
\begin{align*}
&V(\hat{x}_0,\hat{\theta}_0,\theta_0)=\JN^\star(\hat{x}_0,\hat{\theta}_0)+\tilde{c}_{\mathrm{V}} V_\theta(\hat{\theta}_0-\theta_0)\\
{\leq} &( \tilde{c}_{\mathrm{V}} \sigma_{\min}(\Gamma^{-1})+2\gamma_N L_{\mathrm{\gy}})\|\hat{\theta}_0-\theta_0\|^2\\
&+4\gamma_N\left(\|x_0-x_{\mathrm{rd},\theta_0}\|^2+\|v_0\|^2\right),
\end{align*}
using Proposition~\ref{prop:cost_controllable}, Assumption~\ref{ass:steady_regular}\ref{item:ass_unique} and $y_{\mathrm{rd},\theta_0}=y_{\mathrm{rd},\hat{\theta}_0}$.
Furthermore, recall that $\|\tilde{w}_k\|$ can be bounded proportional to the process and measurement noise using \eqref{eq:LMS_w_bound}. 
Thus, all terms in~\eqref{eq:proof_global_adaptive_telescopic} can be bounded by the expression in Objective~\ref{objective_1}, 
see~\cite[App.~D]{adaptive_arxiv} for the exact constants $C_1,C_2>0$. 
}
\ifbool{arxiv}{\clearpage}{}

\section{Proof of Section~\ref{sec:regional}}
\label{app:regional}
 \textit{Proof of Theorem~\ref{thm:stable_artificial_feedback}:}
Propositions~\ref{prop:cost_controllable} and~\ref{prop:finite_tail} directly generalize under Assumption~\ref{ass:stable_feedback} and $\Vfkappa$, however, the results only hold in the local neighbourhood $\mathbb{S}_{\mathrm{loc}}$.
Compared to Theorem~\ref{thm:stable_artificial}, further steps are required to derive a region of attraction $\mathbb{X}_{\bar{\mathcal{J}}}$ which is not restricted to the (potentially small) neighbourhood $\mathbb{S}_{\mathrm{loc}}$. 
By considering a fixed artificial setpoint $x_{\mathrm{s}}$ as a candidate solution, Inequality~\eqref{eq:RDP_feedback} reduces to the stability analysis in~\cite[Thm.~4.37]{kohler2021dynamic}. 
Specifically, a case distinction ensures that Inequalities~\eqref{eq:cost_control} and \eqref{eq:approx_CLF} remain valid for the tail of the prediction horizon $k\in\mathbb{I}_{\geq N_0}$, with $N_0:=\max\left\{0,\dfrac{\bar{\mathcal{J}}-\bar{\gamma} c_{\mathrm{loc}}}{c_{\mathrm{loc}}}\right\}$ and $\bar{\gamma}=\max_k \gamma_k$. 
Thus, a simple sufficient condition for Inequality~\eqref{eq:RDP_feedback}  with $\alpha>0$ and, e.g., $\omega=1$, is given by 
\begin{align}
\label{eq:local_horizon}
N\geq N_0+\dfrac{\log(\bar{\gamma})+\log(\epsilon_{\mathrm{f}})}{\log(\bar{\gamma})-\log(\bar{\gamma}-1)}\end{align}
with $\epsilon_{\mathrm{f}}$ according to~\eqref{eq:epsilon_f}, see~\cite[Equ.~(4.53)]{kohler2021dynamic}. 

\textit{Proof of Corollary~\ref{cor:stable_artificial_feedback}:} 
Given Theorem~\ref{thm:stable_artificial_feedback}, we adapt the proof of Corollary~\ref{corol:nominal}, which leverages Proposition~\ref{prop:artificial_bound}. 
The proof of Proposition~\ref{prop:artificial_bound} leverages the upper bound $\bar{\gamma}$ from Proposition~\ref{prop:cost_controllable}, which only holds on the local set $\mathbb{S}_{\mathrm{loc}}$ under Assumption~\ref{ass:stable_feedback}. 
However, a simple case distinction shows that the larger bound $\bar{\gamma}_{\bar{\mathcal{J}}}:=\max\{\bar{\gamma},\frac{\bar{\mathcal{J}}}{c_{\mathrm{loc}}}\}$ is valid in the region of attraction $\mathbb{X}_{\bar{\mathcal{J}}}$, see, e.g., \cite{soloperto2021nonlinear}. 
Thus, Proposition~\ref{prop:artificial_bound} remains valid for all $(x,\theta)\in\mathbb{X}_{\bar{\mathcal{J}}}$ with a possibly smaller constant $\bar{a}>0$.
We define $V_\kappa(x,\theta):=\JNkappa^\star(x,\theta)-\|y_{\mathrm{rd},\theta}-y_{\mathrm{d}}\|_T^2$, which is a monotonically decreasing function. Thus, the sublevel set $\mathbb{X}_{\bar{\mathcal{J}}}$ also remains positively invariant. Exponential stability follows with the Lyapunov function $V_\kappa$, analogous to Corollary~\ref{corol:nominal}.

\textit{Proof of Theorem~\ref{thm:robust_stability_regional_feedback}:} 
\ifbool{arxiv}{Lemmas~\ref{lemma:prediction_lipschitz} and \ref{lemma:continuity_steady_state_theta} are completely independent of $\kappa$ and thus remain valid. 
The continuity results in Lemmas~\ref{lemma:continuity_openloop_cost} and \ref{lemma:continuity_opt_cost} hold analogously for the cost $\JNkappa$ and $\JNkappa^\star$ due to the assumed Lipschitz continuity of $\kappa$ (Asm.~\ref{ass:stable_feedback}).}{%
The continuity result in Lemma~\ref{lemma:continuity_opt_cost} holds analogously for the cost $\JNkappa^\star$ due to the assumed Lipschitz continuity of $\kappa$ (Asm.~\ref{ass:stable_feedback}).} 
Thus, the derivation of Inequality~\eqref{eq:Lyap_decrease_robust_feedback} is analogous to Theorem~\ref{thm:robust_stability_global}, taking into account the \emph{regional} nominal stability properties from Theorem~\ref{thm:stable_artificial_feedback}. 
The main difference is that the analysis requires that $\JNkappa^\star(\hat{x}_k,\hat{\theta}_k)\leq \bar{\mathcal{J}}$ remains valid. 
This invariance follows from~\eqref{eq:Lyap_decrease_robust_feedback} with the assumed bound~\eqref{eq:robust_feedback_bound_mismatch}.

\textit{Proof of Lemma~\ref{lemma:bounded_feedback}:}
We know that $ x_k\in\mathcal{R}$, with 
\begin{align}
\label{eq:compute_Gamma_invariantSet_feedback}
\mathcal{R}=\{x|&~\exists \theta\in\Theta, v\in\mathbb{V}, (x_{\mathrm{s}},u_{\mathrm{s}},y_{\mathrm{s}})\in\bar{\mathbb{S}}(\theta):\\
&\|x+v-x_{\mathrm{s}}\|_Q^2\leq \bar{\mathcal{J}}\}.\nonumber
\end{align}
Compactness of the set of steady-states (Asm.~\ref{ass:regularity}\ref{item:ass_steady_state}) and $\Theta,\mathbb{V}$ compact imply that $\mathcal{R}$ in~\eqref{eq:compute_Gamma_invariantSet_feedback} is compact. Thus, Assumption~\ref{ass:param_gain} can be satisfied by choosing $\Gamma$ \ifbool{arxiv}{as in Lemma~\ref{lemma:stable_bounded} with~\eqref{eq:compute_gamma} using the compact set $\mathcal{R}$ in~\eqref{eq:compute_Gamma_invariantSet_feedback}}{sufficiently small, similar to Lemma~\ref{lemma:stable_bounded}}.

\textit{Proof of Theorem~\ref{thm:adaptive_feedback}:}
Inequalities~\eqref{eq:LMS_step}
--\eqref{eq:LMS_w_bound} and Assumption~\ref{ass:regularity}\ref{item:ass_Lipschitz_differentiable} ensure 
\begin{align*}
&\|\tilde{w}_k+\tilde{x}_{1|k}\|\\
\leq& \max\{L_{\mathrm{f}},1\}(\|w_k\|+\|v_k\|+\|v_{k+1}\|+\|\theta_k-\hat{\theta}_k\|),\\
&\|\hat{\theta}_{k+1}-\hat{\theta}_k\|^2\leq \sigma_{\min}(\Gamma^{-1})\|\tilde{w}_k+\tilde{x}_{1|k}\|^2. 
\end{align*}
Thus, $\theta_k,\hat{\theta}_k\in\Theta$, $w_k\in\mathbb{W}$, $v_k\in\mathbb{V}$ with $\Theta,\mathbb{V},\mathbb{W}$ sufficiently small ensure that Inequality~\eqref{eq:robust_feedback_bound_mismatch} holds, which ensures $(\hat{x}_k,\hat{\theta}_k)\in\mathbb{X}_{\bar{\mathcal{J}}}$, $\forall k\in\mathbb{I}_{\geq 0}$ with Theorem~\ref{thm:robust_stability_regional_feedback}.
Thus, we can also apply the results from Lemma~\ref{lemma:bounded_feedback} and Theorem~\ref{thm:LMS}. 
The remainder of the proof is analogous to Theorem~\ref{thm:adaptive_global}. 

 \textit{Proof of Corollary~\ref{cor:adaptive_regional}:}
For $\mathbb{W}=\{0\}$, $\mathbb{V}=\{0\}$, we have $\tilde{w}_k=0$. 
Analogous to Inequality~\eqref{eq:proof_global_adaptive_LMS_combined}, we obtain 
\begin{align*}
&V_\kappa(\hat{x}_{k+1},\hat{\theta}_{k+1},\theta_{k+1})-V_\kappa(\hat{x}_k,\hat{\theta}_k,\theta_k)\\
\leq&-(1-\rho_{\mathrm{V}})\JN^\star(\hat{x}_{k},\hat{\theta}_{k}).
\end{align*}
Thus, $\JNkappa^\star(\hat{x}_k,\hat{\theta}_k)\leq V_\kappa(\hat{x}_k,\hat{\theta}_k,\theta_k)\leq \bar{\mathcal{J}}$ holds recursively. 
The remainder of the proof is analogous to Theorem~\ref{thm:adaptive_global}. 
Inequality~\eqref{eq:desired_stability} with $w_k=0$, $v_k=0$, $\Delta \theta_k=0$, and the bounded initial condition directly yields~\eqref{eq:bounded_error_regional}. 
\ifbool{arxiv}{\clearpage}{}
\section{\change{Boundedness} with local Lipschitz continuity under small parameter variations}
\label{app:bounded_TV}
 In the following, we relax the \change{Lipschitz bound} (Asm.~\ref{ass:regularity}\ref{item:ass_Lipschitz_differentiable}) to a local Lipschitz condition.
\begin{assumption}
\label{ass:Lipschitz_local}
There exist \change{constants}  $L_{\mathrm{f},\theta},L_{\mathrm{f}}\geq 0$, such that for all $x,\tilde{x}\in\mathbb{R}^{\nx}$, $\theta,\tilde{\theta}\in\Theta$, $u,\tilde{u}\in\mathbb{U}$, $w,\tilde{w}\in\mathbb{W}$: 
\begin{align}
\label{eq:local_Lipschitz_fw}
&\|\fw(x,u,\theta,w)-\fw(\tilde{x},\tilde{u},\tilde{\theta},\tilde{w})\|
\leq L_{\mathrm{f},\theta}\|\theta-\tilde{\theta}\| (1+\|x\|)\nonumber\\
&+L_{\mathrm{f}}\left[\|x-\tilde{x}\|+\|u-\tilde{u}\|+\|w-\tilde{w}\|\right]. 
\end{align} 
\end{assumption}
Compared to Assumption~\ref{ass:regularity}\ref{item:ass_Lipschitz_differentiable}, the Lipschitz constant w.r.t. $\theta$ can increase with $\|x\|$, which is important as we do not assume a \change{bound} on $\|x\|$. 
The following proposition shows that Assumption~\ref{ass:Lipschitz_local}  holds naturally for general linear systems. 
\begin{proposition}
\label{prop:local_Lipschitz_linear}
Let Assumption~\ref{ass:linear_param} hold and suppose we have a linear dynamical system (cf. Sec.~\ref{sec:discussion_linear}). 
For any compact sets $\Theta,\mathbb{U},\mathbb{W}$, there exist \change{constants} $L_{\mathrm{f},\theta},L_{\mathrm{f}}\geq 0$ satisfying~\eqref{eq:local_Lipschitz_fw}.
\end{proposition}
\ifbool{arxiv}{%
\begin{pf} 
Inequality~\eqref{eq:local_Lipschitz_fw} follows from the compact sets $\mathbb{U},\mathbb{W},\Theta$:
\begin{align*}
&\|\fw(x,u,\theta,w)-\fw(\tilde{x},\tilde{u},\tilde{\theta},\tilde{w})\|\\
\leq &\|G(x,u,w)(\theta-\tilde{\theta})\|\\
&+\|A_{\tilde{\theta}}(x-\tilde{x})+B_{\tilde{\theta}}(u-\tilde{u})+G_{\tilde{\theta}}(w-\tilde{w})\|\\ 
\leq & c_1 (\|x\|+c_2)\|\theta-\tilde{\theta}\|+c_3\left[\|x-\tilde{x}\|+\|u-\tilde{u}\|+\|w-\tilde{w}\|\right],
\end{align*}
where $c_3=\max_{\theta\in\Theta}\max\{\|A_\theta\|,\|B_\theta\|,\|G_\theta\|,\|e_\theta\|\}$, $c_2=\max_{u\in\mathbb{U},w\in\mathbb{W}}\|u\|+\|w\|+1$, and $c_1$ is a Lipschitz constant of $A_\theta,B_\theta,G_\theta,e_\theta$. 
\end{pf}}{}
The following result generalizes Lemma~\ref{lemma:stable_bounded}. 
\begin{lemma}
\label{lemma:stable_bounded_local}
Let Assumptions~\ref{ass:regularity}\ref{item:ass_steady_state}, \ref{ass:stable}, and \ref{ass:Lipschitz_local} hold. 
\change{Consider} known compact sets $\mathbb{X}_0$, $\mathbb{W}$, $\mathbb{V}$, $\Theta$ satisfying $w_k\in\mathbb{W}$, $u_k\in\mathbb{U}$, $v_k\in\mathbb{V}$, $\theta_k\in\Theta$, $k\in\mathbb{I}_{\geq 0}$, and $x_0\in\mathbb{X}_0$.
Furthermore, suppose there exists a sufficiently small set $\Omega\subseteq \mathbb{R}^{\ntheta}$ with $0\in\mathrm{int}(\Omega)$, such that the parameter variation satisfies $\Delta \theta_k\in\Omega$, $\forall k\in\mathbb{I}_{\geq 0}$.
\change{Then, there exists a sufficiently small gain $\Gamma\succ 0$ satisfying Assumption~\ref{ass:param_gain}.}
\end{lemma}
\ifbool{arxiv}{%
\begin{pf}
The proof follows the arguments from Lemma~\ref{lemma:stable_bounded}. 
Let us denote $\bar{x}_k=x_{\mathbf{u}^k}(k,x,\theta)$ and $\tilde{x}_k=x_{\tilde{\mathbf{u}}^k}(k,\tilde{x},\tilde{\theta})$, $k\in\mathbb{I}_{[0,M-1]}$. 
Compact sets $\mathbb{U},\Theta$ and $\gu$ Lipschitz continuous imply
\begin{align}
\label{eq:proof_local_Lipschitz_1} 
\|\bar{x}_k\|\leq \|\bar{x}_k-\gu(u,\theta)\|+\|\gu(u,\theta)\|\stackrel{\eqref{eq:Lyap_converse_bounds}}{\leq} W(x,\theta,u)+c_{\gu},
\end{align}
with a \change{constant} $c_{\gu}\geq 0$.
With some abuse of notation, we abbreviate
\begin{align}
\label{eq:proof_local_Lipschitz_2}
L_{W}:=L_{\mathrm{f},\theta}(1+c_{\gu}+W(x,\theta,u)).
\end{align} 
The prediction error satisfies
\begin{align*}
&\|\bar{x}_{k+1}-\tilde{x}_{k+1}\|\\
\stackrel{\eqref{eq:local_Lipschitz_fw}}{\leq}& L_{\mathrm{f}}\left[\|\bar{x}_k-\tilde{x}_k\|+\|u-\tilde{u}\|\right]
+L_{\mathrm{f},\theta}\|\theta-\tilde{\theta}\|(1+\|\bar{x}_k\|)\\ 
\stackrel{\eqref{eq:proof_local_Lipschitz_1},\eqref{eq:proof_local_Lipschitz_2}}{\leq} & L_{\mathrm{f}}\left[\|\bar{x}_k-\tilde{x}_k\|+\|u-\tilde{u}\|\right]+L_W\|\theta-\tilde{\theta}\|.
\end{align*}
Recursive application yields
\begin{align*}
&\|\bar{x}_k-\tilde{x}_k\|\\
\leq &L_{\mathrm{f}}^k\|x-\tilde{x}\|+\left[L_W\|\theta-\tilde{\theta}\|+L_{\mathrm{f}}\|u-\tilde{u}\|\right]\sum_{j=1}^{k-1}L_{\mathrm{f}}^{j-1}.
\end{align*}
Analogous to Inequality~\eqref{eq:Lyap_converse_cont}, we arrive at 
\begin{align}
\label{eq:Lyap_converse_cont_local}
&W(x,\theta,u)-W(\tilde{x},\tilde{\theta},\tilde{u})\\ 
\leq&c_{\mathrm{x}}\|x-\tilde{x}\|+c_{\mathrm{u}}\|u-\tilde{u}\|+c_{\theta}\|\theta-\tilde{\theta}\|(1+W(x,\theta,u)),\nonumber
\end{align}
with \change{constants} $c_{\mathrm{x}}$, $c_{\mathrm{u}}$, $c_\theta\geq 0$, where the factor $1+W$ is due to \eqref{eq:proof_local_Lipschitz_2}. 
Analogous to~\eqref{eq:Lyap_converse_ISS}, this yields 
\begin{align}
\label{eq:Lyap_converse_ISS_local}
&W(x_{k+1},\theta_{k+1},u_{k+1})\\ 
\stackrel{\eqref{eq:Lyap_converse_cont_local},\eqref{eq:Lyap_converse_contraction}}{\leq} &(\lambda+c_\theta\|\Delta \theta_k\|) W(x_k,\theta_k,u_k)\nonumber\\
&+c_\theta\|\Delta \theta_k\|+c_{\mathrm{u}}\|u_{k+1}-u_{k}\|+c_{\mathrm{x}}L_{\mathrm{f}}\|w_k\|.\nonumber
\end{align}
Given the compact sets $\mathbb{W},\mathbb{U},\Omega$, we define
\begin{align*}
\bar{c}:=c_\theta\max_{\Delta\theta\in\Omega}\|\Delta\theta\|+c_{\mathrm{u}}\max_{u,\tilde{u}\in\mathbb{U}} \|u-\tilde{u}\|+c_{\mathrm{x}}L_{\mathrm{f}}\max_{w\in\mathbb{W}}\|w\|.
\end{align*} 
Given any $\tilde{\lambda}\in(\lambda,1)$, we have $\lambda+c_\theta\|\Delta \theta_k\|\leq \tilde{\lambda}<1$ with $\Omega=\left\{\Delta \theta\in\mathbb{R}^{\ntheta}|~\|\Delta\theta\|\leq \dfrac{\tilde{\lambda}-\lambda}{c_\theta}\right\}$, $0\in\mathrm{int}(\Omega)$. 
Thus, we have
\begin{align*}
&W(x_{k+1},\theta_{k+1},u_{k+1}) \leq \tilde{\lambda}W(x_k,\theta_k,u_k)+\bar{c}.
\end{align*}
The rest of the proof is analogous to Lemma~\ref{lemma:stable_bounded}. 
\end{pf}}{%
The proof is analogous to Lemma~\ref{lemma:stable_bounded}, see \cite[App.~F]{adaptive_arxiv} for details.}
This result also ensures \change{boundedness} of $x_k$. 
Thus, Inequality~\eqref{eq:local_Lipschitz_fw} also yields a \change{Lipschitz bound} for the closed-loop system and hence the results from Section~\ref{sec:global} can be adapted to this setting. 
The main difference compared to Section~\ref{sec:global} is that the theoretical results hold only for sufficiently small variations in the unknown model parameters $\Delta \theta_k$. 
\ifbool{arxiv}{\clearpage}{}
\ifbool{arxiv}{
\change{\section{Effect of design parameters on stability \& robustness}
\label{app:constants} 
In the following, we discuss how the design constants $N,M,\omega$ affect the stability and robustness.
\subsection{Open-loop stable systems (Sec.~\ref{sec:global})}
Nominal stability requires $\alpha>0$ in Theorem~\ref{thm:adaptive_global}.
From a computational point of view, it is desirable to achieve this with a small horizon $N,M$. The smallest horizon can be obtained by having a small prediction horizon $N$, the rollout horizon $M$ large enough such that $C_\ell\rho^M<1$ and the weight $\omega$ sufficiently large.
Quantitative robust stability properties in Theorem~\ref{thm:robust_stability_global} are captured by the contraction rate
$\rho_{\mathrm{V}}=1-\alpha\underline{c}_1/(2 \gamma_N)$ and the gain $c_{\mathrm{V}}=\dfrac{1+\epsilon}{\epsilon}c_{\mathrm{J}^\star}$, $\epsilon=(1-\rho_{\mathrm{V}})$.
The factors $\rho_{\mathrm{V}}$ and $\epsilon$
improve as $\alpha\rightarrow 1$, which is the case for either $N$ or $M$ large. In addition, a weight $\omega$ close to one gives the smallest bound $\gamma_N$.
Hence, nominal stability properties are best if a large horizon $N,M$ is chosen and the weight $\omega$ is close to one, as standard in MPC~\cite{grune2017nonlinear}.
The gain $c_{\mathrm{V}}$ depends additionally on the continuity constant $c_{\mathrm{J}^\star}$ from Lemma~\ref{lemma:continuity_opt_cost}, which depends on the horizon-dependent constants $c_{\mathrm{x},k}$ from Lemma~\ref{lemma:prediction_lipschitz} that are based on Lipschitz continuity. 
As discussed below Lemma~\ref{lemma:prediction_lipschitz}; 
owing to the fact that exponential stability on the compact domain implies incremental exponential stability, horizon independent uniform bounds $c_{\mathrm{x},k}$ also exist~\cite[Thm.~8]{karapetyan2025closed}.  
Thus, $c_{\mathrm{J}^\star}$ admits a bound independent of the horizons $N,M$, but that increases with $\omega$. 
This implies that we can obtain a bound on the term $c_{\mathrm{J}^\star}/(1-\rho_{\mathrm{V}})$ in Theorem~\ref{thm:robust_stability_global}, which decreases by increasing the horizon $N,M$ and choosing a weight $\omega$ close to one.   
For a given computational budget, increases in the rollout horizon $M$ have a stronger benefit on robustness and in general yield a smaller computational cost compared to increases in the prediction horizon $N$.
\subsection{Unstable systems (Sec.~\ref{sec:regional})}
Compared to the case of open-loop stable systems, the analysis for unstable systems requires a bound on the uncertainty~\eqref{eq:robust_feedback_bound_mismatch} and is valid on a given region of attraction characterized by $\bar{\mathcal{J}}$. 
Appendix~\ref{app:regional} shows that for a given $N,M,\omega$ satisfying $\alpha>0$, the horizon $N$ needs to be increased linearly with increases in $\bar{\mathcal{J}}$ to ensure that the constants $\alpha$ and $\rho_{\mathrm{V}}$ remain unchanged. 
The effect on the constant $c_{\mathrm{V}}$ is less straightforward. 
Given the Lipschitz bound (Assumption~\ref{ass:regularity}), the bound $c_{\mathrm{V}}$ is independent of $\bar{\mathcal{J}}$ but scales exponentially with the horizon $N$. 
Appendix~\ref{app:proof_global_robust} highlights that this can be circumvented using horizon-independent bounds owing to the incremental stability that can be established for open-loop exponentially stable systems on a compact set~\cite[Thm.~8]{karapetyan2025closed}. 
For the analysis in Theorem~\ref{thm:robust_stability_regional_feedback}, we can similarly utilise the local stability under the feedback $\kappa$ in the rollout horizon $M$ to obtain a constant $c_{\mathrm{V}}$ that is independent of $M$, but may still increase with $N$. 
Thus, we can improve the robustness properties by increasing the horizon $M$ and choosing a weight $\omega$ close to one. 
Increasing the region of operation through increasing $\bar{\mathcal{J}}$ in the analysis does not necessarily increase the robustness margin. 
Specifically, a larger $\bar{\mathcal{J}}$ requires a larger horizon $N$, which will give a larger gain $c_{\mathrm{V}}$. 
Note that Assumption~\ref{ass:stable_feedback} only guarantees local exponential stability on the potentially small set $\mathbb{S}_{\mathrm{loc}}$. Consequently, there is a fundamental limit to the magnitude of uncertainty that can be handled.} }{}
\end{document}